\documentclass[journal, 10pt]{IEEEtran}
\usepackage[utf8]{inputenc}
\date{}
\usepackage{amsmath}   
\usepackage{amsthm}
\usepackage{tabularx,amsmath,amssymb,graphicx,paralist,subfigure,pdfpages,cite,algorithm,algpseudocode,paralist,multirow,comment}
\usepackage[autostyle]{csquotes}
\newtheorem{lem}{Lemma} 
\newtheorem{theorem}{Theorem}

\newtheorem{prop}{Proposition}
\newtheorem{cor}{Corollary}
\newtheorem{rmk}{Remark}
\newtheorem{assump}{Assumption}
\usepackage{lipsum}
\usepackage{verbatim}
\usepackage{algorithm}
\usepackage{algpseudocode}
\usepackage{bbm}
\usepackage[colorlinks]{hyperref}

\hypersetup{
citecolor = blue
}

\def\mbb{\mathbb}%R
\def\mb{\mathbf}%vector
\def\mc{\mathcal}%set

\def\wh{\widehat}
\def\wt{\widetilde}
\def\ol{\overline}
\def\ul{\underline}

\def\bds{\boldsymbol}

\def\ra{\rightarrow}
\def\R{\mbb{R}}
\def\a{\alpha}
\def\E{\mbb{E}}
\def\F{\mc{F}}
\def\s{\sigma}
\def\l{\langle}
\def\r{\rangle}
\def\x{\mb{x}}
\def\y{\mb{y}}
\def\g{\mb{g}}
\def\P{\mbb{P}}
\def\G{\mb{G}}
\def\H{\mb{H}}
\def\W{\mb{W}}
\def\J{\mb{J}}
\def\I{\mb{I}}
\def\vp{\varphi}
\def\nf{\nabla\mb{f}}
\def\1{\mathbbm{1}}

\def\DSGT{\textbf{\texttt{GT-DSGD}}}
\def\SGD{\textbf{\texttt{SGD}}}
\def\DSGD{\textbf{\texttt{DSGD}}}

\def\l{\left\langle}
\def\r{\right\rangle}
\def\bxi{\bds\xi}

\def\agk{\ol{\g}_k}

\def\ln{\left\|}
\def\rn{\right\|}
\def\n{\nonumber}

\def\u{\mb{u}}

\usepackage[colorlinks]{hyperref}
\hypersetup{
citecolor = blue
}

\newenvironment{TH0}
  {\begin{proof}[Proof of Theorem~\ref{conv_ncvx}]}
  {\end{proof}}
  
\newenvironment{TH1}
  {\begin{proof}[Proof of Theorem~\ref{conv_PL}]}
  {\end{proof}}
  
\newenvironment{TH2}
  {\begin{proof}[Proof of Theorem~\ref{PL_as}]}
  {\end{proof}}
  
\newenvironment{TH3}
  {\begin{proof}[Proof of Theorem~\ref{F_ave_rate}]}
  {\end{proof}}
  
\newenvironment{C2}
  {\begin{proof}[Proof of Corollary~\ref{TRT}]}
  {\end{proof}}

\begin{document}
\title{An improved convergence analysis for decentralized online stochastic non-convex optimization}
\author{Ran Xin,  Usman A. Khan, and Soummya Kar
		\thanks{RX and SK are with the ECE Dept. at Carnegie Mellon University, \texttt{\{ranx,soummyak\}@andrew.cmu.edu}. UAK is with the ECE Dept. at Tufts University, \texttt{khan@ece.tufts.edu}. The work of SK and RX has been partially supported by NSF under award \#1513936. The work of UAK has been partially supported by NSF under awards \#1903972 and \#1935555. 
% 		Preliminary arxiv versions of this paper appeared in~\cite{GT_SAGA_arxiv,GT_SVRG_arxiv}.
		}
	}
\maketitle

\begin{abstract}
In this paper, we study decentralized online stochastic non-convex optimization over a network of nodes. Integrating a technique called gradient tracking in decentralized stochastic gradient descent, we show that the resulting algorithm, \textbf{\texttt{GT-DSGD}}, enjoys certain desirable characteristics towards minimizing a sum of smooth non-convex functions. In particular, for general smooth non-convex functions, we establish non-asymptotic characterizations of \textbf{\texttt{GT-DSGD}} and derive the conditions under which it achieves network-independent performances that match the centralized minibatch \textbf{\texttt{SGD}}. In contrast, the existing results suggest that \textbf{\texttt{GT-DSGD}} is always network-dependent and is therefore strictly worse than the centralized minibatch \textbf{\texttt{SGD}}.
When the global non-convex function additionally satisfies the Polyak-\L ojasiewics (PL) condition, we establish the linear convergence of \textbf{\texttt{GT-DSGD}} up to a steady-state error with appropriate constant step-sizes. Moreover, under stochastic approximation step-sizes, we establish, for the first time, the optimal global sublinear convergence rate on almost every sample path, in addition to the asymptotically optimal sublinear rate in expectation. Since strongly convex functions are a special case of the functions satisfying the PL condition, our results are not only immediately applicable but also improve the currently known best convergence rates and their dependence on problem parameters.

\end{abstract}

\section{Introduction}\label{intro}
This paper considers decentralized non-convex optimization where~$n$ nodes cooperate to solve the following problem:
\begin{align*}
\mbox {P1:} \qquad \min_{\mb x\in \mbb R^p} F(\mb x) :=\frac{1}{n}\sum_{i=1}^n f_i(\mb x),
\end{align*}
such that each function~${f_i:\mbb R^p\ra\mbb R}$ is local and private to node~$i$ and the nodes communicate over a balanced directed graph ${\mc G = \{\mc{V},\mc{E}\}}$, where~${\mc{V} = \{1,\cdots,n\}}$ is the set of node indices and~$\mc{E}$ is the collection of ordered pairs~${(i,j),i,j\in\mc{V}}$, such that node~$j$ sends information to node~$i$. Throughout the paper, we assume that each local~$f_i$ is smooth and non-convex. We focus on an \textit{online}\footnote{\color{black}We note that ``online" sometimes also refers to time-varying objective functions, which is different from the problem setup in this paper.}  setup where data samples are collected in real-time and hence each node~$i$ only has access to a noisy sample~$\mb g_i$ of the true gradient at each iteration, %instead of the exact (local) gradient $\nabla f_i$, 
such that~$\mb g_i$ is an unbiased estimate of $\nabla f_i$ with bounded variance. 
Problems of this nature have found significant interest in signal processing, machine learning, and control.
See e.g.,~\cite{OPT_ML,PIEEE_Xin}, for comprehensive surveys on these problems. 

Based on the classical stochastic gradient descent~(\SGD)~\cite{OPT_ML}, a well-known solution to Problem~P1 is decentralized \textbf{\texttt{SGD}} (\textbf{\texttt{DSGD}}) \cite{DSGD_nedich,diffusion_Chen}.
% However, a major limiting factor in the performance of~\textbf{\texttt{DSGD}}~is the dissimilarity among the local functions~$f_i$'s and the global function~$F$. 
However, the convergence of~\textbf{\texttt{DSGD}} for non-convex problems~has only been established under certain regularity assumptions such as uniformly bounded difference between local and global gradients~\cite{DSGD_NIPS,SGP_ICML,DSGD_vlaski_2}, or coercivity of each local function~\cite{DSGD_Swenson}. It has also been observed that if the data distributions across the nodes are heterogeneous, the practical performance of~\textbf{\texttt{DSGD}}~degrades significantly~\cite{MP_Pu,SED,PIEEE_Xin}. One notable line of work towards improving the performance of \textbf{\texttt{DSGD}} is EXTRA~\cite{EXTRA} and Exact Diffusion~\cite{Exact_Diffusion}, where the convergence under the stochastic non-convex setting is established without the aforementioned regularity assumptions~\cite{D2}; however, they require the weight matrix to be symmetric and the smallest eigenvalue is lower bounded by $-1/3$. Another family of algorithms to eliminate the performance limitation of \textbf{\texttt{DSGD}} is based on gradient tracking, introduced in~\cite{AugDGM,Next_Scutari}, where the basic idea is to replace the local gradients with a tracker of the global gradient $\nabla F$. 
Decentralized first-order methods with gradient tracking have been well studied under exact gradients, where relevant work can be found, e.g., in~\cite{harnessing,DIGing,MP_scutari,AB_Xin,PushPull_Pu}. However, the convergence behavior of gradient tracking methods has many unanswered questions when it comes to non-convex online stochastic problems~\cite{GNSD,DSGT_KY}. 

\textbf{Main contributions.} This paper considers \textbf{\texttt{GT-DSGD}}~\cite{MP_Pu}, that adds gradient tracking to \textbf{\texttt{DSGD}}, for online stochastic non-convex problems and rigorously develops novel results, key insights, and new analysis techniques that fill the theory gaps in the existing literature on gradient tracking methods~\cite{GNSD,DSGT_KY,MP_Pu}. The main contributions are described in the following: 

\noindent \textit{(1) General smooth non-convex problems}: 
We explicitly characterize the non-asymptotic, transient and steady-state performance of \textbf{\texttt{GT-DSGD}} and derive the conditions under which they are comparable to that of the centralized minibatch \textbf{\texttt{SGD}}. In particular, we show that its non-asymptotic mean-squared rate is network-independent and further matches the centralized minibatch \textbf{\texttt{SGD}} when the number of iterations is large enough. In sharp contrast, the existing results in~\cite{GNSD,DSGT_KY} suggest that the convergence rate and steady-state performance of \textbf{\texttt{GT-DSGD}} are always network-dependent and therefore
are strictly worse than that of the centralized minibatch \textbf{\texttt{SGD}}; see Section~\ref{sec_main_ncvx} for details. 

\noindent \textit{(2) Problems satisfying the global Polyak-\L ojasiewicz (PL) condition}: We analyze \textbf{\texttt{GT-DSGD}} when the global (smooth non-convex) function~$F$ further~satisfies the PL condition. For both constant and decaying step-sizes, we explicitly characterize the non-asymptotic, transient and steady-state behaviors in expectation, and establish the conditions under which they are comparable to that of the centralized minibatch \textbf{\texttt{SGD}}. We further establish global sublinear convergence rates on almost every sample path. The obtained sample path-wise rates are order-optimal (in the sense of polynomial time decay). To the best of our knowledge, these are the first results on path-wise convergence rate for online decentralized stochastic optimization under non-convexity, thus generalizing prior results in the decentralized stochastic approximation literature, e.g.,~\cite{GLE_kar}, where the convergence analysis is mostly performed under assumptions of local convexity. As special cases, these results improve the current state-of-the-art on exact gradient methods under the PL condition~\cite{ZO_GT} and stochastic strongly convex problems~\cite{MP_Pu}; see Section~\ref{sec_main_PL} for~details. 

\noindent\textit{(3) Convergence analysis:} {\color{black}We emphasize that the analysis techniques in this work are substantially different from the existing ones~\cite{MP_Pu},~\cite{GNSD},~\cite{DSGT_KY} and may be applied to other gradient methods built upon similar principles. We describe a few key features in the following. 
We establish tighter bounds on the stochastic gradient tracking process,  by exploiting the unbiasedness of the online stochastic gradients, based on which all convergence theorems are derived; see Section~\ref{sec_GT_general_bounds}.
To prove the convergence under general non-convexity, we characterize a descent inequality explicitly with network consensus errors and further show that the cumulative consensus errors along the algorithm path are dominated by the cumulative descent effect of the local gradients; see Section~\ref{sec_ncvx_proof}.  
Towards the convergence analysis under the global PL condition, we derive the uniform boundedness of gradient tracking errors that is crucial in simplifying the ensuing analysis; see Lemma~\ref{Y_bounded}. Subsequently, we construct an appropriate stochastic process that forms an almost supermartingale~\cite{RS_theorem} to prove sublinear rates on almost every sample path; see Section~\ref{S_PL_as}. To develop the convergence results in mean under the global PL condition, we use the analytical tools developed for recursive processes with time-varying step-sizes; see Section~\ref{S_PL_ms_decay}.}

\textbf{Road map and notation. }
{\color{black}The rest of the paper is organized as follows. Section~\ref{S_aa} describes the assumptions and the \textbf{\texttt{GT-DSGD}} algorithm. In Section~\ref{S_mr}, we present the main results and discuss the contributions of this work in the context of the current state-of-the-art, whereas Section~\ref{sec_main_ncvx} and~\ref{sec_main_PL} respectively focus on the general non-convex and the PL case. We present detailed numerical experiments in Section~\ref{s_exp} to demonstrate the main theoretical results in this paper. Section~\ref{S_ncvx} establishes general bounds on the stochastic gradient tracking process and proves the convergence for smooth non-convex functions. Sections~\ref{S_PL_ms_cst},~\ref{S_PL_as} and~\ref{S_PL_ms_decay} provide the convergence analysis under the PL condition on top of the results obtained in Section~\ref{S_ncvx}. In particular, Sections~\ref{S_PL_ms_cst} and~\ref{S_PL_ms_decay} focus on the convergence in mean with constant and decaying step-sizes respectively while Section~\ref{S_PL_as} focuses on the almost sure convergence. Section~\ref{S_conc} concludes the paper.
% The rest of the paper is structured as follows. Section~\ref{S_aa} describes the assumptions and the \textbf{\texttt{GT-DSGD}} algorithm. In Section~\ref{S_mr}, we present the main convergence results and discuss the contributions of this work in the context of the current state-of-the-art, where Section~\ref{sec_main_ncvx} and~\ref{sec_main_PL} respectively focus on the general non-convex and the PL case. We present detailed numerical experiments in Section~\ref{s_exp} to demonstrate the main theoretical results in this paper. Section~\ref{S_ncvx} establishes general bounds on the stochastic gradient tracking process and proves the convergence for smooth non-convex functions. Sections~\ref{S_PL_ms_cst},~\ref{S_PL_as} and~\ref{S_PL_ms_decay} provide the convergence analysis under the PL condition on top of the results obtained in Section~\ref{S_ncvx}. In particular, Sections~\ref{S_PL_ms_cst} and~\ref{S_PL_ms_decay} focus on the convergence in mean with constant and decaying step-sizes respectively while Section~\ref{S_PL_as} focuses on almost sure convergence. Section \ref{S_conc} concludes the paper. 

We use lowercase bold letters to denote vectors and uppercase bold letters for matrices. The matrix,~$\mb{I}_d$ (resp.~$\mb{O}_d$), represents the~$d\times d$ identity (resp. zero matrix); $\mb{1}_d$ and~$\mb{0}_d$ are the~$d$-dimensional column vectors of all ones and zeros, respectively. We denote~$[\x]_i$ as the~$i$-th entry of a vector~$\x$.
The Kronecker product of two matrices~$\mb{A}$ and~$\mb{B}$ is denoted by~$\mb{A}\otimes \mb{B}$. We use~$\left\|\cdot\right\|$ to denote the Euclidean norm of a vector or the spectral norm of a matrix. For a matrix~$\mb{X}$, we use~$\rho(\mb{X})$ to denote its spectral radius,~$\mb{X}^*$ to denote its adjugate,~$\det(\mb X)$ to denote its determinant,~$[\mb{X}]_{i,j}$ to denote its~$(i,j)$th element and~$\mbox{diag}(\mb{X})$ as the diagonal matrix that consists of the diagonal entries of~$\mb{X}$. Matrix-vector inequalities are interpreted in the entry-wise sense.  We use~$\sigma(\cdot)$ to denote the~$\sigma$-algebra generated by the random variables and/or sets in its argument.} 

\section{Assumptions and the~\DSGT~Algorithm}\label{S_aa}
We are interested in finding a first-order stationary point of Problem P1 via local computation and communication at each node. We first enlist the necessary assumptions that are standard in the literature~\cite{OPT_ML,MP_Pu,SED,book_polyak}.

% \vspace{-0.2cm}
% \subsection{Assumptions}
\begin{assump}[\textbf{Objective functions}]\label{f}
%\normalfont
Each~$f_i$ is~$L$-smooth, i.e.,~$\exists L>0$ s.t.
$\|\nabla f_i(\mb{x}) - \nabla f_i(\mb{y})\| \leq L\|\x - \mb{y}\|,\forall \x,\mb{y}\in\R^p.$
Moreover,~$F$ is bounded below, i.e.,~$F^*:=\inf_{\x}F(\x)>-\infty$.
\end{assump}

\begin{assump}[\textbf{Network model}]\label{net}
%\normalfont
The directed communication network is strongly-connected and admits a primitive doubly-stochastic weight matrix~$\ul{\mb{W}} = \{\ul{w}_{ir}\}\in\R^{n\times n}$.
\end{assump}

We consider iterative processes that generate at each~node~$i$ a sequence of state vectors~${\{\x^i_k: k\geq0\}}$, where~$\mb{x}^i_0$ is assumed to be a constant. At each iteration~$k$, each node~$i$ is able to call the local oracle that returns a stochastic gradient~$\g_i(\x^i_k,\bxi^i_k)$, where~$\bxi^i_k$ is a random vector in $\R^q$ and ${\g_i:\R^p\times\R^q\ra\R^p}$ is a Borel-measurable function. 
For example, $\g_i(\x^i_k,\bxi^i_k)$ may be considered as the stochastic gradient evaluated at the state~$\x_k^i$ with the data sample~$\bxi_k^i$ observed at node~$i$ and iteration~$k$. 
We work with a rich enough probability space $(\Omega,\F,\P)$ and define the natural filtration (an increasing family of~sub-$\s$-algebras of~$\F$) as, $\forall k\geq1$,
\begin{align*}
\F_k := \s\left(\left\{\bxi^i_t: 0\leq t\leq k-1, i\in\mc{V} \right\}\right), \qquad\F_0 :=&~\{\Omega,\phi\},
\end{align*}
where~$\phi$ is the empty set. 
The intuitive meaning of~$\F_k$ is that it contains the historical information of the algorithm iterates in question up to iteration~${k-1}$. 

\begin{assump}[\textbf{Oracle model}]\label{o}
%\normalfont
The stochastic gradient process~$\{\g_i(\x_k^i,\bxi^i_k):\forall k\geq0,\forall i\in\mc{V}\}$ satisfies:
\begin{itemize}
	\item $\E\left[\g_i(\x_k^i,\bxi^i_k)|\F_k\right] = \nabla f_i(\x^i_k), \forall k\geq0, \forall i\in\mc{V}$;
	\item $\E\big[\ln\g_i(\x^i_k,\bxi^i_k) - \nabla f_i(\x^i_k)\rn^2 | \F_k\big]\leq \nu_i^2, \forall k\geq0, \forall i\in\mc{V}$, for some constant~$\nu_i>0$;
	\item The family~$\left\{\bxi_k^i: \forall k\geq0, \forall i\in\mc{V} \right\}$ of random vectors is independent.
\end{itemize}
\end{assump}

We denote ${\nu_a^2 := \frac{1}{n}\sum_{i=1}^n\nu_i^2}$, the average of the variance of local stochastic gradients.
We are also interested in the case when \emph{the global objective function $F$} further satisfies the Polyak-Łojasiewicz (PL) condition that was introduced in~\cite{book_polyak}.
\begin{assump}\label{PL}
%\normalfont
$\exists\mu > 0$ s.t. the global function~$F:\R^p\ra\R$ satisfies~$2\mu\left(F(\x) - F^*\right) \leq \ln\nabla F(\x) \rn^2, \forall \x\in\mbb{R}^p$.  
\end{assump}
\noindent When Assumption~\ref{PL} holds, we denote~$\kappa: = \frac{L}{\mu}\geq1$, which can be interpreted %acts 
as the condition number of~$F$; see Lemma~\ref{upperL}.
Note that under the PL condition, every stationary point~$\x^*$ of~$F$ is a global minimum of~$F$, while~$F$ is not necessarily convex. Assumption~\ref{PL} holds, e.g., in certain reinforcement learning problems~\cite{RL_PL}, see~\cite{book_polyak,PL_1} for more details.

\vspace{0.2cm}
\noindent \textbf{Algorithm.}
\DSGT, introduced in~\cite{MP_Pu} for smooth strongly convex problems and formally described in Algorithm~\ref{alabel}, recursively descends in the direction of an auxiliary variable~$\mb y_k^i$ at each node, instead of the local stochastic gradient~$\mb g_i(\mb x_k^i,\bxi_k^i)$. The auxiliary variable~$\mb y_k^i$ is constructed under the dynamic average consensus principle~\cite{DAC} and tracks a time-varying signal~${\sum_i \mb g_i(\mb x_k^i,\bxi_k^i)}$, which mimics the global gradient; see~\cite{MP_Pu,PIEEE_Xin} for further intuition and explanation.  We note that \textbf{\texttt{GT-DSGD}} uses the adapt-then-combine (ATC) structure~\cite{diffusion_Chen} resulting in improved stability of the algorithm.

\begin{algorithm}[tbph]
\caption{\textbf{\texttt{GT-DSGD}} at each node~$i$}
%\algsetup{indent=2em}
\label{alabel}
\begin{algorithmic}[1]
\Require{$\x^i_{0}$;~$\{\alpha_k\}$;~$\{\ul{w}_{ir}\}$;
$\mb{y}_i^{0} = \mb{0}_p$; $\g_r\big(\x_{-1}^r,\bxi_{-1}^r\big) := \mb{0}_p$.}
\For{$k = 0, 1, \ldots,$}
\vspace{-0.1cm}
\begin{align*}
\mb{y}^{i}_{k+1} &= \sum_{r=1}^{n}\ul{w}_{ir}\big(\mb{y}^r_{k} + \g_r(\x^r_k,\bxi^r_k) - \g_r(\x^r_{k-1},\bxi^r_{k-1})\big)\\\mb{x}^{i}_{k+1} &= \sum_{r=1}^{n}\ul{w}_{ir}\big(\mb{x}^{r}_{k} - \a_k\mb{y}^{r}_{k+1}\big)
\end{align*}
\vspace{-0.3cm}
\EndFor
\end{algorithmic}
\end{algorithm}

\section{Main results}\label{S_mr}
In this section, we present our main convergence results for \textbf{\texttt{GT-DSGD}} and compare them with the corresponding state-of-the-art. For analysis purposes and the ease of presentation of main results, we let~$\x_k,\y_k,\g_{k}$, all in~$\R^{np}$, respectively concatenate~$\x_k^i$'s,~$\y_k^i$'s,~$\g_i(\x_k^i,\bxi_k^i)$'s, and write \textbf{\texttt{GT-DSGD}} in the following matrix form:~$\forall k\geq0$,
\begin{subequations}
\begin{align}
\mb{y}_{k+1} =&~\mb{W}\left(\mb{y}_{k} + \g_{k} - \g_{k-1}\right), \label{y}\\
\mb{x}_{k+1} =&~\mb{W}\left(\mb{x}_{k} - \a_k\mb{y}_{k+1}\right), \label{x}
\end{align}
\end{subequations}
where~${\mb{W} = \ul{\mb{W}}\otimes \I_p}$. We denote the exact averaging matrix as~${\mb{J} := (\frac{1}{n}\mb{1}_n\mb{1}_n^\top)\otimes \I_p}$ and ${\lambda := \|\mb{W} - \mb{J}\|}$, which characterizes the network connectivity. Under Assumption~\ref{net}, we have~${\lambda\in[0,1)}$; see \cite{book_matrix_analysis}.
For convenience, we let $\nf_k\in\R^{np}$ concatenate all local exact gradients~$\nabla f_i(\x_k^i)$'s and denote
\begin{align*}
\ol{\mb x}_k :=&~\frac{1}{n}(\mb{1}_n^\top\otimes \I_p)\mb{x}_k, \quad\ol{\mb y}_k :=\frac{1}{n}(\mb{1}_n^\top\otimes \I_p)\mb{y}_k, \\
\ol{\nf}_{k} :=&~\frac{1}{n}(\mb{1}_n^\top\otimes \I_p)\nabla\mb{f}_k,\quad
\ol{\g}_k := \frac{1}{n}(\mb{1}_n^\top\otimes\I_p)\g_{k}.
\end{align*}
We assume without loss of generality that~$\mb{x}_0^i = \mb{x}_0^r,\forall i,r\in\mc{V}$.

\subsection{General smooth non-convex functions}\label{sec_main_ncvx}
In this subsection, we are concerned with the convergence of~\DSGT~for general smooth non-convex functions.

\begin{theorem}\label{conv_ncvx}
%\normalfont
Let Assumptions~\ref{f},~\ref{net}, and~\ref{o} hold and consider \textbf{\texttt{GT-DSGD}} under a constant step-size~${\a_k = \a,\forall k\geq0}$, such that~${0<\a\leq\min\big\{1,\frac{1-\lambda^2}{12\lambda},\tfrac{(1-\lambda^2)^2}{4\sqrt{6}\lambda^2}\big\}\frac{1}{2L}}$, then,~$\forall K>1$,
\begin{align*}
&\underbrace{\frac{1}{n}\sum_{i=1}^n\frac{1}{K}\sum_{k=0}^{K-1}\E\left[\big\|\nabla F(\x_k^i)\big\|^2\right]}_{\text{Mean-squared stationary gap}}
\leq\underbrace{\frac{4(F(\ol{\x}_{0}) - F^*)}{\a K}  
+\frac{2\a\nu^2_a L}{n}}_{\text{Centralized minibatch SGD}} \n\\
&\qquad\qquad\qquad\quad+\underbrace{\frac{448\a^2L^2\lambda^2\nu_a^2}{(1-\lambda^2)^3} + \frac{64\a^2L^2\lambda^4}{(1-\lambda^2)^3 K}\frac{\left\|\nf_0\right\|^2}{n}}_{\text{Decentralized network effect}}.
\end{align*}
Further,~$\frac{1}{n}\sum_{i=1}^n\frac{1}{K}\sum_{k=0}^{K-1}\E\big[\|\nabla F(\x_k^i)\|^2\big]$ decays at the rate of~$\mc{O}(\frac{1}{K})$ up to a steady-state error such that
\begin{align*}
\limsup_{K\ra\infty}\frac{1}{n}\sum_{i=1}^n\frac{1}{K}&\sum_{k=0}^{K-1}\E\left[\ln\nabla F(\x_k^i)\rn^2\right] \n\\
\leq&~\underbrace{\frac{2\a\nu^2_a L}{n}}_{\text{Centralized minibatch SGD}}  +\underbrace{\frac{448\a^2L^2\lambda^2\nu_a^2}{(1-\lambda^2)^3}}_{\text{Decentralized network effect}}.
\end{align*}
\end{theorem}

Theorem~\ref{conv_ncvx} is proved in Section~\ref{S_ncvx}. 
%Here we highlight some key features in the following remarks. 
\vspace{-0.1cm}
\begin{rmk}[\textbf{Transient and steady-state performance}]
\normalfont
Theorem~\ref{conv_ncvx} explicitly characterizes the non-asymptotic performance of \textbf{\texttt{GT-DSGD}} for general smooth non-convex functions with an appropriate constant step-size. In particular, the stationary gap of \textbf{\texttt{GT-DSGD}} for any finite number of iterations~$K$ is bounded by the sum of four terms. The first two terms are independent of the network spectral gap~$1-\lambda$ and match the complexity of the centralized minibatch \textbf{\texttt{SGD}} up to constant factors~\cite{OPT_ML}. The third and the fourth terms depend on~$1-\lambda$ reflecting the decentralized network and are in the order of~$\mc{O}(\a^2)$. This is a much tighter characterization compared with the existing results~\cite{GNSD,DSGT_KY} on \textbf{\texttt{GT-DSGD}} and leads to provably faster non-asymptotic rate, see Remark~\ref{TRT_ncvx_r} below. Theorem~\ref{conv_ncvx} also shows that as ${K\ra\infty}$, the stationary gap of~\DSGT~decays sublinearly at the rate of~$\mc{O}(1/K)$ up to a steady-state error. It can be observed that if~${\a = \mc{O}\big(\frac{(1-\lambda)^3}{\lambda^2nL}\big)}$, then the steady state stationary gap of \textbf{\texttt{GT-DSGD}} matches that of the centralized minibatch \textbf{\texttt{SGD}} up to constant factors. The existing analysis~\cite{DSGT_KY}, however, suggests that under the same choice of the step-size~$\a$, the steady state stationary gap of \textbf{\texttt{GT-DSGD}} is strictly worse than the centralized minibatch~\textbf{\texttt{SGD}}.    
\end{rmk}

The following corollary of Theorem~\ref{conv_ncvx} is concerned with the non-asymptotic convergence rate of~\DSGT~over a finite time horizon for general smooth non-convex functions. 
\begin{cor}\label{TRT_ncvx}
%\normalfont
Let Assumptions~\ref{f},~\ref{net}, and~\ref{o} hold and suppose that~${\ln\nf_0\rn^2 = \mc{O}(n)}$. Setting~${\a = \sqrt{n/K}}$ in Theorem~\ref{conv_ncvx}, for ${K\geq 4nL^2\max\left\{1,\frac{144\lambda^{2}}{(1-\lambda^2)^2},\frac{96\lambda^4}{(1-\lambda^2)^4}\right\}}$, we obtain:
\begin{align*}
&\frac{1}{n}\sum_{i=1}^n\frac{1}{K}\sum_{k=0}^{K-1}\E\left[\ln\nabla F(\x_k^i)\rn^2\right] 
\leq\underbrace{\frac{4(F(\ol{\x}_{0}) - F^*)}{\sqrt{nK}}  +\frac{2\nu^2_a L}{\sqrt{nK}}}_{\text{Centralized minibatch SGD}} \\
&\qquad\qquad\qquad\qquad+\underbrace{\frac{448n\lambda^2\nu_a^2L^2}{(1-\lambda^2)^3K} 
+\frac{64L^2\lambda^4\left\|\nf_0\right\|^2}{(1-\lambda^2)^3 K^2}}_{\text{Decentralized network effect}}.
\end{align*}
Thus, if~$K$ further satisfies that~${K\geq K_{nc}:=\mc{O}\left(\frac{n^3\lambda^4L^2}{(1-\lambda)^6}\right)}$, then we have
\begin{align*}
\frac{1}{n}\sum_{i=1}^n\frac{1}{K}\sum_{k=0}^{K-1}\E\left[\ln\nabla F(\x_k^i)\rn^2\right] 
=\mc{O}\left(\frac{\nu^2_a L}{\sqrt{nK}}\right).
\end{align*}
\end{cor}

\begin{rmk}[\textbf{Non-asymptotic mean-squared rate and transient time for network independence}]\label{TRT_ncvx_r}
\normalfont
Corollary~\ref{TRT_ncvx} shows that if the number of iterations is large enough, i.e.,~${K\geq K_{nc}}$, by setting ${\a = \frac{\sqrt{n}}{\sqrt{K}}}$, the non-asymptotic rate of \textbf{\texttt{GT-DSGD}} matches that of the centralized minibatch \textbf{\texttt{SGD}} up to factors of universal constants. This discussion shows that, in the regime that~$K\geq K_{nc}$, \textbf{\texttt{GT-DSGD}} achieves a network-independent linear speedup compared with the centralized minibatch \textbf{\texttt{SGD}} that processes all data at a single node. In other words, the number of stochastic gradient computations required to achieve an approximate stationary point is reduced by a factor of~$1/n$ at each node in the network. These results significantly improve the existing convergence guarantees of \textbf{\texttt{GT-DSGD}} for general smooth non-convex {functions~\cite{GNSD,DSGT_KY}}. In particular, references~\cite{GNSD,DSGT_KY} show that if~${\a = \frac{c_0}{\sqrt{K}}}$, where~$K$ is large enough and~$c_0$ is some positive constant, \textbf{\texttt{GT-DSGD}} achieves the convergence rate of~$\frac{c_1}{\sqrt{K}}$, where~$c_1$ is a function of the network spectral gap~$(1-\lambda)$. The convergence results in~\cite{GNSD,DSGT_KY} thus suggest that the rate of \textbf{\texttt{GT-DSGD}} is always network-dependent and is strictly worse than that of the centralized minibatch \textbf{\texttt{SGD}} and hence fail to characterize the network-independent performance of~\DSGT.
\end{rmk}

\vspace{-0.3cm}
{\color{black}\begin{rmk}[\textbf{Comparison with \textbf{\texttt{DSGD}}}]\label{GT_versus_DSGD}\normalfont
We observe from Corollary~\ref{TRT_ncvx} that the convergence of \textbf{\texttt{GT-DSGD}} is robust to the difference between the local and the global functions. In other words, \textbf{\texttt{GT-DSGD}} outperforms \textbf{\texttt{DSGD}} when data distributions across the nodes are significantly heterogeneous, since the convergence rate of the latter explicitly depends on a factor that measures the heterogeneity between the local and the global functions~\cite{DSGD_NIPS}.   
However, the transient time for \textbf{\texttt{GT-DSGD}} to achieve network independent performance has a network dependence of~$\mc{O}((1-\lambda)^{-6})$ which is worse than that of \textbf{\texttt{DSGD}} where the dependence is~$\mc{O}((1-\lambda)^{-4})$. Moreover, we note that \textbf{\texttt{GT-DSGD}} requires two consecutive rounds of communication per node per iteration to update the state and the gradient tracker variables respectively, compared to~$\DSGD$.

% it may be more challenging to practically deploy \textbf{\texttt{GT-DSGD}} compared with \textbf{\texttt{DSGD}} since it incurs two consecutive rounds of communication per iteration to transmit the state and the gradient tracker variables respectively. 

%In a nutshell, \textbf{\texttt{GT-DSGD}}, due to its robustness, may be more appropriate in heterogeneous settings, e.g., in federated learning environments~\cite{FDL_spm}; on the other hand, if the data distributions are homogeneous, e.g., in data centers, DSGD may become a more practical choice. 
\end{rmk}}

\vspace{-0.4cm}
\subsection{Smooth non-convex functions under PL condition}\label{sec_main_PL}
In this subsection, we discuss the performance of \textbf{\texttt{GT-DSGD}} when the global objective function~$F$ further satisfies the PL condition. We begin with the case of constant step-size.
\begin{theorem}\label{conv_PL}
%\normalfont
Let Assumption~\ref{f},~\ref{net},~\ref{o} and~\ref{PL} hold. If the step-size~$\a_k = \a,\forall k\geq0$, satisfies that
\begin{align*}
0< \a\leq\ol{\a}:=\min\left\{\frac{1}{2L},\frac{(1-\lambda^2)^2}{42\lambda^2 L}, \frac{1-\lambda^2}{24\lambda L\kappa^{1/4}},\frac{1-\lambda^2}{2\mu}\right\},   
\end{align*}
{\color{black}then~$\E[\ln\x_k-\J\x_k\rn^2]$ and~$\E[F(\ol{\x}_k) - F^*]$ decay linearly at the rate of~$\mc{O}((1-\mu\a)^k)$ up to a steady-state error such that}
\begin{align*}
&\limsup_{k\ra\infty}\E\bigg[\frac{\ln\x_k-\J\x_k\rn^2}{n}\bigg] \leq\frac{288\lambda^4\a^5L^3\kappa\nu_a^2}{n(1-\lambda^2)^4} + \frac{144\lambda^2\a^2\nu_a^2}{(1-\lambda^2)^3}, \n\\
&\limsup_{k\ra\infty}\E\left[F(\ol{\x}_k) - F^*\right] \leq\frac{3\a\kappa\nu_a^2}{2n} + \frac{72\lambda^2\a^2\kappa L\nu_a^2}{(1-\lambda^2)^3}.
\end{align*}
{\color{black}Moreover,~$\frac{1}{n}\sum_{i=1}^n\E\left[F({\x}_k^i) - F^*\right]$ decays linearly at the rate of~$\mc{O}((1-\mu\a)^k)$ up to a steady-state error such that}
\begin{align*}
&\limsup_{k\ra\infty}\frac{1}{n}\sum_{i=1}^n\E\left[F({\x}_k^i) - F^*\right] \\
&\qquad\qquad\qquad= \underbrace{\mc{O}\left(\frac{\a\kappa\nu_a^2}{n}\right)}_{\text{Centralized minibatch SGD}} + 
\underbrace{\mc{O}\left(\frac{\lambda^2\a^2\kappa L\nu_a^2}{(1-\lambda)^3}\right)}_{\text{Decentralized network effect}}.
\end{align*}
\end{theorem}
Theorem~\ref{conv_PL} is proved in Section~\ref{S_PL_ms_cst}. 
%Here we highlight some key features in the following remarks. 

\vspace{-0.15cm}
\begin{rmk}[\textbf{Transient and steady-state performance}]
\normalfont
Theorem~\ref{conv_PL} shows that when the global objective function~$F$ satisfies the PL condition and the constant step-size~$\a$ is less than~$\ol{\a}$,
the optimality gap of \textbf{\texttt{GT-DSGD}} decays linearly up to a steady-state error that is the sum of two terms. The first term is independent of the network and matches that of the centralized minibatch \textbf{\texttt{SGD}} up to constant factors, while the second term is due to the network and is controlled by~$\mc{O}(\a^2)$. In contrast to~\cite{MP_Pu}, which requires a stronger assumption that the global objective function is strongly convex, we note that our stability range of the step-size $\a$ is larger by a factor of~$\mc{O}(\kappa^{5/12})$; this relaxed upper bound on~$\a$ further leads to a faster linear convergence when exact gradients are available, see Remark~\ref{fullgradient}. Next, it can be verified from Theorem~\ref{conv_PL} that to match the steady-state error performance of the centralized minibatch \textbf{\texttt{SGD}} (up to constant factors), it suffices to choose the step-size~$\a$ in~\textbf{\texttt{GT-DSGD}} such that~${\a=\mc{O}\big(\frac{(1-\lambda)^3}{\lambda^2nL}\big)}$, which is larger by a factor of~$\mc{O}(\kappa)$ than the corresponding result in~\cite{MP_Pu}; in other words, Theorem~\ref{conv_PL} demonstrates a tighter and faster convergence rate to achieve the same steady-state error.
\end{rmk}

\vspace{-0.35cm}
\begin{rmk}[\textbf{Global linear convergence under exact gradient oracle}]\label{fullgradient}
\normalfont
Theorem~\ref{conv_PL} further shows that when the exact gradient oracle is available at each node, i.e.,~${\nu_i^2=0, \forall i\in\mc{V}}$, \textbf{\texttt{GT-DSGD}} reduces to its deterministic counterpart~\cite{harnessing,DIGing,AugDGM} and achieves global linear convergence to an optimal solution with an appropriate constant step-size. In other words, when~$\a = \ol{\a}$, it achieves an~$q$-accurate optimal solution in $\mc{O}\big(\max\big\{\kappa,\frac{\lambda^2\kappa}{(1-\lambda)^2},\frac{\lambda\kappa^{5/4}}{1-\lambda},\frac{1}{1-\lambda}\big\}\log\frac{1}{q}\big)$ iterations. This result improves upon the state-of-the-art gradient computation and communication complexity under the PL condition~\cite{ZO_GT}.
%by a factor of~$\mc{O}\big(\min\big\{\kappa^{1/3},\frac{\kappa^{1/12}}{1-\lambda}\big\}\big)$. 
The gradient computation complexity can be further improved to $\mc{O}\big(\kappa\log\frac{1}{\epsilon}\big)$ by performing $\mc{O}\big(\frac{1}{1-\lambda}\log\frac{\kappa}{1-\lambda}\big)$ rounds of consensus communication at each iteration. This gradient computation complexity result matches the state-of-the-art~\cite{NIDS} on decentralized exact gradient methods (without Nesterov acceleration), which further requires a stronger assumption that each local function is convex and the global function is strongly convex. In contrast, we only require the PL condition on the global objective~$F$.
\end{rmk}
\vspace{-0.1cm}
We now proceed to the case of decaying step-sizes. The next result shows the sample path-wise performance of \textbf{\texttt{GT-DSGD}} under a family of stochastic approximation step-sizes~\cite{SAbook_AMS}, i.e.,~$\a_k>0,$~${\sum_{k=0}^{\infty}\a_k = \infty}$, and~${\sum_{k=0}^{\infty}\a_k^2 < \infty}$, which enables the exact sublinear convergence in contrast to the inexact linear convergence under a constant step-size.
%\vspace{-0.5cm}
\begin{theorem}\label{PL_as}
%\normalfont
Let Assumptions~\ref{f},~\ref{net},~\ref{o}, and~\ref{PL} hold. Consider the step-size sequence~$\{\a_k\}$ such that~$\a_k = \delta(k+\vp)^{-\epsilon},\forall k\geq0$, where~${\epsilon\in(0.5,1]}$,~${\delta\geq1/\mu}$, and $\vp\geq \max\big\{(\delta/\ol{\a})^{1/\epsilon},\frac{4}{1-\lambda^2}\big\}$ for~$\ol{\a}$ given in Theorem~\ref{conv_PL}. Then~$\forall i,j\in\mc{V}$ and for arbitrarily small $\epsilon_1>0$, we have:
\begin{align*}
&\P\left(\sum_{k=0}^{\infty}k^{2\epsilon-1-\epsilon_1}\big\|\x_k^i - \x_k^j\big\|^2 < \infty \right) = 1, \n\\
&\P\Big(\lim_{k\ra\infty}k^{2\epsilon-1-\epsilon_1}\big(F(\x_k^i) - F^*\big) = 0 \Big) = 1.   
\end{align*}
\end{theorem}
Theorem~\ref{PL_as} is proved in Section~\ref{S_PL_as}.

\begin{rmk}[\textbf{Global sublinear rate on almost every sample path}]
\normalfont
Theorem~\ref{PL_as} guarantees that \textbf{\texttt{GT-DSGD}} exhibits a global sublinear convergence on almost every sample path, under decaying step-sizes, when the global function~$F$ satisfies the PL condition. This result is of significant practical value in that it is applicable to every instantiation of the algorithm while the expectation type convergence only characterizes, roughly speaking, the performance on average. Furthermore, in the case of general non-degenerate variances (see Assumption~\ref{o}), these path-wise rates are order-optimal, in the sense of polynomial time decay; this follows by considering the stochastic approximation reformulation of the optimization problem (i.e., the problem of obtaining zeros of the gradient function~$\nabla F(\x)$) and invoking standard central limit type arguments, see~\cite{SAbook_AMS}.)
To the best of our knowledge, Theorem~\ref{PL_as} is the first to show path-wise convergence for online decentralized stochastic optimization under non-convexity, thus generalizing prior results in the decentralized stochastic approximation and optimization literature, such as~\cite{GLE_kar}, where such analysis is performed under assumptions of local convexity. 
%Our proof is built on an almost supermartingale convergence theorem~\cite{RS_theorem}, which may be of independent interest to derive path-wise rates for decentralized stochastic methods.
\end{rmk}

Finally, we consider the convergence rate of \textbf{\texttt{GT-DSGD}} in expectation when~$\a_k = \mc{O}(1/k),\forall k\geq0$. 
\begin{theorem}\label{F_ave_rate}
%\normalfont
Let Assumptions~\ref{f},~\ref{net},~\ref{o}, and~\ref{PL} hold. Consider the step-size sequence~$\{\a_k\}$ such that~$\a_k = \beta(k+\gamma)^{-1}$,${\forall k\geq0}$, where~${\beta>2/\mu}$, and~${\gamma \geq \max\big\{\frac{\beta}{\ol{\a}},\frac{8}{1-\lambda^2}\big\}}$ for~$\ol{\a}$ given in Theorem~\ref{conv_PL}.
We have:~$\forall k\geq0$,
\begin{align*}
\frac{1}{n}\sum_{i=1}^n\E\big[F(\x_k^i) -& F^*\big] 
\leq\underbrace{\frac{2L\nu_a^2\beta^2}{n(\mu\beta-1)(k+\gamma)}}_{\text{Centralized minibatch SGD}} 
\\
&+ \underbrace{\frac{2\left(F(\ol{\x}_0) - F^*\right)}{(k/\gamma+1)^{\mu\beta}} +\frac{3L^2\wh{x}\beta^3}{n(\mu\beta-2)(k+\gamma)^{2}}}_{\text{Decentralized network effect}}, 
\end{align*}
where~$\wh{x}$ is a positive constant given in~\eqref{induction_cons}.
\end{theorem}

The non-asymptotic rate in Theorem~\ref{F_ave_rate} shows that \textbf{\texttt{GT-DSGD}} asymptotically achieves network independent~$\mc{O}(1/k)$ rate in mean when the global objective function~$F$ satisfies the PL condition, matching the~$\Omega(1/k)$ oracle lower bound~\cite{OPT_ML}. The following corollary examines the number of transient iterations required to achieve network-independence under specific choices of parameter~$\beta$ and~$\gamma$ in Theorem~\ref{F_ave_rate}.

\begin{cor}\label{TRT}
Let Assumptions~\ref{f},~\ref{net},~\ref{o}, and~\ref{PL} hold. Set~${\beta = 6/\mu}$ and ${\gamma = \max\big\{\frac{6}{\mu\ol{\a}},\frac{8}{1-\lambda^2}\big\}}$ in Theorem~\ref{F_ave_rate} and suppose that $\ln\nf_0\rn^2 = \mc{O}(n)$. Then we have:
\begin{align*}
\frac{1}{n}\sum_{i=1}^n\E\left[F(\x_k^i) - F^*\right] 
=&~\mc{O}\left(\frac{\kappa^{2}\left(F(\ol{\x}_0) - F^*\right)}{k^{2}}
+\frac{\kappa\nu_a^2}{n\mu k}\right),
\end{align*}
if~$k$ is large enough such that~$k\gtrsim K_{\text{PL}}$, where
\begin{align*}
K_{\text{PL}} 
:=&~\frac{\lambda^2n\kappa}{(1-\lambda)^3} 
+ \frac{\lambda\kappa^{5/4}}{1-\lambda}
+ \kappa
+ \frac{\lambda^{3/2}\kappa^{11/8}}{(1-\lambda)^{3/2}}
+ \frac{\kappa^{-1/2}}{(1-\lambda)^{3/2}} \n\\
&+ \frac{\lambda^2n\kappa^{1/2}L(F(\ol{\x}_0)-F^*)}{(1-\lambda)^2\nu_a^2}.
\end{align*}
\end{cor}
Theorem~\ref{F_ave_rate} and Corollary~\ref{TRT} are proved in Section~\ref{S_PL_ms_decay}.

\begin{figure*}
\centering
\includegraphics[width=2.2in]{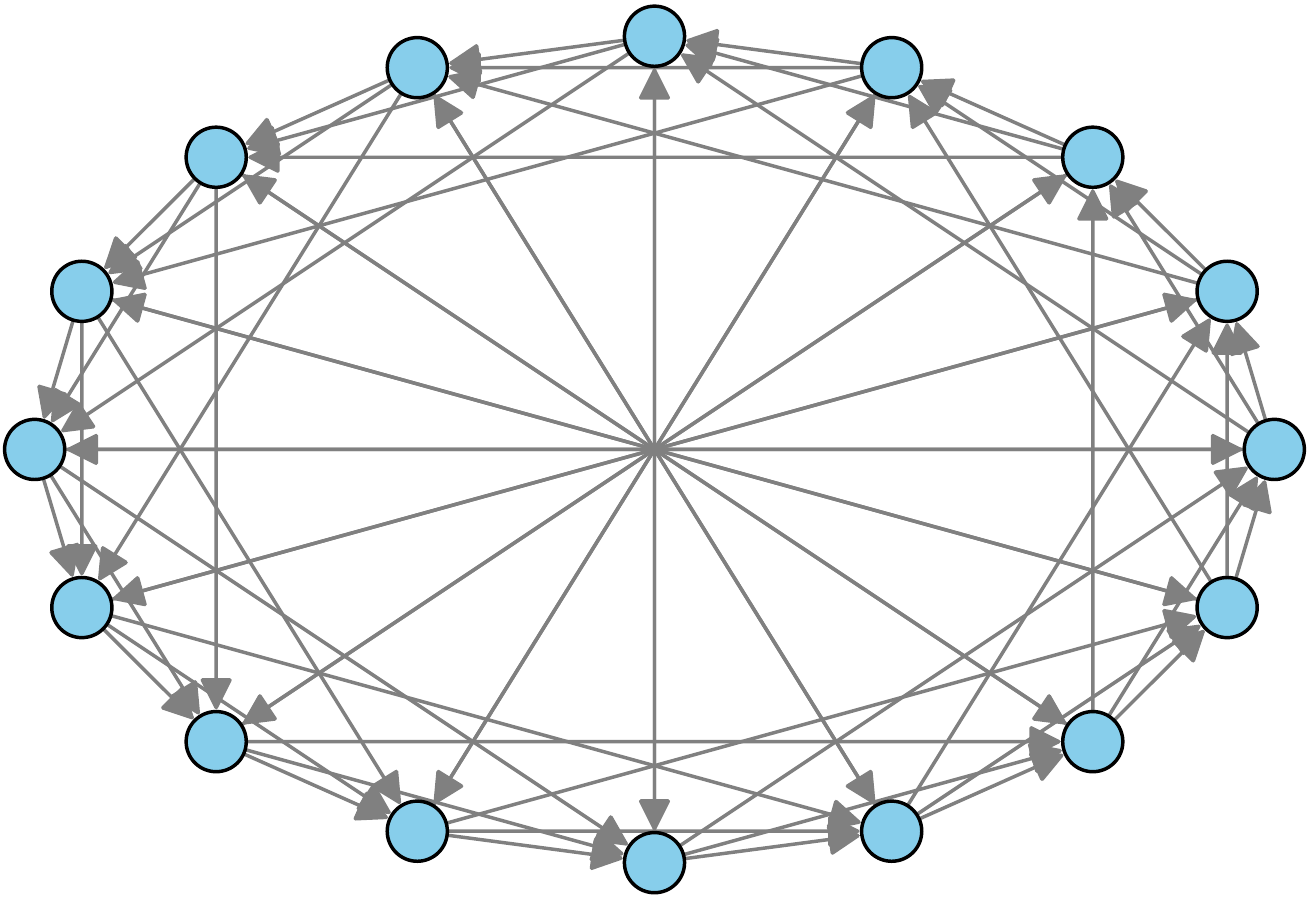}\quad~~
\includegraphics[width=2in]{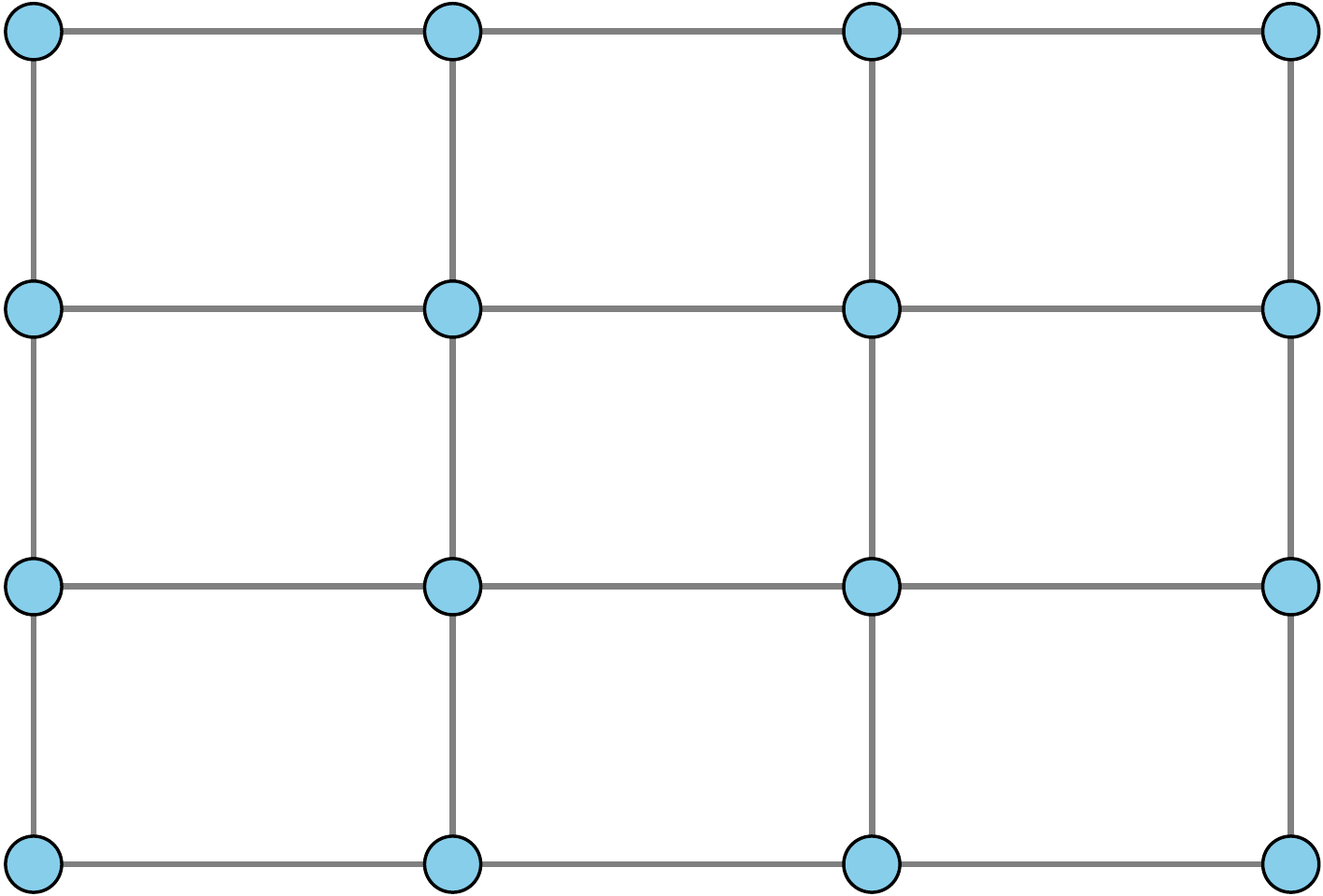}\quad~~
\includegraphics[width=2.2in]{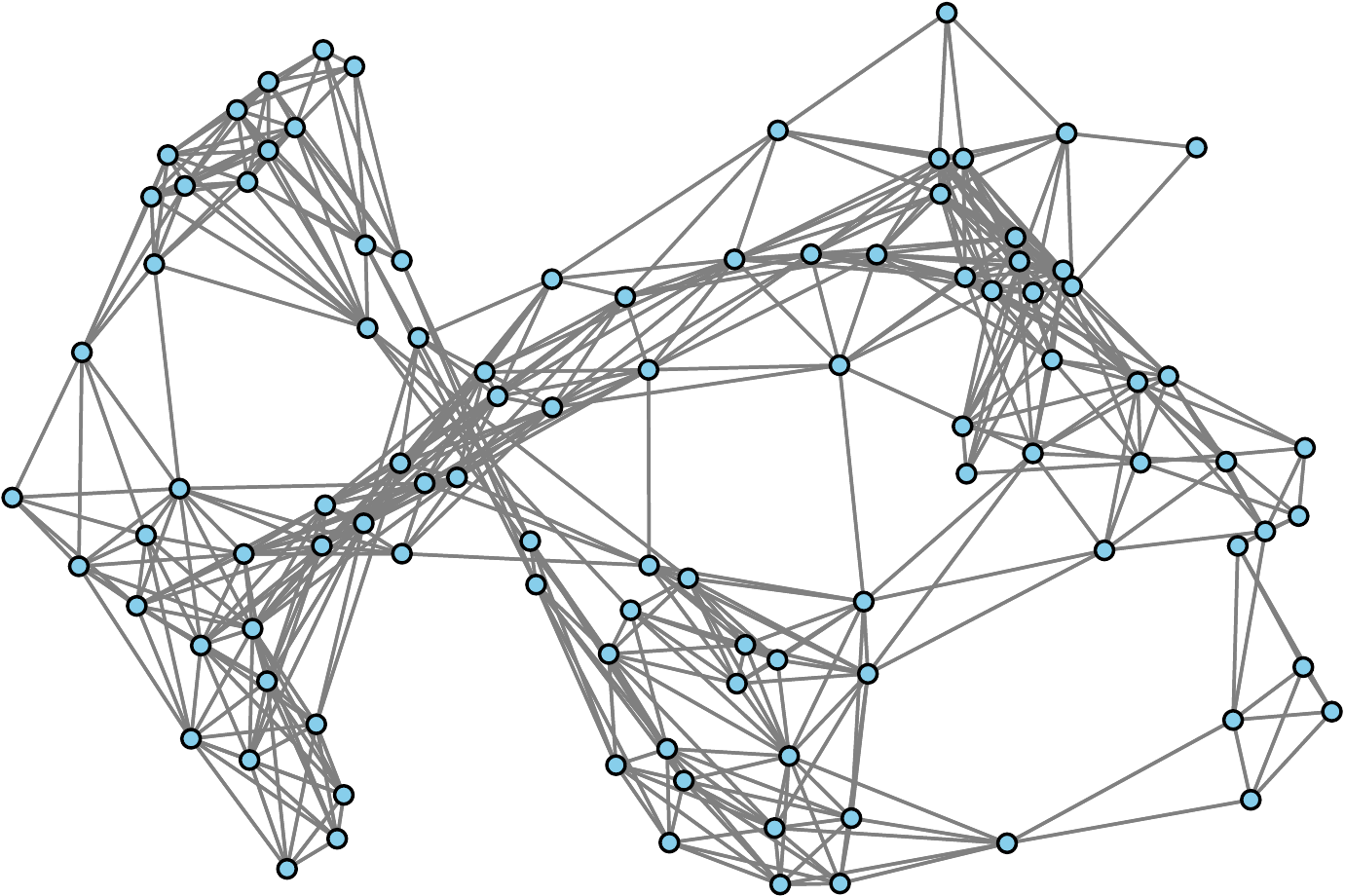}
\caption{A directed exponential graph with~$16$ nodes, an undirected grid graph with~$16$ nodes, and an undirected geometric graph with~$100$ nodes.}
\label{networks_sample}
\end{figure*}

\begin{figure*}
\centering
\includegraphics[width=2.3in]{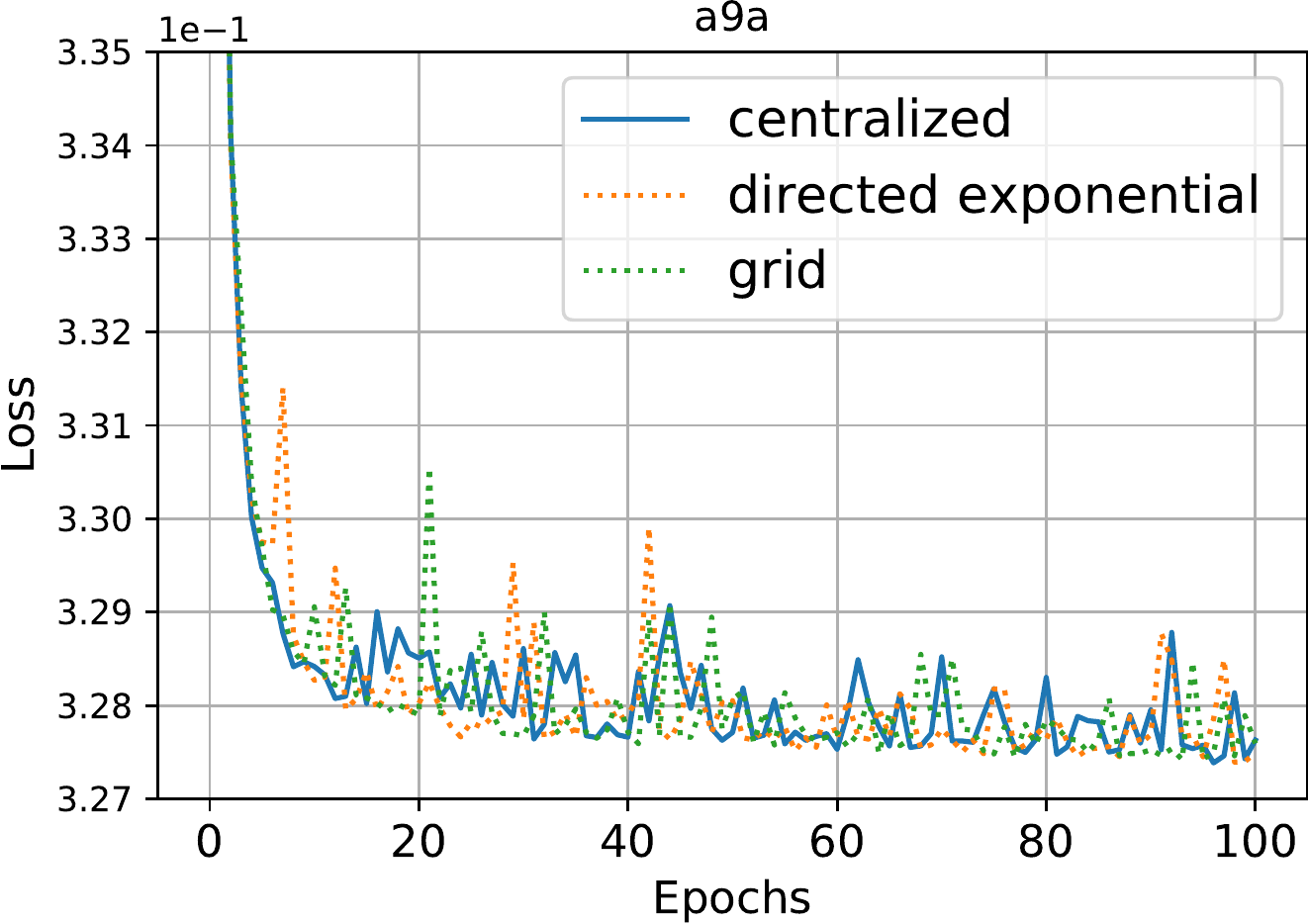}~
\includegraphics[width=2.3in]{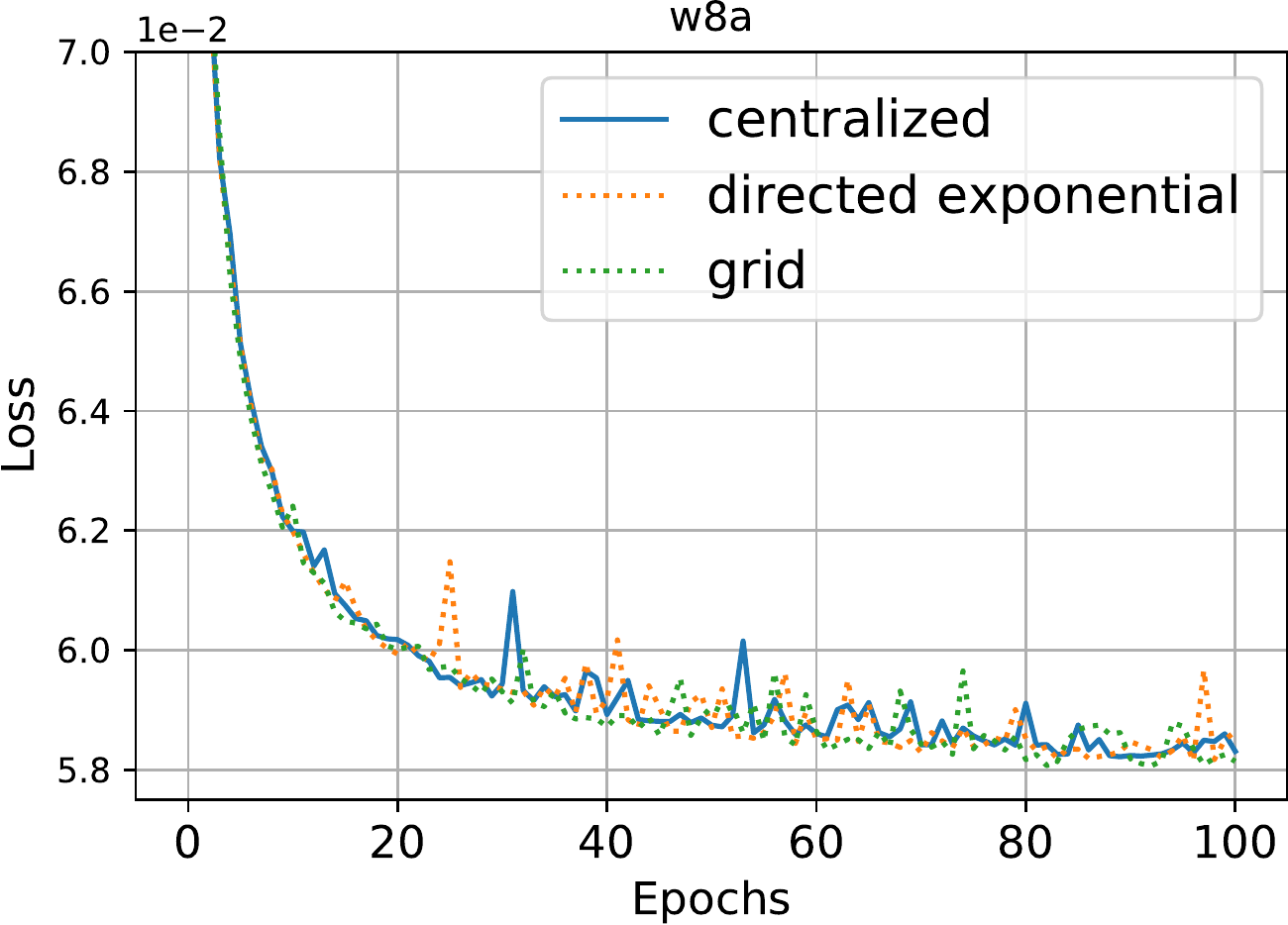}~
\includegraphics[width=2.3in]{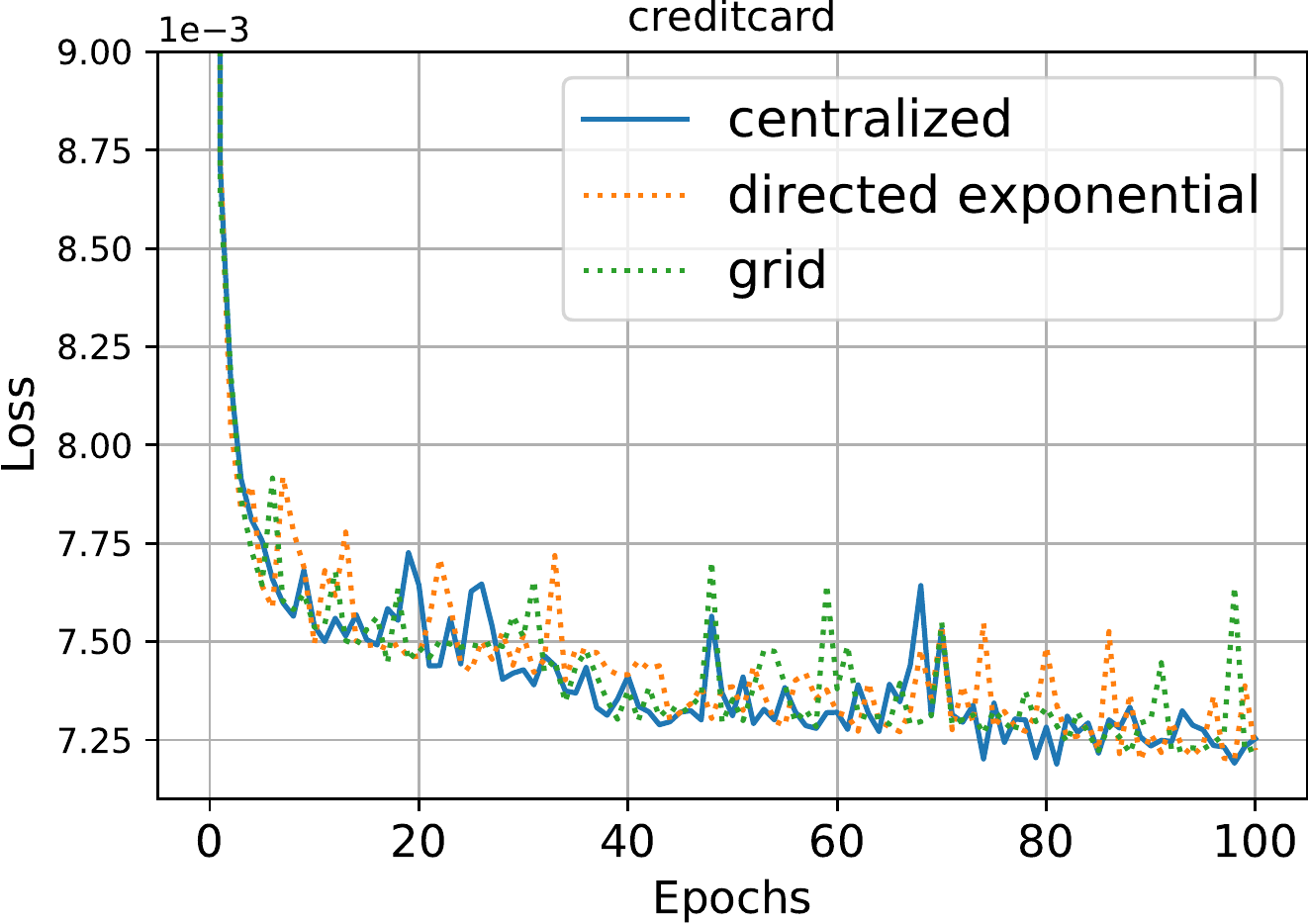}
\caption{The performance of \textbf{\texttt{GT-DSGD}} for non-convex logistic regression over different graphs and comparison with the centralized minibatch \textbf{\texttt{SGD}} on the a9a, w8a and creditcard datasets.}
\label{lr_loss}
\end{figure*}

\begin{figure*}
\centering
\includegraphics[width=2.3in]{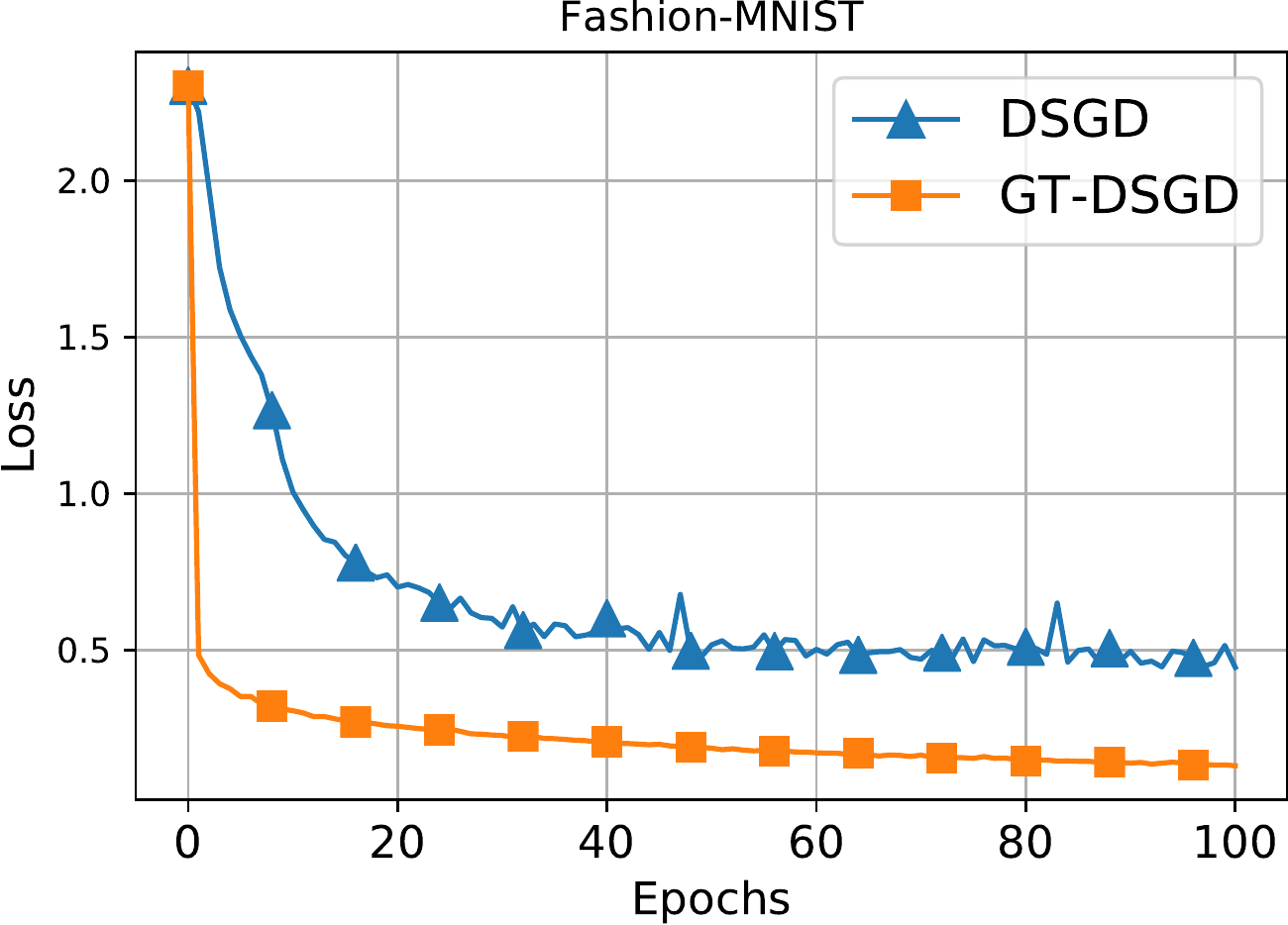}~
\includegraphics[width=2.3in]{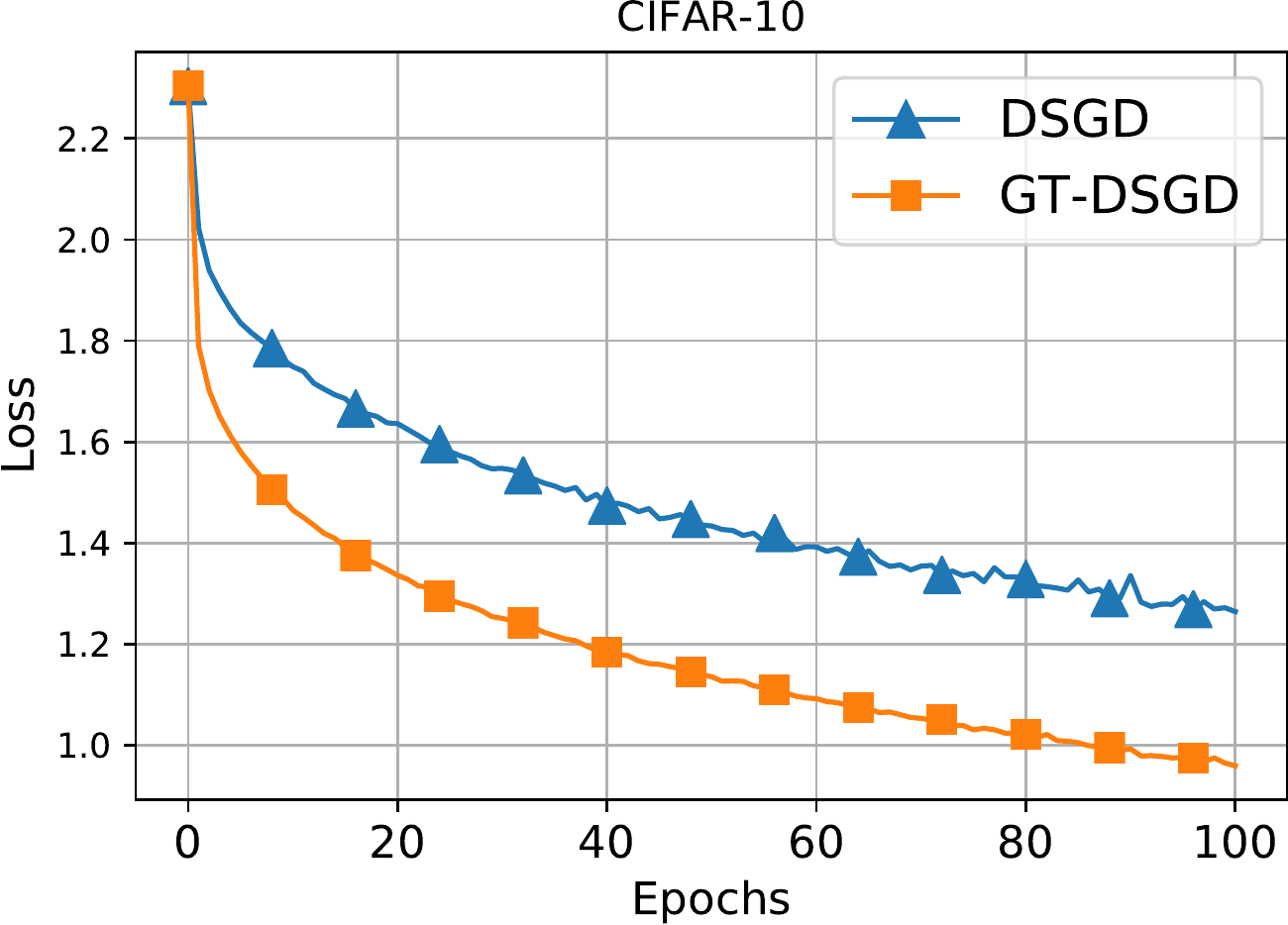}~
\includegraphics[width=2.3in]{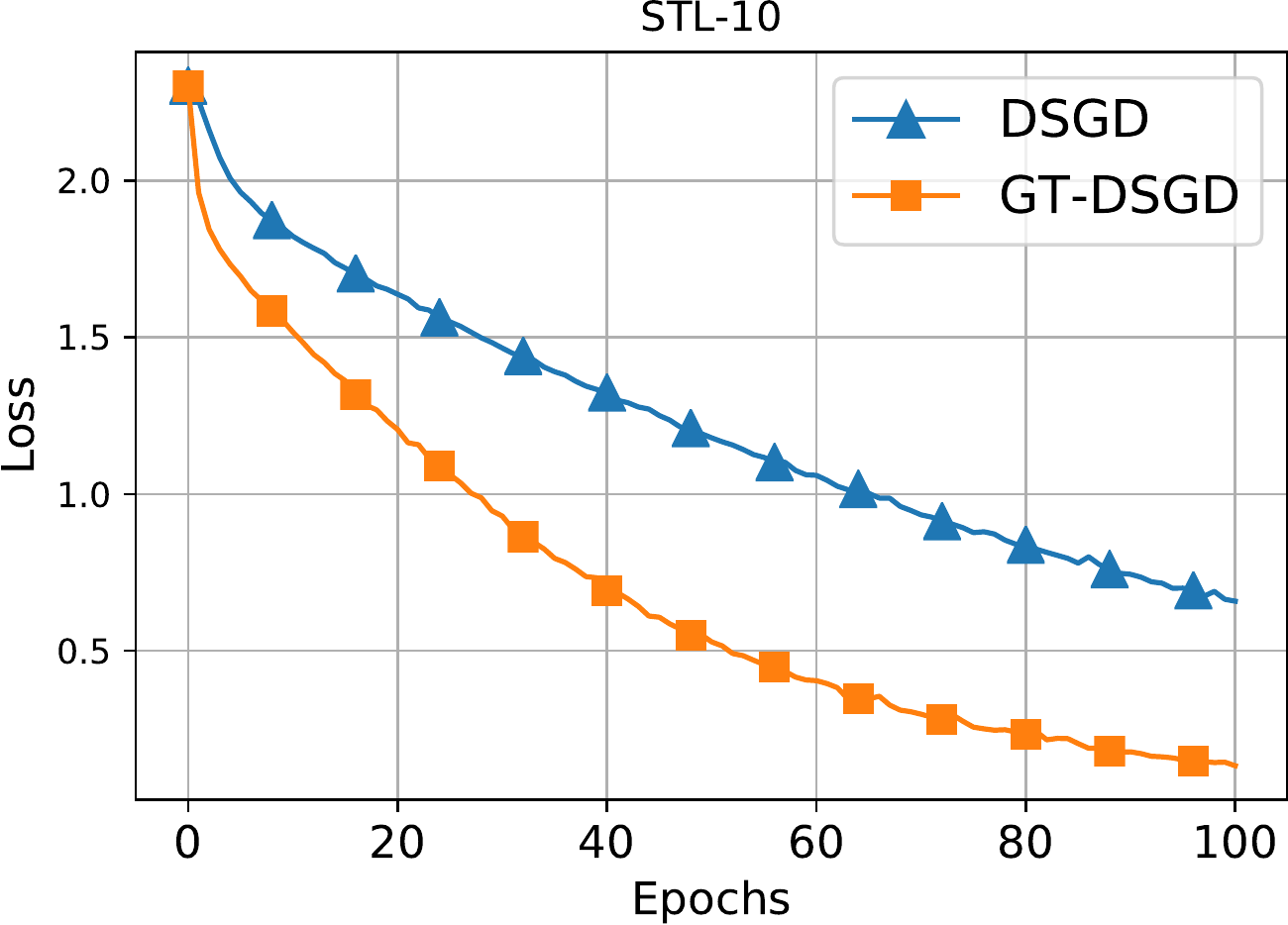}
\caption{Performance comparison between \textbf{\texttt{GT-DSGD}} and \textbf{\texttt{DSGD}} for one-hidden-layer neural network under heterogeneous data distributions across the nodes on the Fashion-MNIST, CIFAR-10 and STL-10 datasets.}
\vspace{-0.2cm}
\label{neural_net_loss}
\end{figure*}

\begin{rmk}[\textbf{Transient time for network independent rate}]
\normalfont
Corollary~\ref{TRT} shows after~$K_{\text{PL}}$ iterations, the convergence rate of \textbf{\texttt{GT-DSGD}} matches that of the centralized minibatch \textbf{\texttt{SGD}}~\cite{OPT_ML} up to constant factors and therefore achieves an asymptotic linear speedup. We now compare this transient time with the existing literature. First, Ref.~\cite{MP_Pu} shows that, under the strong convexity of~$F$, \textbf{\texttt{GT-DSGD}} asymptotically converges at~$\mc{O}(1/k)$; however, the convergence rate derived in~\cite{MP_Pu} depends on arbitrary constants and therefore the transient time is not clear. Second, recent work~\cite{DSGD_Pu,DSGD_SPM_Pu} shows that when each local function~$f_i$ is strongly convex, the corresponding transient time of \textbf{\texttt{DSGD}} is~$\mc{O}\big(n\kappa^6(1-\lambda)^{-2}\big)$. Our results on the transient time~$K_{PL}$ therefore significantly improve upon the dependence of the condition number~$\kappa$ under weaker assumptions on the objective functions, while being moderately worse in terms of the network dependence, i.e.~$1-\lambda$.
\end{rmk}

\section{Numerical Experiments}\label{s_exp}
{\color{black}
% In this section, we present numerical simulations to demonstrate the main theoretical results in Section~\ref{S_mr} with the help of both real-world datasets, summarized in Table~\ref{datasets}, and synthetic functions. 
In this section, we present numerical experiments to demonstrate the main theoretical results in Section~\ref{S_mr} with the help of learning problems on real-world datasets, summarized in Table~\ref{datasets}, and minimizing certain synthetic functions to illustrate the PL condition.
We consider three different graph topologies, i.e., a directed exponential graph with~$16$ nodes, an undirected grid graph with~$16$ nodes, and an undirected geometric graph with~$100$ nodes; see Fig.~\ref{networks_sample}. 
The primitive doubly stochastic weights are set to be equal for the exponential graph and are generated by the Metroplis rule~\cite{PIEEE_nedich} for the grid and the geometric graphs. The second largest singular values~$\lambda$ associated with the weight matrices of these graphs are~$0.6,0.93$ and~$0.99$, respectively. Towards the stochastic gradient oracle, we consider two different setups: (i) each node has access to a finite collection of data samples and the stochastic gradient is computed with respect to one randomly selected data sample at each iteration; (ii) each node has access to the gradient of its local function subject to random noise, with zero-mean and bounded variance, at each iteration. 
The performance metric of interest is the average of global function values across the nodes~$\frac{1}{n}\sum_{i=1}^{n}F(\x_k^i)$, which we refer to as \textit{loss}, versus the number of epochs\footnote{Each epoch is one effective pass of local data samples at each node.} in~(i) and the number of iterations in~(ii). We manually optimize the parameters of all algorithms across all experiments to achieve their best performances.

\begin{table}[hbt]
\caption{A summary of the datasets used in numerical experiments, available at \href{https://www.openml.org/}{https://www.openml.org/}.}
\vspace{-0.3cm}
\begin{center}
\begin{tabular}{|c|c|c|c|}
\hline
\textbf{Dataset} & \textbf{train}  & \textbf{dimension} & \textbf{classes} \\ \hline
a9a & $48,\!832$ & $124$ & $2$ \\ \hline
w8a & $60,\!000$ & $301$ & $2$  \\ \hline
creditcard & $100,\!000$ & $30$ & $2$  \\ \hline
Fashion-MNIST & $60,\!000$ & $785$ & $10$ \\ \hline
CIFAR-10 & $50,\!000$ & $3073$  & $10$ \\ \hline
STL-10 &  $5,\!000$ & $27649$ & $10$ \\ \hline
\end{tabular}
\end{center}
\label{datasets}
\end{table}

To study the convergence behavior of \textbf{\texttt{GT-DSGD}}, we conduct three different experiments: binary classification with non-convex logistic regression~\cite{LR_NCVX}, multiclass classification with neural networks, and minimizing synthetic non-convex functions that satisfy the global PL condition. We compare the performance of \textbf{\texttt{GT-DSGD}} with \textbf{\texttt{DSGD}}~\cite{DSGD_NIPS} to illustrate the advantages of the former in the setting of heterogeneous data distributions across the nodes; moreover, we use the centralized minibatch \textbf{\texttt{SGD}} as the benchmark to illustrate the scenarios in which \textbf{\texttt{GT-DSGD}} achieves a network-independent performance. The experimental results are described in the next subsections.
It can be verified that the numerical results of \textbf{\texttt{GT-DSGD}} are consistent with the theory in this paper.

\begin{figure*}
\centering
\includegraphics[width=1.75in]{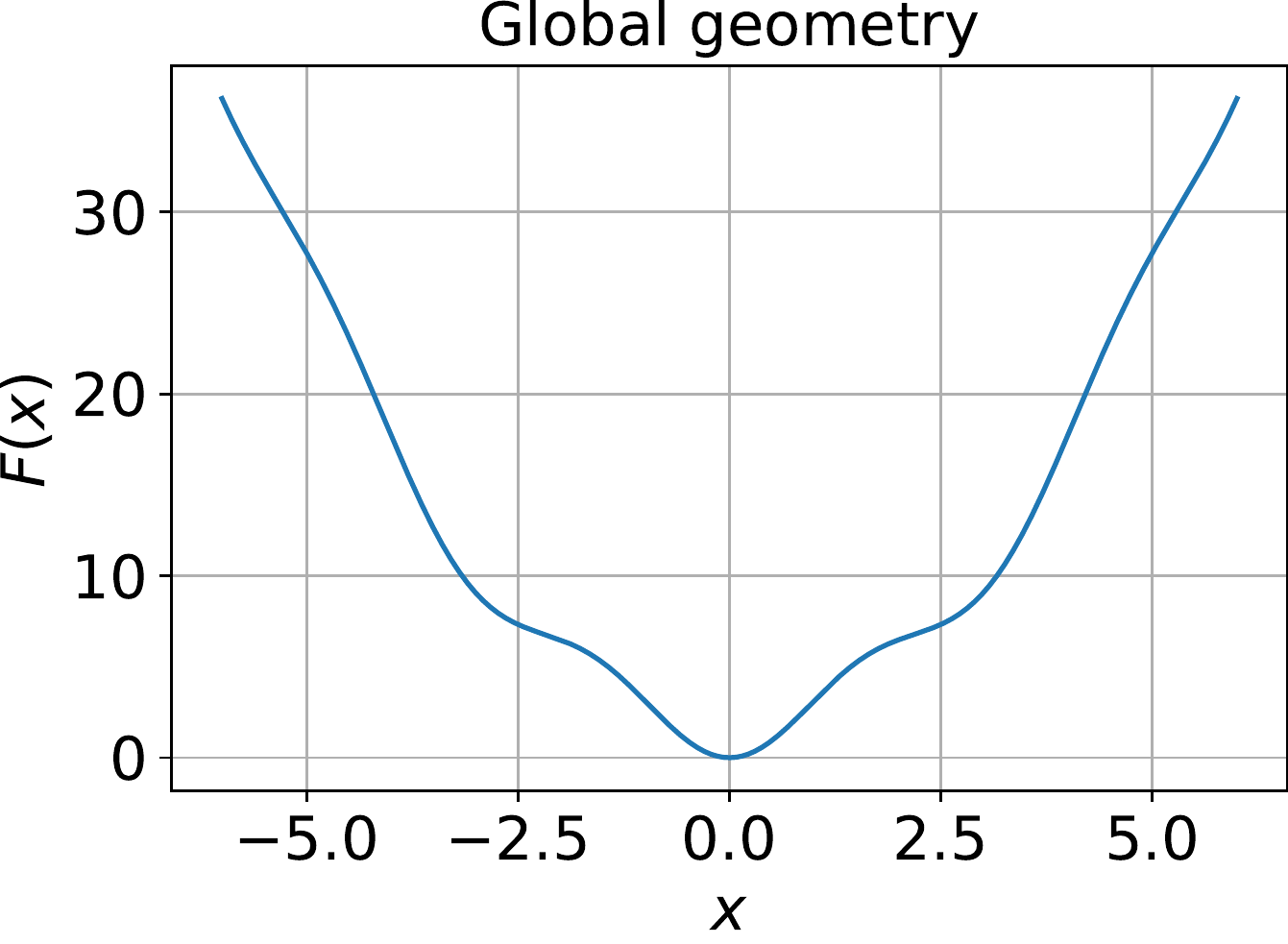} 
\includegraphics[width=1.75in]{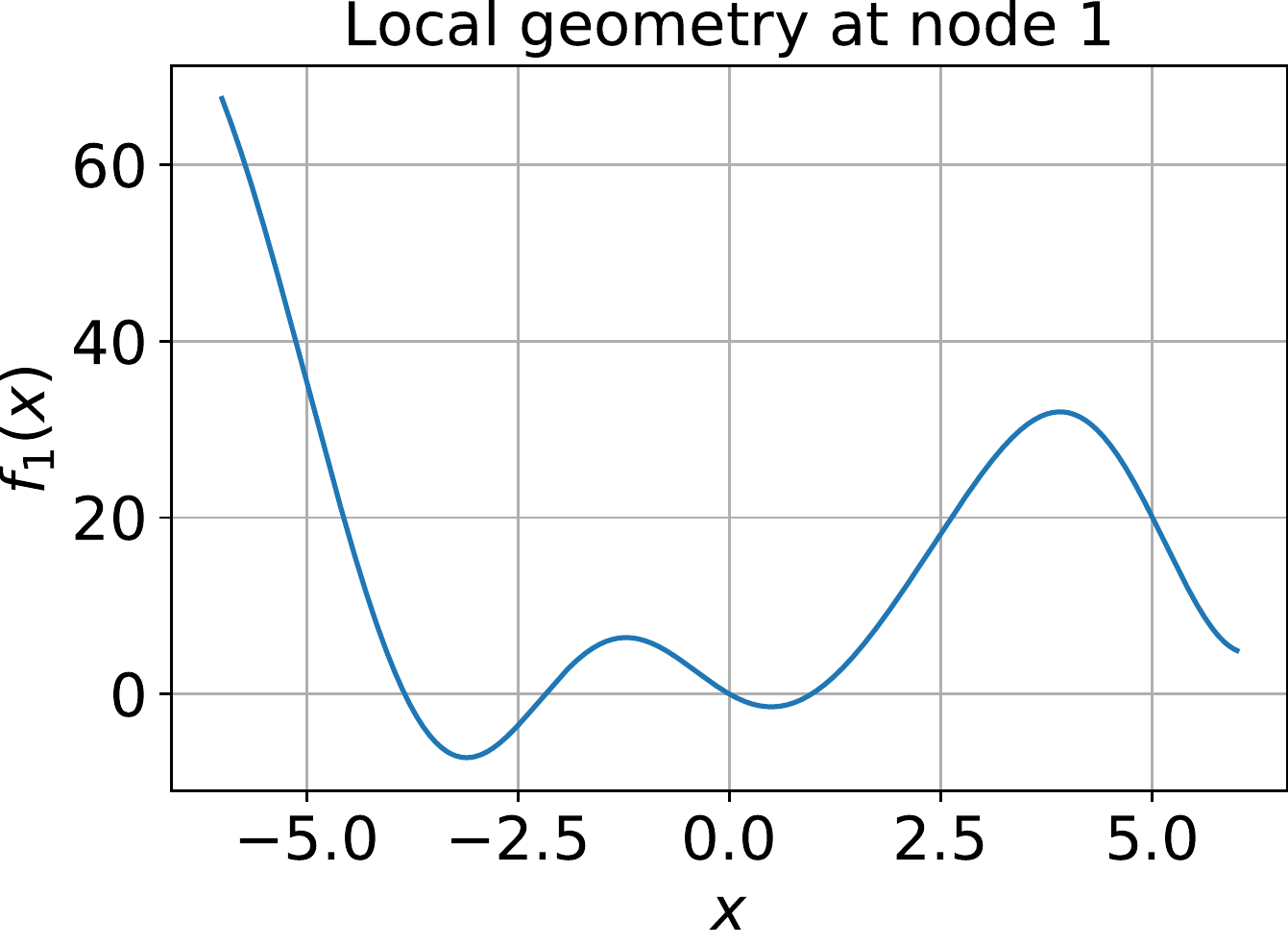}
\includegraphics[width=1.75in]{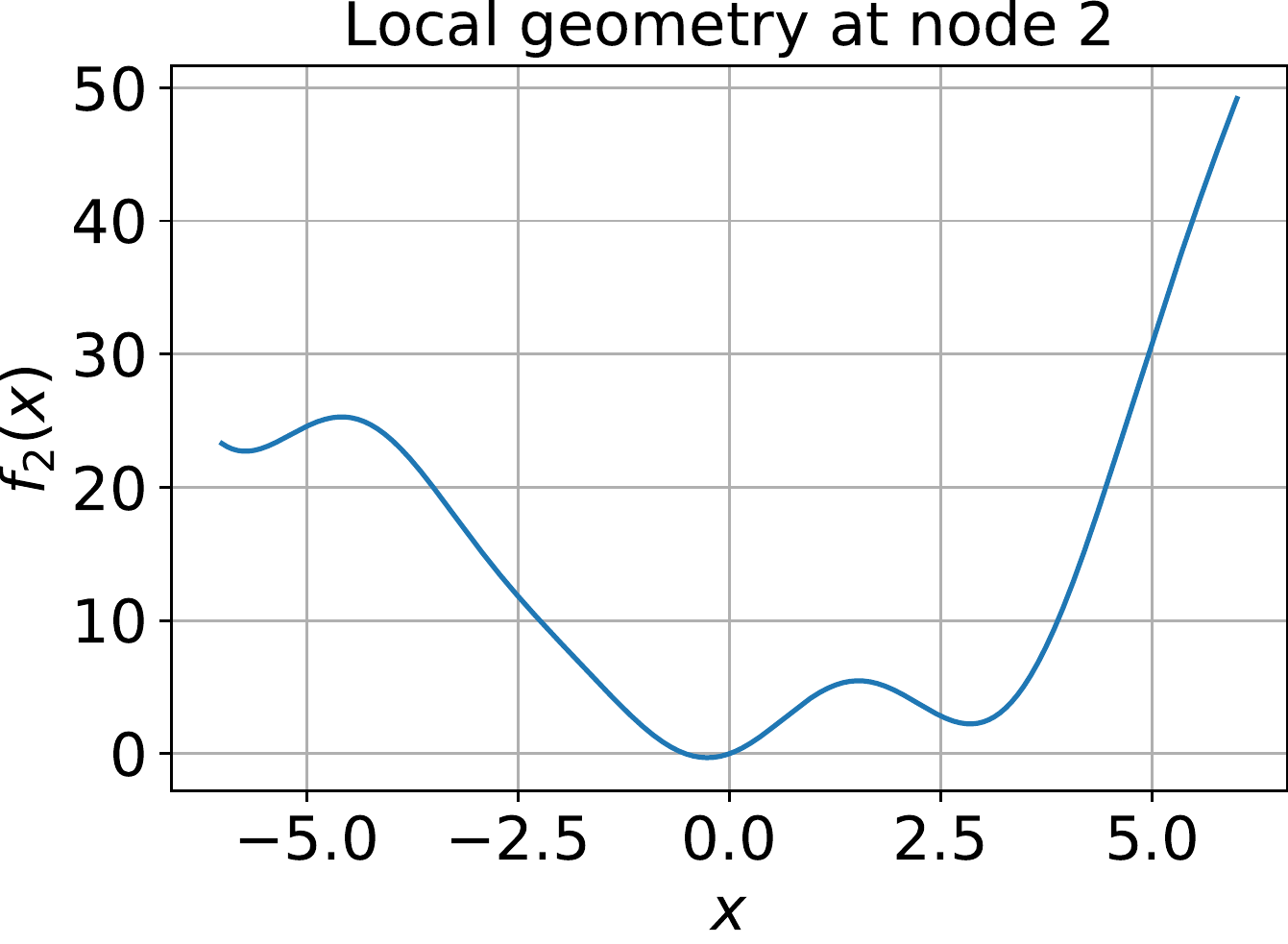}
\includegraphics[width=1.75in]{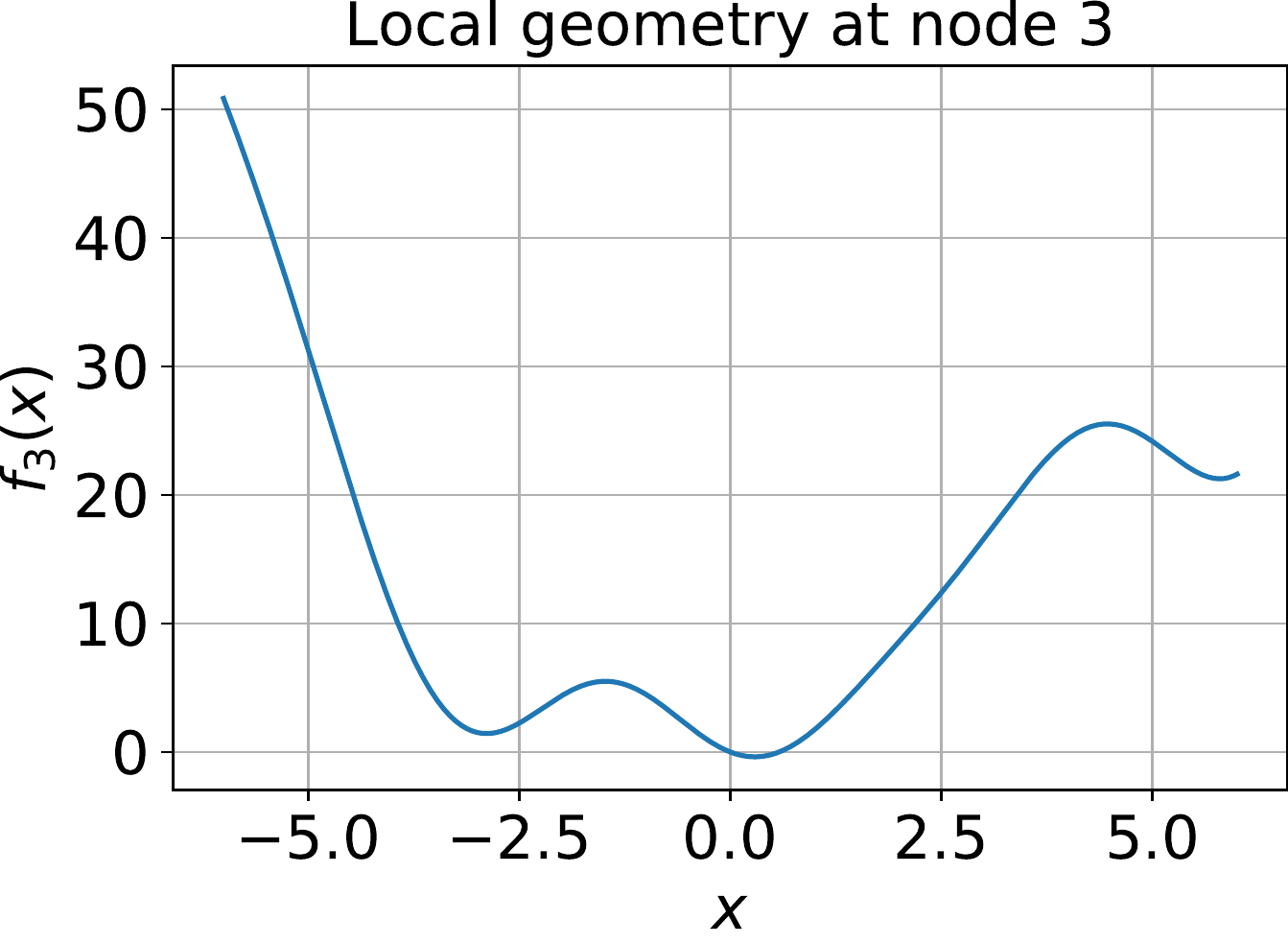}
\caption{The global and local geometries in the experiment with synthetic functions that satisfy the global PL condition.}
\label{PL_functions}
\end{figure*}

\begin{figure*}
\centering
\subfigure[][]{\includegraphics[width=1.75in]{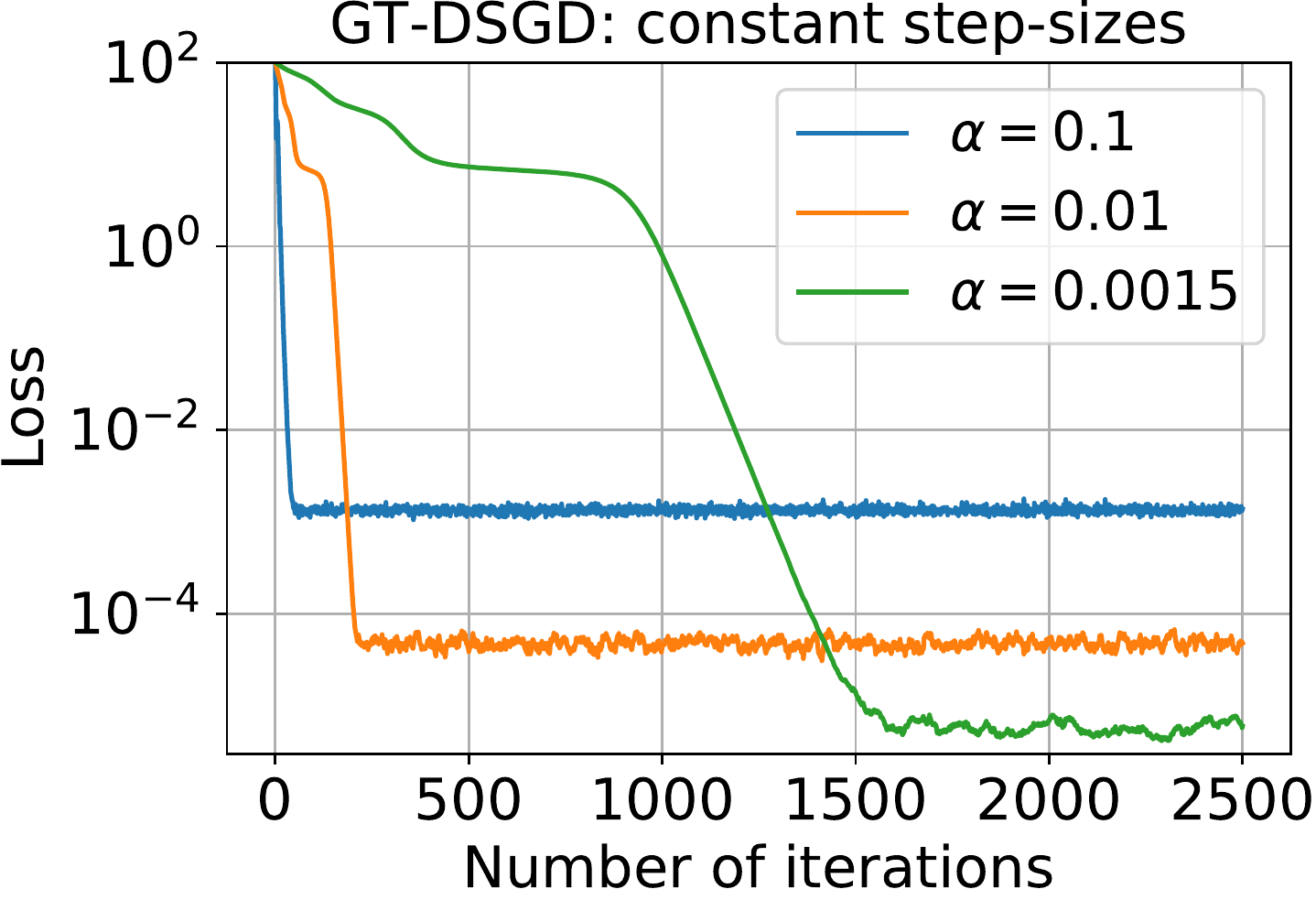}}
\subfigure[][]{\includegraphics[width=1.75in]{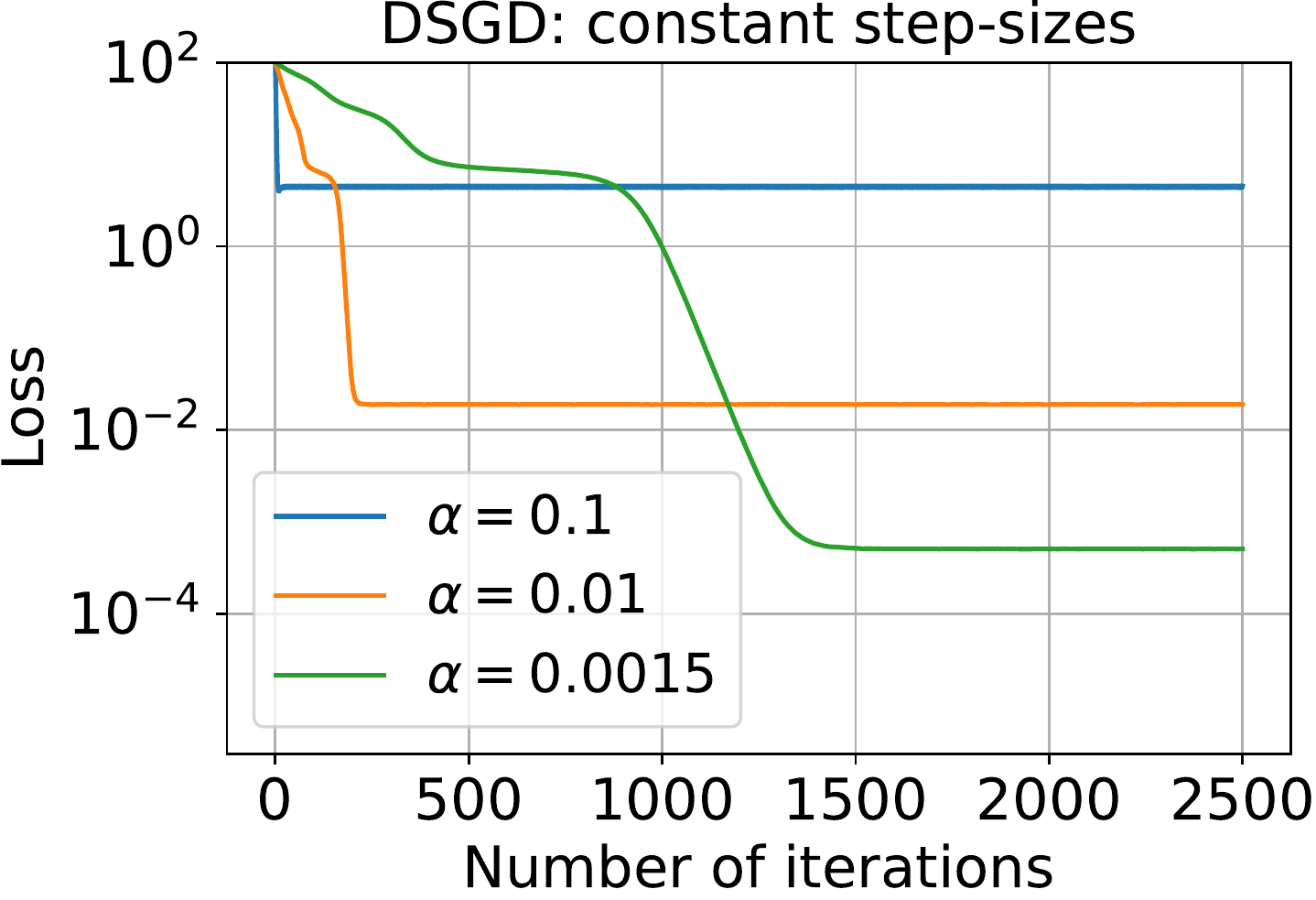}}
\subfigure[][]{\includegraphics[width=1.75in]{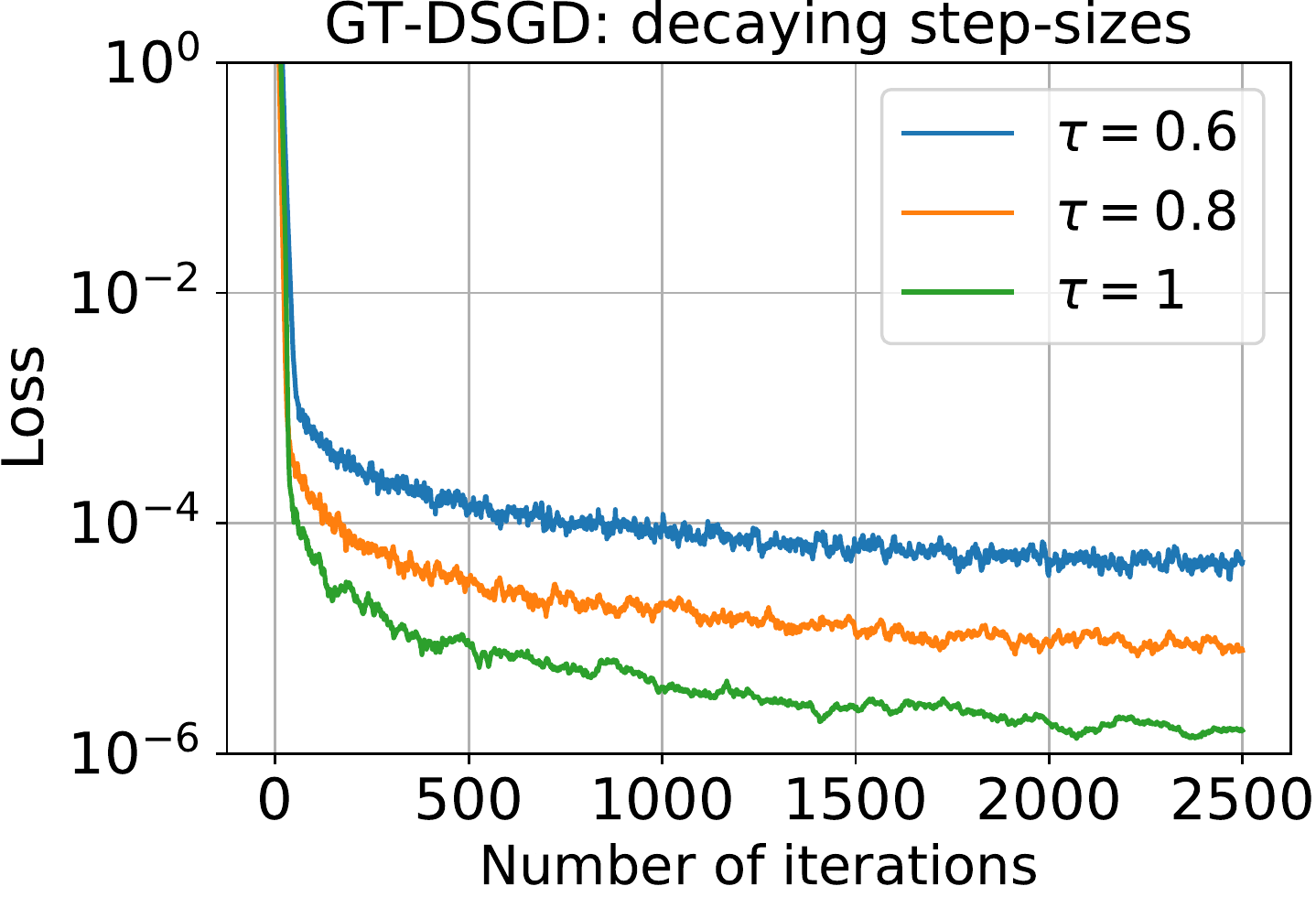}}
\subfigure[][]{\includegraphics[width=1.75in]{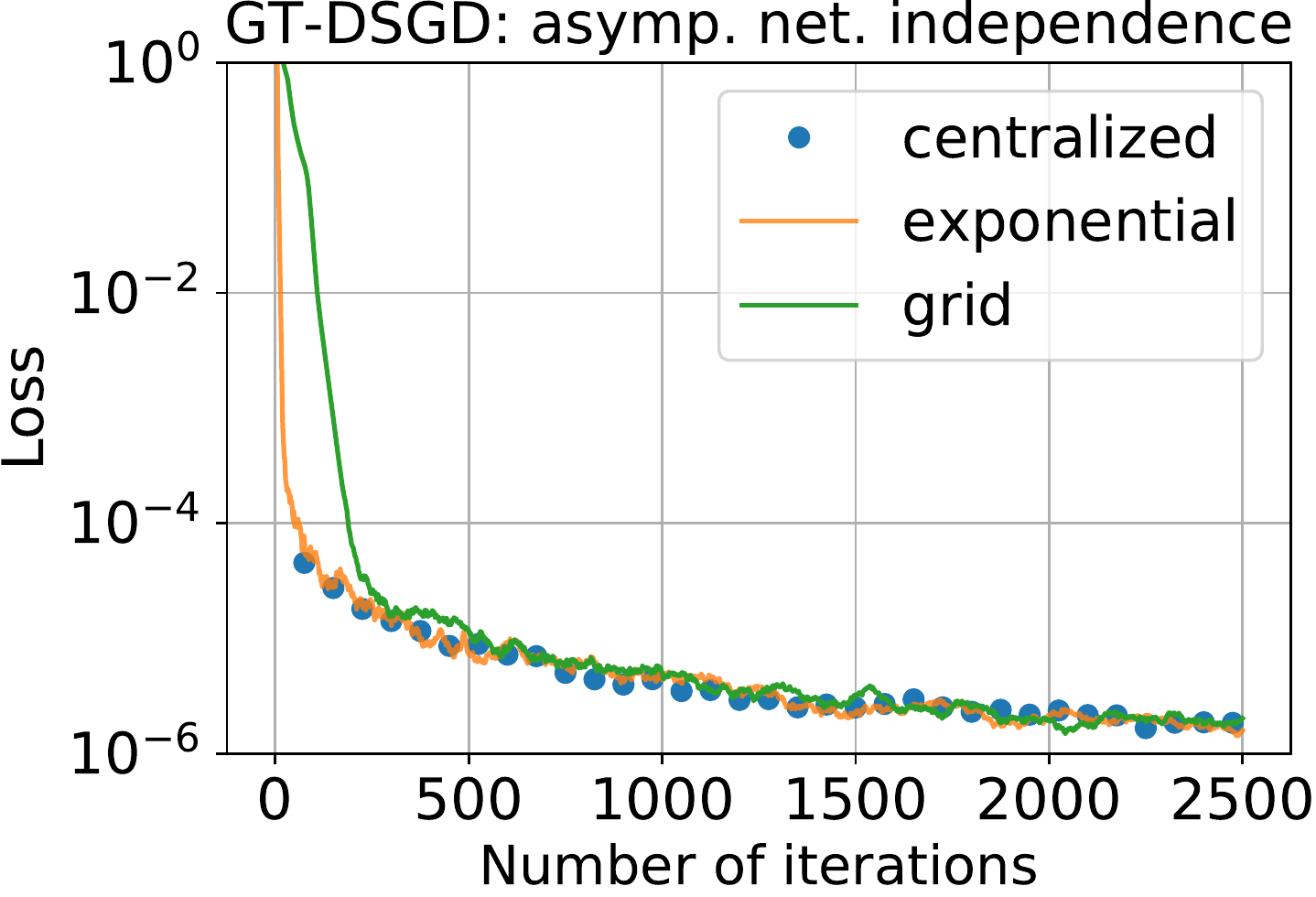}}
\caption{Convergence of \textbf{\texttt{GT-DSGD}} and \textbf{\texttt{DSGD}} under the global PL condition: (a)(b) Inexact linear convergence with different constant step-sizes~$\alpha$. (c) Exact sublinear convergence of \textbf{\texttt{GT-DSGD}} with decaying step-sizes~$\alpha_k = (k+3)^{-\tau}$ under different values of~$\tau$. (d) Exact sublinear convergence of \textbf{\texttt{GT-DSGD}} over different graphs in comparison with the centralized minibatch \textbf{\texttt{SGD}}, all with the decaying step-size~$\a_k = (k+3)^{-1}$.}
\label{PL_loss}
\end{figure*}

\subsection{Non-convex logistic regression for binary classification}
We first consider a binary classification problem with the help of a non-convex logistic regression model~\cite{LR_NCVX}. Specifically, the decentralized optimization problem of interest is given by~$\min_{\x\in\mathbb{R}^p}F(\mb{x}) 
:= \frac{1}{n}\sum_{i=1}^{n}f_{i}(\x)
+ r(\x),$
such that
\begin{align*}
f_{i}(\x) = \frac{1}{m}\sum_{j=1}^m\log\left[1+e^{-(\mb{x}^\top\bds{\theta}_{ij})\xi_{ij}}\right],~r(\x) = \sum_{d=1}^p\frac{R[\x]_d^2}{1+[\x]_d^2},
\end{align*}
where~$\bds{\theta}_{i,j}$ is the feature vector,~$\xi_{i,j}$ is the corresponding binary label, and~$r(\x)$ is a non-convex regularizer with~$R = 10^{-4}$. 

We compare the performance of \textbf{\texttt{GT-DSGD}} over the directed exponential and the grid graphs, both with~$16$ nodes, to the centralized \textbf{\texttt{SGD}} with a minibatch size of~$16$. We consider the best possible constant step-size for both algorithms. 
The numerical results over the a9a, w8a, and creditcard datasets are shown in Fig.~\ref{lr_loss}. It can be observed that, across all datasets, the convergence behavior of \textbf{\texttt{GT-DSGD}} matches that of the centralized minibatch \textbf{\texttt{SGD}} and is independent of the underlying graph topology, as long as the total number of iterations is large enough. This observation is consistent with Corollary~\ref{TRT_ncvx}, demonstrating the network-independent convergence of \textbf{\texttt{GT-DSGD}} under an appropriate constant step-size for general smooth non-convex functions.

\subsection{Neural network for multiclass classification}
We next compare the performance of~\textbf{\texttt{DSGD}} (without gradient tracking) and~\textbf{\texttt{GT-DSGD}}, both with a constant step-size, when the data distributions across the nodes are significantly heterogeneous. To this aim, we consider a harsh problem setup where the data samples are distributed over the~$100$-node geometric graph in Fig.~\ref{networks_sample} such that each node has the same number of data samples and the samples belong to only one or two classes (out of~$10$ possible classes). We consider decentralized training of a neural network with one fully connected hidden layer of~$64$ neurons and sigmoid activation. The experimental results over the Fashion-MNIST, CIFAR-10, and STL-10 datasets are shown in Fig.~\ref{neural_net_loss}. We observe that \textbf{\texttt{GT-DSGD}} significantly outperforms \textbf{\texttt{DSGD}} in this setting, demonstrating the robustness of \textbf{\texttt{GT-DSGD}} to heterogeneous data across the nodes; see also Remark~\ref{GT_versus_DSGD}.

\vspace{-0.2cm}
\subsection{Synthetic functions that satisfy the global PL condition}
Finally, we show the performance of \textbf{\texttt{GT-DSGD}} when the global function satisfies the PL condition and compare it with \textbf{\texttt{DSGD}} and the centralized minibatch \textbf{\texttt{SGD}}. In particular, each local function is chosen as~$f_i(x) = x^2 + 3\sin^2(x) + a_ix\cos(x)$, such that~${\sum_{i=1}^na_i = 0}$ and~${a_i\neq0,\forall i\in\mc{V}}$, leading to the global function~${F(x) = x^2 + 3\sin^2(x)}$,
which is clearly non-convex and further satisfies the PL condition~\cite{PL_1}. It can be verified that each local function is highly nonlinear and significantly different from the global function; see Fig.~\ref{PL_functions}. We inject random Gaussian noise with mean~$0$ and the standard deviation~$0.5$ to the gradient computation at each node. The corresponding numerical results can be found in Fig.~\ref{PL_loss}, where the experiments in Fig.~\ref{PL_loss}(a)-(c) are performed over the directed exponential graph with~$16$ nodes.
It can be observed from Fig.~\ref{PL_loss}(a) that \textbf{\texttt{GT-DSGD}} achieves inexact linear convergence under constant step-sizes; moreover, a smaller step-size leads to a smaller steady-state error but at a slower rate. Compared with the convergence of \textbf{\texttt{DSGD}} under constant step-sizes shown in Fig.~\ref{PL_loss}(b), \textbf{\texttt{GT-DSGD}} achieves a smaller steady-state error much faster benefiting from gradient tracking that effectively exploits the global geometry. Fig.~\ref{PL_loss}(c) shows that \textbf{\texttt{GT-DSGD}} achieves exact sublinear convergence to the optimal solution with decaying step-sizes of the form~${\alpha_k = (k+3)^{-\tau}}$ under different values of~$\tau$ chosen in~$(0.5,1]$. Clearly, a larger $\tau$ leads to a faster rate as Theorem~\ref{PL_as} suggests. Finally, we observe from Fig.~\ref{PL_loss}(d) that the convergence rate of \textbf{\texttt{GT-DSGD}} with~${\tau = 1}$ matches that of the centralized minibatch \textbf{\texttt{SGD}} with the same decaying step-size after a small number of transient iterations over different graphs. This phenomenon demonstrates the asymptotically network-independent and optimal~$\mc{O}(1/k)$ rate achieved by \textbf{\texttt{GT-DSGD}}. This observation is consistent with Theorem~\ref{F_ave_rate}.
}

\vspace{-0.1cm}
\section{Convergence analysis: the general non-convex case}\label{S_ncvx}
It is straightforward to verify that the random variables generated by \textbf{\texttt{GT-DSGD}} are square-integrable and that $\mb{x}_{k},\mb{y}_{k}$ are $\F_k$-measurable and $\g(\x_k,\bxi_k)$ is $\F_{k+1}$-measurable, $\forall k$. In this section, we derive general bounds on the stochastic gradient tracking process, which may be of independent interest, and prove Theorem~\ref{conv_ncvx}.
We start by presenting some standard results on decentralized stochastic gradient tracking algorithms; their proofs can be found, e.g., in~\cite{harnessing,MP_Pu,GTVR}.
\begin{lem}\label{basic}
Under Assumption~\ref{f}-\ref{o}, We have the following:
\begin{enumerate}[(a)]
    \item $\ln\mb{W}\mb{x}-\mb{J}\mb{x}\rn\leq\lambda\ln\mb{x}-\mb{J}\mb{x}\rn,\forall\mb{x}\in\mbb{R}^{np}$. \label{W} 
    \item $\ol{\y}_{k+1} = \agk,\forall k\geq0$.        \label{track}
    \item $\big\|\ol{\nf}_{k}-\nabla F(\ol{\mb{x}}_k)\big\|^2\leq  \frac{L^2}{n}\left\|\mb{x}_k-\mb{J}\mb{x}_k\right\|^2, \forall k\geq0$. \label{Lbound}
    \item {$\E[\langle \g_i(\x_k^i,\bxi_k^i) - \nabla f_i(\x_k^i),\g_r(\x_k^r,\bxi_k^r) - \nabla f_r(\x_k^r)\rangle|\F_k] = 0$}, $\forall k\geq0, \forall i,r\in\mc{V}$ such that~$i\neq r$. \label{indep}
    \item $\E\big[\|\agk - \ol{\nf}_{k}\|^2 |\F_k \big] \leq \nu_a^2/n,\forall k\geq0$. \label{vrbound}
\end{enumerate}
\end{lem}
As a consequence of the state update of \textbf{\texttt{GT-DSGD}} described in~\eqref{x} and Lemma~\ref{basic}(\ref{track}), we have:~$\forall k\geq0$, 
\begin{align}\label{ave}
\ol{\x}_{k+1} 
= \ol{\x}_k - \a_k \ol{\y}_{k+1} 
= \ol{\x}_k - \a_k \agk, 
\end{align}
i.e., the mean state~$\ol{\x}_k$ of the network proceeds in the direction of the average of local stochastic gradients~$\agk$. The following lemma provides several useful relations on the consensus process of the state vectors across the network~\cite{GTVR}.
\begin{lem}\label{cons}
Let Assumption~\ref{net} hold. We have the following inequalities:~$\forall k\geq0$,
\begin{gather*}
\left\|\mb{x}_{k+1}-\mb{J}\mb{x}_{k+1}\right\|^2 \leq \frac{1+\lambda^2}{2}\left\|\mb{x}_k-\mb{J}\mb{x}_k\right\|^2 \n\\
\qquad\qquad\qquad\qquad\qquad\qquad~~+ \frac{2\a_k^2\lambda^2}{1-\lambda^2}\left\|\mb{y}_{k+1}-\mb{J}\mb{y}_{k+1}\right\|^2. \\
\left\|\mb{x}_{k+1}-\mb{J}\mb{x}_{k+1}\right\|^2 \leq 2\lambda^2\left\|\mb{x}_k-\mb{J}\mb{x}_k\right\|^2 \n\\
\qquad\qquad\qquad\qquad\qquad\qquad~~+ 2\a_k^2\lambda^2\left\|\mb{y}_{k+1}-\mb{J}\mb{y}_{k+1}\right\|^2. \n\\
\left\|\mb{x}_{k+1}-\mb{J}\mb{x}_{k+1}\right\| \leq \lambda\left\|\mb{x}_k-\mb{J}\mb{x}_k\right\|^2 + \a_k\lambda\left\|\mb{y}_{k+1}-\mb{J}\mb{y}_{k+1}\right\|.
\end{gather*}
\end{lem}

\subsection{A descent inequality}
In this subsection, we establish a key descent inequality that characterizes the expected decrease of the value of the global objective function~$F$ over each iteration in light of~\eqref{ave}.
\begin{lem}\label{ds}
Let~Assumptions~\ref{f}-\ref{o} hold. If~$0<\a_k\leq\frac{1}{2L}$, then we have:~$\forall k\geq0$,
\begin{align*}
\E\left[F(\ol{\x}_{k+1})|\F_k\right] 
\leq&~F(\ol{\x}_{k}) -\frac{\a_k}{2}\ln\nabla F(\ol{\x}_{k})\rn^2 - \frac{\a_k}{4}\ln\ol{\nf}_{k}\rn^2 \n\\
&+ \frac{\a_k L^2}{2}\frac{\ln\x_k - \mb{J}\x_k\rn^2}{n} +\frac{\a_k^2 L\nu^2_a}{2n}.
\end{align*}
\end{lem}
\begin{proof}
Since~$F$ is~$L$-smooth, we have~\cite{book_polyak}:~$\forall \x,\y\in\R^p$,
\begin{align}\label{L_lemma}
F(\y) \leq F(\x) + \l \nabla F(\x), \y - \x \r + \frac{L}{2}\|\y-\x\|^2.
\end{align}
Setting~$\y = \ol{\x}_{k+1}$ and~$\x = \ol{\x}_{k}$ in~\eqref{L_lemma} to obtain:~$\forall k\geq0$,
\begin{align*}
F(\ol{\x}_{k+1}) \leq F(\ol{\x}_{k}) -\a_k\l \nabla F(\ol{\x}_{k}), \agk \r + \frac{\a_k^2L}{2}\|\agk\|^2.
\end{align*}
Conditioning on~$\F_k$, by~$\E[\agk|\F_k] = \ol{\nf}_{k}$, obtains:~$\forall k\geq0$,
\begin{align}\label{d0}
&\E[F(\ol{\x}_{k+1})|\F_k] \n\\
\leq&~F(\ol{\x}_{k}) -\a_k\l \nabla F(\ol{\x}_{k}), \ol{\nf}_{k} \r + \frac{\a_k^2L}{2}\E\left[\ln\agk\rn^2|\F_k\right] \nonumber\\
=&~F(\ol{\x}_{k}) -\frac{\a_k}{2}\ln\nabla F(\ol{\x}_{k})\rn^2 - \frac{\a_k}{2}\ln\ol{\nf}_{k}\rn^2 \n\\
&+ \frac{\a_k}{2}\ln\nabla F(\ol{\x}_{k}) - \nf_{k}\rn^2
+ \frac{\a_k^2L}{2}\E\left[\ln\agk\rn^2|\F_k\right] \nonumber\\
\leq&~F(\ol{\x}_{k}) -\frac{\a_k}{2}\|\nabla F(\ol{\x}_{k})\|^2 - \frac{\a_k}{2}\|\ol{\nf}_{k}\|^2\n\\
&+ \frac{\a_k L^2}{2n}\|\x_k - \mb{J}\x_k\|^2 + \frac{\a_k^2L}{2}\E\left[\ln\agk\rn^2|\F_k\right],
\end{align}
where the equality above uses~$\l\x,\y\r = \frac{1}{2}(\|\x\|^2 + \|\y\|^2 - \|\x-\y\|^2 ),\forall \x,\y\in\R^p$, and the last inequality is due to Lemma~\ref{basic}(\ref{Lbound}). For the last term in~\eqref{d0}, note that:~$\forall k\geq0$,
\begin{align}\label{g_ave}
\E\left[\left\|\agk\right\|^2|\F_k\right] 
=&~\E\left[\left\|\agk - \ol{\nf}_{k} + \ol{\nf}_{k}\right\|^2|\F_k\right] \n\\
=&~\E\left[\left\|\agk - \ol{\nf}_{k}\right\|^2|\F_k\right] + \|\ol{\nf}_{k}\|^2 \n\\
\leq&~\nu_a^2/n + \left\|\ol{\nf}_{k}\right\|^2,
\end{align}
where the second equality uses that~$\ol{\nf}_{k}$ is~$\F_k$-measurable and~$\E[\agk|\F_k] = \ol{\nf}_{k}$, and the last inequality uses Lemma~\ref{basic}(\ref{vrbound}).
We now use~\eqref{g_ave} in~\eqref{d0} to obtain:~$\forall k\geq0$,
\begin{align*}
&\E[F(\ol{\x}_{k+1})|\F_k] 
\leq F(\ol{\x}_{k}) -\frac{\a_k}{2}\|\nabla F(\ol{\x}_{k})\|^2 
+ \frac{\a_k^2L\nu_a^2}{2n}  \n\\
&\qquad\qquad~~- \frac{\a_k\left(1-\a_k L\right)}{2}\|\ol{\nf}_{k}\|^2 + \frac{\a_k L^2}{2n}\|\x_k - \mb{J}\x_k\|^2.
\end{align*}
The proof follows by noting that~$1-\a_k L\geq\frac{1}{2}$, if~{$0<\a_k\leq\frac{1}{2L}$}, $\forall k\geq0$, in the inequality above.
\end{proof}
Compared with the corresponding descent inequality for the centralized stochastic gradient descent, see, e.g.,~\cite{OPT_ML,book_polyak},
the descent inequality for \textbf{\texttt{GT-DSGD}} derived in
Lemma~\ref{ds} has an additional network consensus error term~$\ln\x_k - \mb{\mb{J}}\x_k\rn$. We therefore seeks for means to control this perturbation in order to establish the convergence of \textbf{\texttt{GT-DSGD}}. We will bound the consensus and the gradient tracking error jointly.

%\vspace{-0.3cm}
\subsection{Bounding the gradient tracking error}\label{sec_GT_general_bounds}
% \begin{proof}
% Using the solution estimate update~\eqref{x} and the fact that~$\mb{J} \mb{\mb{W}} = \mb{J}$, we have:
% \begin{align}
% &\left\|\mb{x}_{k+1}-\mb{J}\mb{x}_{k+1}\right\|^2 \n\\
% =&~ \left\|\mb{W}(\mb{x}_{k} - \a_k\mb{y}_{k+1})-\mb{J}\left(\mb{x}_{k} - \a_k\mb{y}_{k+1}\right)\right\|^2 \nonumber\\
% =&~ \left\|\mb{W}\mb{x}_{k} - \mb{J}\mb{x}_{k}-\a_k \left(\mb{W}\mb{y}_{k+1}-\mb{J}\mb{y}_{k+1}\right)\right\|^2
% \label{consensus0}
% \end{align}
% We use Young's inequality that~$\|\mb{a}+\mb{b}\|^2\leq(1+\eta)\|\mb{a}\|^2+(1+\eta^{-1})\|\mb{b}\|^2,\forall\mb{a},\mb{b}\in\mathbb{R}^{np},\forall\eta>0$, and Lemma~\ref{basic}(\ref{W}) in~\eqref{consensus0} to obtain:
% \begin{align*}
% \left\|\mb{x}_{k+1}-\mb{J}\mb{x}_{k+1}\right\|^2
% \leq&~\left(1+\eta\right)\lambda^2\left\|\mb{x}_{k} - \mb{J}\mb{x}_{k}\right\|^2 \n\\
% &+\left(1+\eta^{-1}\right)\a_k^2\lambda^2\left\|\mb{y}_{k+1} - \mb{J}\mb{y}_{k+1}\right\|^2.
% \end{align*}
% Setting~$\eta$ respectively as $\frac{1-\lambda^2}{2\lambda^2}$ and~$1$ in the inequality above proves the first two inequalities in the lemma. The third inequality in the lemma follows by taking the square root of~\eqref{consensus0} and using Lemma~\ref{basic}(\ref{W}). 
% \end{proof}
\noindent
% Note that Lemma~\ref{ds} is a meaningful inequality that characterizes the decrease of the global function value at each iteration. In the rest of the proof, we seek the conditions on the step-sizes under which 
% \begin{align*}
% - \frac{\a_k}{4}\|\ol{\nf}_{k}\|^2  + \frac{\a_k L^2}{2n}\|\x_k - \mb{J}\x_k\|^2 < 0,
% \end{align*}
% and hence Lemma~\ref{ds} reduces to the descent inequality for centralized stochastic gradient descent.
%\newpage
In this subsection, we analyze the gradient tracking process.
\begin{lem}\label{gt}
Let Assumption~\ref{f}-\ref{o} hold. We have:~$\forall k\geq0$,
\begin{align*}
&\E\big[\ln\y_{k+2} -\mb{J}\y_{k+2}\rn^2 \big] \n\\
\leq&~\lambda^2\E\big[\ln \mb{y}_{k+1}-\mb{J}\mb{y}_{k+1}\rn^2\big]
+\lambda^2\E\big[\ln\g_{k+1} - \g_{k}\rn^2\big] \n\\
&+2\E\left[\l \left(\mb{W}-\mb{J}\right)\mb{y}_{k+1},\left(\mb{W} -\mb{J}\right)\left(\nf_{k} - \g_{k}\right) \r \right] \n\\
&+2\E\left[\l \left(\mb{W}-\mb{J}\right)\mb{y}_{k+1},\left(\mb{W} -\mb{J}\right)\left(\nf_{k+1} - \nf_{k}\right) \r \right] 
\end{align*}
\end{lem}
\begin{proof}
Using the gradient tracking update~\eqref{y}, and the fact that~$\mb{W} \mb{J} = \mb{J} \mb{W} = \mb{J}$, we have:~$\forall k\geq0$,
\begin{align}\label{gt0}
&\ln\y_{k+2} -\mb{J}\y_{k+2}\rn^2 \n\\
=&~\ln \mb{W}\left(\mb{y}_{k+1} + \g_{k+1} - \g_{k}\right) -\mb{J}\left(\mb{y}_{k+1} + \g_{k+1} - \g_{k}\right)\rn^2 \n\\
=&~\ln \mb{W}\mb{y}_{k+1}-\mb{J}\mb{y}_{k+1} + \left(\mb{W} -\mb{J}\right)\left(\g_{k+1} - \g_{k}\right)\rn^2 \n\\
=&~\ln \mb{W}\mb{y}_{k+1}-\mb{J}\mb{y}_{k+1}\rn^2 +
\ln\left(\mb{W} -\mb{J}\right)\left(\g_{k+1} - \g_{k}\right)\rn^2 \n\\
&+ 2\l \left(\mb{W}-\mb{J}\right)\mb{y}_{k+1},~\left(\mb{W} -\mb{J}\right)\left(\g_{k+1} - \g_{k}\right) \r          \n\\
\leq&~\lambda^2\ln \mb{y}_{k+1}-\mb{J}\mb{y}_{k+1}\rn^2
+\lambda^2\ln\g_{k+1} - \g_{k}\rn^2 \n\\
&+2\underbrace{\l\left(\mb{W}-\mb{J}\right)\mb{y}_{k+1},\left(\mb{W} -\mb{J}\right)\left(\g_{k+1} - \g_{k}\right) \r}_{C_1},
\end{align}
where the last inequality is due to Lemma~\ref{basic}(\ref{W}). Towards~$C_1$, since~$\y_{k+1}$ and~$\g_{k}$ are~$\F_{k+1}$-measurable, we have:~$\forall k\geq0$, 
\begin{align}\label{gt01}
&\E\left[C_1|\F_{k+1}\right] \n\\
% =&~\E\left[\E\left[\l \left(\mb{W}-\mb{J}\right)\mb{y}_{k+1},\left(\mb{W} -\mb{J}\right)\left(\g_{k+1} - \g_{k}\right) \r | \F_{k+1} \right]\right] \nonumber\\
=&\l \left(\mb{W}-\mb{J}\right)\mb{y}_{k+1},\left(\mb{W} -\mb{J}\right)\left(\nf_{k+1} - \g_{k}\right) \r \nonumber\\
=&\l \left(\mb{W}-\mb{J}\right)\mb{y}_{k+1},\left(\mb{W} -\mb{J}\right)\left(\nf_{k} - \g_{k}\right) \r  \n\\
&+\l \left(\mb{W}-\mb{J}\right)\mb{y}_{k+1},\left(\mb{W} -\mb{J}\right)\left(\nf_{k+1} - \nf_{k}\right) \r.
\end{align}
The proof then follows by taking the expectation on~\eqref{gt0} and using~\eqref{gt01} in the resulting inequality.
\end{proof}
\noindent
Next, we bound the terms in Lemma~\ref{gt} respectively.
For the second term in Lemma~\ref{gt}, we have the following. 
\begin{lem}\label{gt1}
Let Assumption~\ref{f}-\ref{o} hold. We have:~$\forall k\geq0$,
\begin{align*}
&\E\big[\ln\g_{k+1} - \g_{k}\rn^2\big] \n\\
\leq&~18L^2\E\big[\ln \x_{k} - \mb{J}\x_{k} \rn^2\big] + 6n\a_k^2L^2\E\big[\ln \agk\rn^2\big] \n \\
&+ 12\a_k^2L^2\lambda^2\E\big[\ln\y_{k+1} - \mb{J}\y_{k+1}\rn^2\big] + 3n\nu_a^2.
\end{align*}
\end{lem}
\begin{proof}
Since both~$\nf_{k+1}$ and~$\g_k$ are~$\F_{k+1}$-measurable and $\E[\g_{k+1}|\F_{k+1}] = \nf_{k+1}$, we have:~$\forall k\geq0$,
\begin{align}\label{g10}
&~\E\big[\ln\g_{k+1} - \g_{k}\rn^2\big]\n\\
=&~\E\big[\ln\g_{k+1} -\nf_{k+1}\rn^2\big] + \E\big[\ln \nf_{k+1} - \g_{k}\rn^2\big], \n\\
\leq&~n\nu_a^2 + \E\big[\ln \nf_{k+1} - \g_{k}\rn^2\big] \n\\
\leq&~n\nu_a^2 + 2\E\big[\ln \nf_{k+1} - \nf_{k}\rn^2\big]
+ 2\E\big[\ln \nf_{k} - \g_{k}\rn^2\big] \n\\
\leq&~3n\nu_a^2 + 2L^2\underbrace{\E\big[\ln \x_{k+1} - \x_{k}\rn^2\big]}_{C_2}
\end{align}
where the first inequality uses Assumption~\ref{o} and the last inequality uses Assumption~\ref{o} and the~$L$-smoothness of each~$f_i$.
Towards~$C_2$, we have:~$\forall k\geq0$,
\begin{align}\label{C_2}
C_2
=&~\E\big[\ln \x_{k+1} - \mb{J}\x_{k+1} + \mb{J}\x_{k+1} - \mb{J}\x_{k} + \mb{J}\x_{k}  - \x_{k}\rn^2\big]                       \n\\
\leq&~3\E\big[\ln \x_{k+1} - \mb{J}\x_{k+1} \rn^2\big] + 3n\a_k^2\E\big[\ln \agk\rn^2\big] \n\\
&+ 3\E\big[\ln \x_{k} - \mb{J}\x_{k} \rn^2\big]  \n\\
\leq&~9\E\big[\ln \x_{k} - \mb{J}\x_{k} \rn^2\big] + 3n\a_k^2\E\big[\ln \agk\rn^2\big] \n\\
&+ 6\a_k^2\lambda^2\E\big[\ln\y_{k+1} - \mb{J}\y_{k+1}\rn^2\big],
\end{align}
where the second inequality uses~\eqref{ave} and the last inequality uses Lemma~\ref{cons}. The proof follows by using~\eqref{C_2} in~\eqref{g10}.
\end{proof}
\noindent
For the third term in Lemma~\ref{gt}, we have the following.
\begin{lem}\label{gt2}
Let Assumption~\ref{f}-\ref{o} hold. We have:~$\forall k\geq0$,
\begin{align*}
\E\left[\l \left(\mb{W}-\mb{J}\right)\mb{y}_{k+1},\left(\mb{W} -\mb{J}\right)\left(\nf_{k} - \g_{k}\right) \r \right] \leq \nu_a^2. 
\end{align*}
\end{lem}
\begin{proof}
Using the fact that~$\mb{J}(\mb{W}-\mb{J}) =\mb{O}_{np}$ and the gradient tracking update~\eqref{y}, we have:~$\forall k\geq0$, 
\begin{align}\label{gt21}
&\E\left[\l \left(\mb{W}-\mb{J}\right)\mb{y}_{k+1},\left(\mb{W} -\mb{J}\right)\left(\nf_{k} - \g_{k}\right) \r |\F_k\right] \n\\
=&~\E\left[\l\mb{W}\mb{y}_{k+1},\left(\mb{W} -\mb{J}\right)\left(\nf_{k} - \g_{k}\right) \r |\F_k\right] \n\\
=&~\E\left[\l\mb{W}^2\left(\y_{k} + \g_{k}- \g_{k-1}\right),
\left(\mb{W} -\mb{J}\right)\left(\nf_{k} - \g_{k}\right) \r | \F_k\right] \n\\
=&~\E\left[\l\mb{W}^2\g_{k},
\left(\mb{W} -\mb{J}\right)\left(\nf_{k} - \g_{k}\right) \r | \F_k\right] \n\\
=&~\E\left[\l \mb{W}^2\left(\g_{k}-\nf_k\right),\left(\mb{W} -\mb{J}\right)\left(\nf_{k} - \g_{k}\right) \r |\F_k \right] \n\\
=&~\E\left[(\g_{k} - \nf_{k})^\top (\mb{J}-\mb{W}^\top \mb{W}^2)\left(\g_{k}-\nf_{k}\right)|\F_k\right],
\end{align}
where the third and the fourth equality exploit the fact that the random vectors $\y_k$, $\g_{k-1}$ and $\nf_k$ are~$\F_k$-measurable and that~$\E[\g_{k}|\F_k] = \nf_{k}$.
% To proceed, we recall from Lemma~\ref{indep} that
% \begin{align*}
% \E\!\left[\l \g_i(\x_k^i,\bxi_k^i) - \nabla f_i(\x_k^i),\g_j(\x_k^j,\bxi_k^j) - \nabla f_j(\x_k^j)\r|\F_k\right]\!= 0,    
% \end{align*}
% whenever~$i\neq j,\forall i,j\in\mc{V}$. 
In light of Lemma~\ref{basic}(\ref{indep}),~\eqref{gt21} reduces to
\begin{align}\label{gt22}
&\E\left[\l \left(\mb{W}-\mb{J}\right)\mb{y}_{k+1},\left(\mb{W} -\mb{J}\right)\left(\nf_{k} - \g_{k}\right) \r |\F_k\right] \n\\    
=&~\E\left[(\g_{k} - \nf_{k})^\top \mbox{diag}(\mb{J}-\mb{W}^\top \mb{W}^2)\left(\g_{k}-\nf_{k}\right)|\F_k\right] 
\n\\
% &-\E\left[(\g_{k} - \nf_{k})^\top \mbox{diag}(\mb{W}^\top \mb{W}^2)\left(\g_{k}-\nf_{k}\right)|\F_k\right] \n\\
\leq&~\E\left[(\g_{k} - \nf_{k})^\top \mbox{diag}(\mb{J})\left(\g_{k}-\nf_{k}\right)|\F_k\right], \n\\
=&~\E\left[\|\g_{k}-\nf_{k}\|^2|\F_k\right]/n
\end{align}
where the inequality holds since~$\mbox{diag}(\mb{W}^\top \mb{W}^2)$ is nonnegative. The proof follows by using Assumption~\ref{o} in~\eqref{gt22} and taking the expectation on the resulting inequality.
\end{proof}

For the last term in Lemma~\ref{gt}, we have the following.
\begin{lem}\label{gt3}
Let Assumption~\ref{f}-\ref{o} hold. We have:~$\forall k\geq0$,
\begin{align*}
&\l \left(\mb{W}-\mb{J}\right)\mb{y}_{k+1},\left(\mb{W} -\mb{J}\right)\left(\nf_{k+1} - \nf_{k}\right) \r \n \\
\leq&~(\lambda\a_k L + 0.5\eta_1  + \eta_2)\lambda^2\ln\mb{y}_{k+1}-\mb{J}\mb{y}_{k+1}\rn^2\n\\
&+ \eta_2^{-1}\lambda^2L^2\ln\x_{k}-\mb{J}\x_{k}\rn^2 
+ 0.5\eta_1^{-1}\lambda^2 \a_k^2 L^2 n\ln\agk\rn^2,
\end{align*}
where~$\eta_1$ and~$\eta_2$ are arbitrary positive constants\footnote{We note that~$\eta_1$ and~$\eta_2$ will be fixed later.}.
\end{lem}
\begin{proof}
Using~$(\mb{W}-\mb{J})\mb{J} = \mb{O}_{np}$ and the Cauchy-Schwarz inequality, we have:~$\forall k\geq0$,
\begin{align}\label{g31}
&\l \left(\mb{W}-\mb{J}\right)\mb{y}_{k+1},\left(\mb{W} -\mb{J}\right)\left(\nf_{k+1} - \nf_{k}\right) \r \n    \\
=&~\l\left(\mb{W}-\mb{J}\right)(\mb{y}_{k+1}-\mb{J}\mb{y}_{k+1}),\left(\mb{W} -\mb{J}\right)\left(\nf_{k+1} - \nf_{k}\right) \r \n\\ 
\leq&~\lambda^2L\ln\mb{y}_{k+1}-\mb{J}\mb{y}_{k+1}\rn\ln\x_{k+1}-\x_{k}\rn,
\end{align}
where the last inequality uses~$\|\mb{W}-\mb{J}\| = \lambda$ and the~$L$-smoothness of each~$f_i$. We note that,~$\forall k\geq0$,
\begin{align}\label{x_diff}
&\ln\x_{k+1}-\x_{k}\rn \n\\
=&~\ln\x_{k+1} - \mb{J}\x_{k+1} + \mb{J}\x_{k+1} - \mb{J}\x_{k} + \mb{J}\x_{k} -\x_{k}\rn \n\\
\leq&~\ln\x_{k+1}-\mb{J}\x_{k+1}\rn + \a_k\sqrt{n}\ln\agk\rn + \ln\x_{k}-\mb{J}\x_{k}\rn \n\\
\leq&~
2\ln\x_{k}-\mb{J}\x_{k}\rn + \a_k\sqrt{n}\ln\agk\rn + \a_k\lambda\ln\y_{k+1}-\mb{J}\y_{k+1}\rn.
\end{align}
where the last inequality uses Lemma~\ref{cons}. We use~\eqref{x_diff} in~\eqref{g31} to obtain:~$\forall k\geq0$,
\begin{align}\label{g311}
&\l \left(\mb{W}-\mb{J}\right)\mb{y}_{k+1},\left(\mb{W} -\mb{J}\right)\left(\nf_{k+1} - \nf_{k}\right) \r \n    \\
\leq&~\lambda^3\a_k L\ln\mb{y}_{k+1}-\mb{J}\mb{y}_{k+1}\rn^2 \n\\
&+ \underbrace{\left(\lambda \ln\mb{y}_{k+1}-\mb{J}\mb{y}_{k+1}\rn\right)\left(\lambda\a_k L\sqrt{n}\ln\agk\rn\right)}_{C_3} \n\\
&+ \underbrace{2(\lambda \ln\mb{y}_{k+1}-\mb{J}\mb{y}_{k+1}\rn)(\lambda L\ln\x_{k}-\mb{J}\x_{k}\rn)}_{C_4}.
\end{align}
By Young's inequality, we have that
\begin{align*}
C_3 \leq 0.5\eta_1\lambda^2\ln\mb{y}_{k+1}-\mb{J}\mb{y}_{k+1}\rn^2
+ 0.5\eta_1^{-1}\lambda^2 \a_k^2 L^2 n\ln\agk\rn^2,
\end{align*}
where~$\eta_1>0$ is arbitrary,
and that,
\begin{align*}
C_4 \leq \eta_2\lambda^2\ln\mb{y}_{k+1}-\mb{J}\mb{y}_{k+1}\rn^2 + \eta_2^{-1}\lambda^2L^2\ln\x_{k}-\mb{J}\x_{k}\rn^2,
\end{align*}
where~$\eta_2>0$ is arbitrary. The proof follows by Using the bounds on~$C_3$ and~$C_4$ in~\eqref{g311}.
\end{proof}
\noindent
With the help of auxiliary Lemmas~\ref{gt1}-\ref{gt3}, we now prove an upper bound on the gradient tracking error.
\begin{lem}\label{GT_final}
Let Assumption~\ref{f}-\ref{o} hold. If~$0<\a_k\leq\frac{1-\lambda^2}{24\lambda L}$, then we have:~$\forall k\geq0$,    
\begin{align*}
\E\bigg[\frac{\ln\y_{k+2} -\mb{J}\y_{k+2}\rn^2}{nL^2} \bigg] 
\leq&~\dfrac{1+\lambda^2}{2}\E\bigg[\frac{\ln\y_{k+1} -\mb{J}\y_{k+1}\rn^2}{nL^2} \bigg]  \n\\
&+\dfrac{24\lambda^2}{1-\lambda^2}\E\bigg[\frac{\ln\x_{k}-\mb{J}\x_{k}\rn^2}{n}\bigg] \nonumber\\ 
&+\frac{6\lambda^2\a_k^2}{1-\lambda^2}\E\big[\ln\ol{\nf}_{k}\rn^2\big] + \frac{6\nu_a^2}{L^2}.
\end{align*}
\end{lem}
\begin{proof}
We apply the upper bounds in Lemma~\ref{gt1},~\ref{gt2} and~\ref{gt3} to Lemma~\ref{gt} to obtain:~$\forall k\geq0,\forall\eta_1>0,\forall\eta_2>0$,
\begin{align}\label{GT1}
&\E\big[\ln\y_{k+2} -\mb{J}\y_{k+2}\rn^2 \big] \n\\
\leq&~\lambda^2(1 + 12\lambda^2\a_k^2L^2 + 2\lambda\a_k L + \eta_1 + 2\eta_2)\n\\
&\qquad\times \E\big[\ln \mb{y}_{k+1}-\mb{J}\mb{y}_{k+1}\rn^2\big] \n\\
&+ (3\lambda^2 n + 2)\nu_a^2  \n\\
&+\left(18+2\eta_2^{-1}\right)\lambda^2L^2\E\big[\ln\x_{k}-\mb{J}\x_{k}\rn^2\big] \n\\
&+ \left(6 + \eta_1^{-1}\right)\lambda^2\a_k^2L^2n\E\big[\ln\agk\rn^2\big]. 
\end{align}
We set~$\eta_1 = \frac{1-\lambda^2}{6\lambda^2}$ and~$\eta_2 = \frac{1-\lambda^2}{12\lambda^2}$ in~\eqref{GT1}. It is straightforward to verify that if~$0<\a_k\leq\frac{1-\lambda^2}{24\lambda^2 L},\forall k\geq0$, then we have:
\begin{align}\label{eta}
\lambda^2(1 + 12\lambda^2\a_k^2L^2 + 2\lambda\a_k L + \eta_1 + 2\eta_2) \leq \frac{1+\lambda^2}{2}.
\end{align}
Moreover, recall from~\eqref{g_ave} that
\begin{align}\label{GT2}
\E\big[\|\agk\|^2\big]
\leq \E\big[\|\ol{\nf}_{k}\|^2\big] + \nu_a^2/n.
\end{align}
Using~\eqref{eta},~\eqref{GT2},~$\eta_1 = \frac{1-\lambda^2}{6\lambda^2}$ and~$\eta_2 = \frac{1-\lambda^2}{12\lambda^2}$ in~\eqref{GT1}, we have: if~$0<\a_k\leq\frac{1-\lambda^2}{24\lambda^2 L}$, then
\begin{align*}
&\E\big[\|\y_{k+2} -\mb{J}\y_{k+2}\|^2 \big] \n\\
\leq&~\dfrac{1+\lambda^2}{2}\E\big[\| \mb{y}_{k+1}-\mb{J}\mb{y}_{k+1}\|^2\big] + \left( \dfrac{6\lambda^2\a_k^2L^2}{1-\lambda^2} + 5n\right)\nu_a^2  \n\\
&+\dfrac{24\lambda^2L^2}{1-\lambda^2}\E\big[\|\x_{k}-\mb{J}\x_{k}\|^2\big] + \dfrac{6\lambda^2\a_k^2L^2n}{1-\lambda^2}\E\big[\|\ol{\nf}_{k}\|^2\big].
\end{align*}
The proof follows by~$\frac{6\lambda^2\a_k^2L^2}{1-\lambda^2}\leq1$ if~$0<\a_k\leq\frac{1-\lambda^2}{24\lambda L},\forall k$.
\end{proof}

\subsection{LTI dynamics}\label{sec_ncvx_proof}
In this subsection, we establish the convergence rate of \textbf{\texttt{GT-DSGD}} for general smooth non-convex functions under an appropriate constant step-size such that~$\a_k = \a, \forall k\geq0$.
To this end, we now jointly write Lemma~\ref{cons} and~\ref{GT_final} in the following linear-time-invariant system that characterizes the convergence of consensus and gradient tracking process. 
\begin{prop}\label{LTI_ncvx}
Let Assumption~\ref{f}-\ref{o} hold. If~$0<\a\leq\frac{1-\lambda^2}{24\lambda L}$, then we have the following (entry-wise) matrix-vector inequality hold:~$\forall k\geq0$, 
\begin{align}\label{lti_ncvx}
\mb{u}_{k+1} \leq \mb{G}\mb{u}_{k} + \mb{b}_{k},    
\end{align}
where the state vector~$\mb{u}_k\in\R^2$, the system matrix~$\mb{G}\in\R^{2\times 2}$ and the perturbation vector~$\mb{b}_k\in\R^2$ are given by
\begin{align*}
&\mb{u}_k \!=\! \begin{bmatrix}
\mbb{E}\bigg[\dfrac{\left\|\mb{x}_{k}-\mb{J}\mb{x}_{k}\right\|^2}{n}\bigg]\\[2ex]
\mbb{E}\bigg[\dfrac{\left\|\mb{y}_{k+1}-\mb{J}\mb{y}_{k+1}\right\|^2}{nL^2}\bigg]
\end{bmatrix},
\mb{G} \!=\! \begin{bmatrix}
\dfrac{1+\lambda^2}{2}  & \dfrac{2\a^2\lambda^2L^2}{1-\lambda^2} \\[2ex]
\dfrac{24\lambda^2}{1-\lambda^2}  & \dfrac{1+\lambda^2}{2}
\end{bmatrix}, \n\\
&\mb{b}_k \!=\! \begin{bmatrix}
0\\
\dfrac{6\lambda^2\a^2}{1-\lambda^2}\mathbb{E}\left[\left\|\ol{\nf}_{k}\right\|^2\right] + \dfrac{6\nu_a^2}{L^2}
\end{bmatrix}.
\end{align*}
\end{prop}

% if~$0<\a_k\leq\min\big\{\frac{1-\lambda^2}{4\lambda},1\big\}\frac{1}{2L}$,
% \begin{align}\label{DD1}
% \underbrace{\left[\begin{array}{c}
% \mbb{E}\left[n^{-1}\left\|\mb{x}_{k+1}-\mb{J}\mb{x}_{k+1}\right\|^2\right]\\[2ex] 
% \mbb{E}\left[n^{-1}L^{-2}\left\|\mb{y}_{k+2}-\mb{J}\mb{y}_{k+2}\right\|^2\right]
% \end{array}\right]}_{\mb{u}_{k+1}\in\mbb{R}^2}
% \leq&
% \underbrace{\left[
% \begin{array}{ccc}
% \dfrac{1+\lambda^2}{2}  & \dfrac{2\a_k^2\lambda^2L^2}{1-\lambda^2} \\[2ex]
% \dfrac{9\lambda^2}{1-\lambda^2}  & \dfrac{1+\lambda^2}{2}
% \end{array}
% \right]}_{G\in\mbb{R}^2}
% \underbrace{\left[\begin{array}{c}
% \mbb{E}\left[n^{-1}\left\|\mb{x}_{k}-\mb{J}\mb{x}_{k}\right\|^2\right]\\ [2ex]
% \mbb{E}\left[n^{-1}L^{-2}\left\|\mb{y}_{k+1}-\mb{J}\mb{y}_{k+1}\right\|^2\right]
% \end{array}\right]}_{\mb{u}_{k}\in\mbb{R}^2} \nonumber\\
% &+\underbrace{\left[\begin{array}{c}
% 0\\
% 1.25\lambda^2L^{-2}\mathbb{E}\left[\left\|\ol{\nf}_{k}\right\|^2\right]
% + 4.75L^{-2}\nu_a^2
% \end{array}\right]}_{\mb{b}_k\in\mbb{R}^2}.
% \end{align}
% Therefore, the above LTI system can be compactly written as:
% \begin{align}\label{LTI}
% \mb{u}_{k} \leq G\mb{u}_{k-1} + \mb{b}_{k-1}, \quad\forall k\geq1.
% \end{align}
\noindent
In light of Proposition~\ref{LTI_ncvx}, we first solve the range of~$\a$ such that~$\rho(\mb G)<1$, using the following lemma from~\cite{book_matrix_analysis}.
\begin{lem}\label{rho_bound}
Let~$\mb X\in\mathbb{R}^{d\times d}$ be a non-negative matrix and~$\mb{x}\in\mathbb{R}^d$ be a positive vector. If~$\mb{X}\mb{x}<\mb{x}$, then~$\rho(\mb X)<1$. Moreover, if~$\mb{X x}\leq z\mb{x}$, for some~$z>0$,
then $\rho(\mb X)\leq z$.
\end{lem}
\begin{lem}\label{st}
If~$0<\a\leq\min\big\{\frac{1-\lambda^2}{24\lambda},\frac{(1-\lambda^2)^2}{15\lambda^2}\big\}\frac{1}{L}$, then we have $\rho(\mb G)<1$ and hence~$\sum_{k=0}^{\infty}\mb{G}^k = (\mb{I}_2-\mb{G})^{-1}$.
\end{lem}
\begin{proof}
In the light of~Lemma~\ref{rho_bound}, we solve the range of~$\a$ and a positive vector~$\mb{s} = [s_1,s_2]^\top$ such that~$\mb{Gs} < \mb s$, which is equivalent to the following two inequalities:
\[
\begin{cases}
\dfrac{1+\lambda^2}{2}s_1 + \dfrac{2\a^2\lambda^2L^2}{1-\lambda^2}s_2<s_1\\[1.5ex]
\dfrac{24\lambda^2}{1-\lambda^2}s_1 + \dfrac{1+\lambda^2}{2}s_2 < s_2
\end{cases}
\!\!\!\!\!\!\!\Longleftrightarrow\!
\begin{cases}
\a^2<\dfrac{(1-\lambda^2)^2}{4\lambda^2L^2}\dfrac{s_1}{s_2}\\[1.5ex]
\dfrac{s_1}{s_2}  < \dfrac{(1-\lambda^2)^2}{48\lambda^2}
\end{cases}
\]
% \begin{align*}
% \left\{
% \begin{array}{ll}
% \!\!\!\dfrac{1+\lambda^2}{2}s_1 + \dfrac{2\a^2\lambda^2L^2}{1-\lambda^2}s_2<s_1\\[1.5ex]
% \!\!\!\dfrac{24\lambda^2}{1-\lambda^2}s_1 + \dfrac{1+\lambda^2}{2}s_2 < s_2
% \end{array}
% \right.
% \!\!\!\!
% \iff
% \!\!
% \left\{
% \begin{array}{ll}
% \!\!\a^2<\dfrac{(1-\lambda^2)^2}{4\lambda^2L^2}\dfrac{s_1}{s_2}\\[1.5ex]
% \!\!\dfrac{s_1}{s_2}  < \dfrac{(1-\lambda^2)^2}{48\lambda^2}
% \end{array}
% \right.
% \end{align*}
We set~$s_1/s_2 = (1-\lambda^2)^2/(50\lambda^2)$ and the proof follows by using it to solve for the range of~$\a$ such that the first inequality above holds. 
\end{proof}
\noindent
Now, we prove an upper bound on the accumulated consensus errors along the algorithm path as follows.
\begin{lem}\label{consensus_acc}
Let Assumption~\ref{f}-\ref{o} hold. If~$0<\a\leq\min\big\{\frac{1-\lambda^2}{24\lambda},\frac{(1-\lambda^2)^2}{8\sqrt{6}\lambda^2}\big\}\frac{1}{L}$, then we have the following inequality.
\begin{align*}
\sum_{k=0}^K\E\bigg[\frac{\ln\x_k-\mb{J}\x_k\rn^2}{n}\bigg] 
\leq&~\frac{96\a^4\lambda^4 L^2}{(1-\lambda^2)^4}\sum_{k=0}^{K-1}\mathbb{E}\left[\left\|\ol{\nf}_{k}\right\|^2\right] \n\\
+&\frac{16\a^2\lambda^4}{(1-\lambda^2)^3}\frac{\left\|\nf_0\right\|^2}{n}  + \frac{112\a^2\lambda^2\nu_a^2K}{(1-\lambda^2)^3}.
\end{align*}
\end{lem}
\begin{proof}
We recursively apply~\eqref{lti_ncvx} to obtain:~$\forall k\geq1$,
\begin{align}\label{r_ncvx_flat}
\u_{k} 
\leq \mb{G}^{k}\u_0 + \sum_{t = 0}^{k-1}\mb{G}^t\mb{b}_{k-1-t}.
\end{align}
Summing up~\eqref{r_ncvx_flat} over~$k$ from~$1$ to~$K$, we obtain:~$\forall K\geq1$, 
\begin{align}\label{sum}
\sum_{k=0}^K\u_{k} 
\leq&~\sum_{k=0}^K \G^{k}\u_0 + \sum_{k=1}^K\sum_{t = 0}^{k-1}\G^t\mb{b}_{k-1-t} \n\\
\leq&~\left(\sum_{k=0}^{\infty}\G^{k}\right)\u_0 + \left(\sum_{k=0}^{\infty}\G^{k}\right)\sum_{k=0}^{K-1}\mb{b}_{k} \n\\
=&~(\I_2-\G)^{-1}\u_0 + (\I_2-\G)^{-1}\sum_{k=0}^{K-1}\mb{b}_{k}.
\end{align}
In light of~\eqref{sum}, we next compute an (entry-wise) upper bound on~$(\I_2-\G)^{-1}$ as follows. We note that if~$0<\a\leq\tfrac{(1-\lambda^2)^2}{8\sqrt{6}\lambda^2L}$, 
\begin{align*}
\det(\I_2-\G) = \frac{(1-\lambda^2)^2}{4} - \frac{48\a^2\lambda^4L^2}{(1-\lambda^2)^2} \geq \frac{(1-\lambda^2)^2}{8}.
\end{align*}
Using the lower bound on~$\det(\I_2-\G)$ above, we have that
\begin{align}\label{I-G}
(\I_2-\G)^{-1} = \frac{(\I_2-\G)^*}{\det(\I_2-\G)} 
\leq&~
\begin{bmatrix}
\dfrac{4}{1-\lambda^2} & \dfrac{16\a^2\lambda^2L^2}{(1-\lambda^2)^3}\\[1.5ex] 
\dfrac{192\lambda^2}{(1-\lambda^2)^3} & \dfrac{4}{1-\lambda^2}
\end{bmatrix}.
\end{align}
We use~\eqref{I-G} in~\eqref{sum} with~$\|\x_0-\J\x_0\| = 0$ to obtain:~$\forall K\geq1$,
\begin{align}\label{sum2}
\sum_{k=0}^K&~\E\bigg[\frac{\ln\x_k-\mb{J}\x_k\rn^2}{n}\bigg]
\leq\frac{16\a^2\lambda^2}{(1-\lambda^2)^3}\E\bigg[\frac{\left\|\y_1-\mb{J}\y_1\right\|^2}{n}\bigg] \nonumber\\
&+ \frac{96\a^4\lambda^4 L^2}{(1-\lambda^2)^4}\sum_{k=0}^{K-1}\mathbb{E}\left[\left\|\ol{\nf}_{k}\right\|^2\right] 
+ \frac{96\a^2\lambda^2\nu_a^2K}{(1-\lambda^2)^3}.
\end{align}
Finally, we use the gradient tracking update~\eqref{y} to obtain:
\begin{align}\label{y1}
&\E\big[\left\|\y_1-\mb{J}\y_1\right\|^2\big]\n\\
%=&~\E\big[\left\|\y_1-\mb{J}\y_1\right\|^2\big]\n\\
=&~\E\big[\E\big[\left\|(\mb{W}-\J)\g_{0}\right\|^2\big]|\F_0\big] \n\\
=&~\E\big[\left\|(\mb{W}-\J)(\g_{0}-\nf_0)\right\|^2\big]
+ \E\big[\left\|(\mb{W}-\J)\nf_0\right\|^2\big] \n\\
\leq&~\lambda^2 n\nu_a^2 + \lambda^2\ln\nf_0\rn^2,
\end{align}
where the second equality uses~$\E[\g_0|\F_0] = \nf_0$ and that~$\nf_0$ is constant and the last inequality uses~$\|\W-\J\| = \lambda$.
The proof follows by using~\eqref{y1} in~\eqref{sum2}.
\end{proof}
\noindent
Lemma~\ref{consensus_acc} states that the accumulated consensus error may be bounded by the accumulated average of local exact gradients and the accumulated variance of stochastic gradients. We next show that this bound leads to the convergence of~\DSGT~for general smooth non-convex functions, i.e., Theorem~\ref{conv_ncvx}.

\begin{TH0}
We take the expectation of the descent inequality in Lemma~\ref{ds} and sum up the resulting inequality over~$k$ from~$0$ to~$K-1$,~$\forall K\geq1$, to obtain: if~$0<\a\leq\frac{1}{2L}$,
\begin{align}\label{mean0}
\E\left[F(\ol{\x}_{K})\right] 
\leq&~\E\left[F(\ol{\x}_{0})\right] -\frac{\a}{2}\sum_{k=0}^{K-1}\E\left[\ln\nabla F(\ol{\x}_{k})\rn^2\right] \n\\
&-\frac{\a}{4}\sum_{k=0}^{K-1}\E\left[\ln\ol{\nf}_{k}\rn^2\right] 
+\frac{\a^2\nu^2_a LK}{2n} \n\\
& +\frac{\a L^2}{2}\sum_{k=0}^{K-1}\E\bigg[\frac{\ln\x_k - \mb{J}\x_k\rn^2}{n}\bigg].
\end{align}
Rearranging~\eqref{mean0} and using that~$F$ is bounded below by~$F^*$ obtains: if~$0<\a\leq\frac{1}{2L}$,~$\forall K\geq1$,
\begin{align}\label{mean}
&\sum_{k=0}^{K-1}\E\left[\ln\nabla F(\ol{\x}_{k})\rn^2\right] 
\leq\frac{2(F(\ol{\x}_{0}) - F^*)}{\a}  +\frac{\a\nu^2_a LK}{n} \n\\
&~~~-\frac{1}{2}\sum_{k=0}^{K-1}\E\left[\ln\ol{\nf}_{k}\rn^2\right] 
+ L^2\sum_{k=0}^{K-1}\E\bigg[\frac{\ln\x_k - \mb{J}\x_k\rn^2}{n}\bigg].
\end{align}
Moreover, we observe:~$\forall K\geq1$,
\begin{align*}
&\frac{1}{n}\sum_{i=1}^n\sum_{k=0}^{K-1}\E\left[\ln\nabla F(\x_k^i)\rn^2\right]    \n\\
\leq&~\frac{2}{n}\sum_{i=1}^n\sum_{k=0}^{K-1}\left(\E\left[\ln\nabla F(\x_k^i) - \nabla F(\ol{\x}_k)\rn^2 + \ln\nabla F(\ol{\x}_k)\rn^2\right] \right)\n\\
%\n\\
%&+2\sum_{k=0}^{K-1}\E\left[\ln\nabla F(\ol{\x}_k)\rn^2\right] \n\\
\leq&~2L^2\sum_{k=0}^{K-1}\E\bigg[\frac{\ln\x_k - \J\x_k\rn^2}{n}\bigg]
+2\sum_{k=0}^{K-1}\E\left[\ln\nabla F(\ol{\x}_k)\rn^2\right],
\end{align*}
where the last inequality uses the~$L$-smoothness of~$F$.
Using~\eqref{mean} in the inequality above obtains:~$\forall K\geq1$,
\begin{align}\label{i}
\frac{1}{n}\sum_{i=1}^n\sum_{k=0}^{K-1}\E\left[\ln\nabla F(\x_{k}^i)\rn^2\right] \leq&~\frac{4(F(\ol{\x}_{0}) - F^*)}{\a}  +\frac{2\a\nu^2_a LK}{n} \n\\
- \sum_{k=0}^{K-1}\E\left[\ln\ol{\nf}_{k}\rn^2\right] 
&+ 4L^2\sum_{k=0}^{K-1}\E\bigg[\frac{\ln\x_k - \mb{J}\x_k\rn^2}{n}\bigg].
\end{align}
We finally apply the upper bound derived in Lemma~\ref{consensus_acc} on the term of~\eqref{i} to obtain: If~$0<\a\leq\min\big\{\frac{1}{2},\frac{1-\lambda^2}{24\lambda},\tfrac{(1-\lambda^2)^2}{8\sqrt{6}\lambda^2}\big\}\frac{1}{L}$,
\begin{align*}
&\frac{1}{n}\sum_{i=1}^n\sum_{k=0}^{K-1}\E\left[\ln\nabla F(\x_k^i)\rn^2\right] \n\\
\leq&~\frac{4(F(\ol{\x}_{0}) - F^*)}{\a}  +\frac{2\a\nu^2_a LK}{n} +\frac{448\a^2L^2\lambda^2\nu_a^2K}{(1-\lambda^2)^3} \n\\
-& \Big(1-\frac{384\a^4L^4\lambda^4}{(1-\lambda^2)^4}\Big)\sum_{k=0}^{K-1}\E\left[\ln\ol{\nf}_{k}\rn^2\right] +\frac{64\a^2L^2\lambda^4}{(1-\lambda^2)^3}\frac{\left\|\nf_0\right\|^2}{n}.
\end{align*}
Clearly, if~$0<\a\leq\frac{1-\lambda^2}{5L\lambda}$, then~$1-\frac{384\a^4L^4\lambda^4}{(1-\lambda^2)^4}\geq0$, and the proof follows by dropping the negative term.
\end{TH0}

\section{Convergence analysis under PL condition: constant step-size}\label{S_PL_ms_cst}
In this section, we, built on top of the results established in Section~\ref{S_ncvx}, develop general bounds on the iterates of \textbf{\texttt{GT-DSGD}} when the global function~$F$ further satisfies the PL condition and prove Theorem~\ref{conv_PL}. The following is a useful inequality that may be found in~\cite{book_polyak}.
\begin{lem}\label{upperL}
Let Assumption~\ref{f} hold. We have:~$\forall \x\in\R^p$.
\begin{align*}
\ln\nabla F(\x)\rn^2 \leq 2L\left(F(\x)-F^*\right).    
\end{align*}
\end{lem}
\begin{proof}
By~\eqref{L_lemma} and the fact that~$F$ is bounded below by~$F^*$, we have~$F^* 
\leq F\left(\x - L^{-1}\nabla F(\x)\right) 
\leq F(\x) - \frac{1}{2L}\ln\nabla F(\x)\rn^2$, which yields the desired inequality.
\end{proof}
\noindent
We conclude from~Lemma~\ref{upperL} that, under Assumption~\ref{f} and~\ref{PL}, $\mu\leq L$ and recall $\kappa := \frac{L}{\mu}\geq1$. 
The following lemma is helpful in establishing the performance of~\DSGT~at each node.  
\begin{lem}\label{F_ave}
Let Assumption~\ref{f} hold. We have
\begin{align*}
\frac{1}{n}\sum_{i=1}^n(F(\x_k^i) - F^*) 
\leq&~2\left(F(\ol{\x}_k) - F^*\right)  + L\frac{\ln\x_k-\J\x_k\rn^2}{n}.
\end{align*}
\end{lem}
\begin{proof}
Setting~$\y = \x_k^i$ and~$\x = \ol{\x}_k$ in~\eqref{L_lemma}, we obtain
\begin{align}\label{F_ave_0}
&F(\x_k^i) - F^* \n\\
\leq&~F(\ol{\x}_k) - F^* + \l \nabla F(\ol{\x}_k), \x_k^i - \ol{\x}_k \r + \tfrac{1}{2}L\ln\x_k^i-\ol{\x}_k\rn^2, \n\\
\leq&~F(\ol{\x}_k) - F^* + \ln\nabla F(\ol{\x}_k)\rn \ln\x_k^i - \ol{\x}_k\rn + \tfrac{1}{2}L\ln\x_k^i-\ol{\x}_k\rn^2, \n\\
\leq&~F(\ol{\x}_k) - F^* + \tfrac{1}{2}L^{-1}\ln\nabla F(\ol{\x}_k)\rn^2 + L\ln\x_k^i-\ol{\x}_k\rn^2 \n\\
\leq&~2\left(F(\ol{\x}_k) - F^*\right)  + L\ln\x_k^i-\ol{\x}_k\rn^2,
\end{align}
where the third inequality uses Young's inequality and the last inequality is due to Lemma~\ref{upperL}. Averaging~\eqref{F_ave_0} over~$i$ from~$1$ to~$n$ proves the lemma.
\end{proof}
\noindent
In the following, we refine several results developed in Section~\ref{S_ncvx}. We first use the PL inequality to in Lemma~\ref{ds}. 
\begin{lem}\label{ds_PL}
Let~Assumptions~\ref{f}-\ref{PL} hold. If~$0<\a_k\leq\frac{1}{2L}$, then we have:~$\forall k\geq0$,
\begin{align*}
\E\left[\frac{F(\ol{\x}_{k+1})- F^*}{L}\Big|\F_k\right] 
\leq&~(1 -\mu\a_k)\frac{F(\ol{\x}_{k})- F^*}{L} \n\\
&+ \frac{\a_k L}{2}\frac{\ln\x_k - \mb{J}\x_k\rn^2}{n} +\frac{\a_k^2\nu^2_a}{2n}.
\end{align*}
\end{lem}
\begin{proof}
% The proof follows by using the PL condition in the descent inequality in Lemma~\ref{ds}: if~$0<\a_k\leq\frac{1}{2L}$,
% \begin{align*}
% \E\left[F(\ol{\x}_{k+1})|\F_k\right] 
% \leq&~F(\ol{\x}_{k}) -\mu\a_k\left(F(\ol{\x}_k)-F^*\right) - \frac{\a_k}{4}\ln\ol{\nf}_{k}\rn^2
% \n\\
% &+ \frac{\a_k L^2}{2}\frac{\ln\x_k - \mb{J}\x_k\rn^2}{n} +\frac{\a_k^2 L\nu^2_a}{2n}.
% \end{align*}
The proof follows by using the PL condition in the descent inequality in Lemma~\ref{ds}
and then substracting~$F^*$ from both sides of the resulting inequality.
\end{proof}
\noindent
We next use Lemma~\ref{upperL} to refine Lemma~\ref{GT_final} as follows.
\begin{lem}\label{GT_final_PL}
Let Assumption~\ref{f}-\ref{o} hold. If~$0<\a_k\leq\min\big\{\frac{1-\lambda^2}{12\lambda},1\big\}\frac{1}{2L}$, then we have:~$\forall k\geq0$,
\begin{align*}
\E\bigg[\frac{\ln\y_{k+2} -\mb{J}\y_{k+2}\rn^2}{nL^2} \bigg] 
\leq&~\dfrac{1+\lambda^2}{2}\E\bigg[\frac{\ln\y_{k+1} -\mb{J}\y_{k+1}\rn^2}{nL^2} \bigg]  \n\\
+&\frac{24\lambda^2\a_k^2L^2}{1-\lambda^2}\E\bigg[\frac{F(\ol{\x}_k) - F^*}{L}\bigg] \\
+&\dfrac{27\lambda^2}{1-\lambda^2}\E\bigg[\frac{\ln\x_{k}-\mb{J}\x_{k}\rn^2}{n}\bigg] + \frac{6\nu_a^2}{L^2}.
\end{align*}
\end{lem}
\begin{proof}
By Lemma~\ref{basic}(\ref{Lbound}) and Lemma~\ref{upperL}, we have:~$\forall k\geq0$,
\begin{align}
\ln\ol{\nf}_{k}\rn^2 
\leq&\!~2\ln\nabla F(\ol{\x}_k)\rn^2 + 2\ln\nabla F(\ol{\x}_k) - \ol{\nf}_{k}\rn^2 \n\\
\leq&\!~4L\left(F(\ol{\x}_k) - F^*\right) + 2L^2 n^{-1}\ln\x_k -\J\x_k\rn^2.     
\end{align}
Using the inequality above in Lemma~\ref{GT_final} to obtain:~$\forall k\geq0$,
\begin{align*}
&~~~\E\bigg[\frac{\ln\y_{k+2} -\mb{J}\y_{k+2}\rn^2}{nL^2} \bigg] \n\\
&\leq\left(\dfrac{24\lambda^2}{1-\lambda^2} + \frac{12\lambda^2\a_k^2L^2}{1-\lambda^2} \right)\E\bigg[\frac{\ln\x_{k}-\mb{J}\x_{k}\rn^2}{n}\bigg] + \frac{6\nu_a^2}{L^2}  \n\\
&+ \frac{24\lambda^2\a_k^2L}{1-\lambda^2}\E\big[F(\ol{\x}_k) \!-\! F^*\big]
+\dfrac{1+\lambda^2}{2}\E\bigg[\frac{\ln\y_{k+1} \!-\! \mb{J}\y_{k+1}\rn^2}{nL^2}\bigg].
\end{align*}
The proof follows by~$\frac{12\lambda^2\a_k^2L^2}{1-\lambda^2} \leq \frac{3\lambda^2}{1-\lambda^2}$ if~$0<\a_k\leq\frac{1}{2L}$.
\end{proof}
\noindent

We now write the inequalities in Lemma~\ref{cons},~\ref{ds_PL} and~\ref{GT_final_PL} jointly in a linear dynamics as follows.
\begin{prop}\label{LTI}
Let Assumption~\ref{f}-\ref{PL} hold. If~$0<\a_k\leq\min\big\{1,\frac{1-\lambda^2}{12\lambda},\tfrac{(1-\lambda^2)^2}{4\sqrt{6}\lambda^2}\big\}\frac{1}{2L}$, then we have the following (entry-wise) matrix-vector inequality:~$\forall k\geq0$, 
\begin{align}\label{lti}
\mb{v}_{k+1} \leq \mb{H}_k\mb{v}_{k} + \mb{u}_k,    
\end{align}
where the state vector~$\mb{v}_k\in\R^3$, the system matrix~$\mb{H}\in\R^{3\times 3}$ and the perturbation vector~$\mb{u}_k\in\R^3$ are given by
\begin{align*}
\mb{v}_k =&~\begin{bmatrix}
\mbb{E}\bigg[\dfrac{\left\|\mb{x}_{k}-\mb{J}\mb{x}_{k}\right\|^2}{n}\bigg] \\[2ex]
\mbb{E}\bigg[\dfrac{F(\ol{\x}_k) - F^*}{L}\bigg] \\[2ex]
\mbb{E}\bigg[\dfrac{\left\|\mb{y}_{k+1}-\mb{J}\mb{y}_{k+1}\right\|^2}{nL^2}\bigg]
\end{bmatrix},\quad
\mb{u}_k = \begin{bmatrix}
0\\[2ex]
\dfrac{\a_k^2\nu_a^2}{2n}\\[2ex]
\dfrac{6\nu_a^2}{L^2}
\end{bmatrix},
\\     
\mb{H}_k =&~
\begin{bmatrix}
\dfrac{1+\lambda^2}{2} &  0  & \dfrac{2\a_k^2\lambda^2L^2}{1-\lambda^2} \\[2ex]
\dfrac{\a_k L}{2} &  1-\mu\a_k  & 0 \\[2ex]
\dfrac{27\lambda^2}{1-\lambda^2} & \dfrac{24\lambda^2\a_k^2L^2}{1-\lambda^2}  & \dfrac{1+\lambda^2}{2}
\end{bmatrix}
.
\end{align*}
\end{prop}
\noindent
In the following lemma, we find the range of the step-size~$\a_k$ such that~$\rho(\mb{H}_k)<1, \forall k\geq0$, with the help of Lemma~\ref{rho_bound}. 
\begin{lem}\label{rho_3}
Let Assumption~\ref{f}-\ref{PL} hold. If the step-size sequence~$\a_k$ satisfies for all~$k$ that
\begin{align}\label{a_bar}
0< \a_k\leq \ol{\a} := \min\left\{\frac{1}{2L},\frac{(1-\lambda^2)^2}{42\lambda^2 L}, \frac{1-\lambda^2}{24\lambda L\kappa^{1/4}},\frac{1-\lambda^2}{2\mu}\right\},  
\end{align}
then we have:~$\rho(\mb H_k)\leq 1-\frac{\mu\a_k}{2} < 1$,~$\forall k\geq0$. 
\end{lem}
\begin{proof}
In the light of Lemma~\ref{rho_bound}, we solve for the range of the step-size~$\a_k$ and a positive vector~$\bds\delta = [\delta_1,\delta_2,\delta_3]$ such that~$\H_k\bds\delta\leq\left(1-\frac{\mu\a_k}{2}\right)\bds\delta$, which may be written as
\begin{align}
&\frac{\mu\a_k}{2} + \frac{2\a_k^2\lambda^2L^2}{1-\lambda^2}\frac{\delta_3}{\delta_1} \leq~\frac{1-\lambda^2}{2}, 
\label{v1'}\\
&\kappa\delta_1 \leq~\delta_2,
\label{v2'}\\
&\frac{\mu\a_k}{2}\leq~\frac{1-\lambda^2}{2}-\frac{27\lambda^2}{1-\lambda^2}\frac{\delta_1}{\delta_3} 
-\dfrac{24\lambda^2\a_k^2L^2}{1-\lambda^2}\frac{\delta_2}{\delta_3}.
\label{v3'}
\end{align}
According to~\eqref{v2'}, we fix~$\delta_1 = 1$ and~$\delta_2 = \kappa$. We now impose that~$0<\a_k\leq\tfrac{1-\lambda^2}{2\mu}, \forall k\geq0$. Then, according to~\eqref{v3'}, we choose~$\delta_3>0$ such that~$\frac{27\lambda^2}{1-\lambda^2}\frac{1}{\delta_3} 
+\frac{24\lambda^2\a_k^2L^2}{1-\lambda^2}\frac{\kappa}{\delta_3} \leq \frac{1-\lambda^2}{4}.$
It suffices to fix~$\delta_3 = \frac{108\lambda^2}{(1-\lambda^2)^2} + \frac{96\lambda^2\a_k^2L^2\kappa}{(1-\lambda^2)^2}$. Now, we use the fixed values of~$\delta_1,\delta_2,\delta_3$
and the requirement that~$0<\a_k\leq\tfrac{1-\lambda^2}{2\mu}$ to solve the range of~$\a_k$ such that~\eqref{v1'} holds, i.e., 
\begin{align*}
\frac{216\a_k^2\lambda^4L^2}{(1-\lambda^2)^3}
+ \frac{192\a_k^4\lambda^4L^4\kappa}{(1-\lambda^2)^3}
\leq~\frac{1-\lambda^2}{4}.    
\end{align*}
It therefore suffices to choose~$\a_k$ such that~
\begin{align*}
0< \a_k\leq\min\left\{\frac{1-\lambda^2}{6\lambda L\kappa^{1/4}},\frac{(1-\lambda^2)^2}{42\lambda^2 L}\right\}.
\end{align*}
Summarizing the obtained upper bounds on~$\a_k$ in the discussion completes the proof.
\end{proof}
We note that~$\ol{\a}$ defined in~\eqref{a_bar} is the same as the one given in Theorem~\ref{conv_PL}. The following lemma drives upper bounds on several important quantities.
\begin{lem}\label{I-H}
Let Assumption~\ref{f}-\ref{PL} hold. If~$0< \a_k\leq\ol{\a}$, where~$\ol{\a}$ is given in~\eqref{a_bar}, then we have:~$\forall k\geq0$,
\begin{align*}
\big[\left(\I_3-\H_k\right)^{-1}\mb{u}_k\big]_1 
\leq&~\frac{288\lambda^4\a_k^5L^3\kappa\nu_a^2}{n(1-\lambda^2)^4} + \frac{144\a_k^2\lambda^2\nu_a^2}{(1-\lambda^2)^3},
\n\\
\big[\left(\I_3-\H_k\right)^{-1}\mb{u}_k\big]_2 
\leq&~\frac{3\a_k\nu_a^2}{2\mu n} + \frac{72\lambda^2\a_k^2\kappa\nu_a^2}{(1-\lambda^2)^3}. %\n\\
% \big[\left(\I_3-\H_k\right)^{-1}\mb{u}_k\big]_3 
% \leq&~\frac{72\lambda^2\a_k^3L\kappa\nu_a^2}{n(1-\lambda^2)^2} + \frac{36\nu_a^2}{L^2(1-\lambda^2)}.
\end{align*}
% \begin{align*}
% \left[\big(\I_3-\H_k\right)^{-1}\big]_{1,2}
% &\leq \dfrac{1152\a_k^3 L^3\kappa\lambda^4}{(1-\lambda^2)^4},
% \quad
% \big[\left(\I_3-\H_k\right)^{-1}\big]_{1,3}
% \leq \dfrac{24\a_k^2\lambda^2L^2}{(1-\lambda^2)^3}, \nonumber\\
% \big[\left(\I_3-\H_k\right)^{-1}\big]_{2,2}
% &\leq\frac{3}{\mu\a_k}, \qquad \qquad \quad~
% \big[\left(\I_3-\H_k\right)^{-1}\big]_{2,3}
% \leq\dfrac{12\lambda^2\a_k^2 L^2\kappa}{(1-\lambda^2)^3}, \n\\
% \big[\left(\I_3-\H_k\right)^{-1}\big]_{2,3}
% &\leq\frac{288\lambda^2\a_k\kappa L}{(1-\lambda^2)^2}, \qquad~
% \big[\left(\I_3-\H_k\right)^{-1}\big]_{3,3}
% \leq\dfrac{6}{1-\lambda^2}.
% \end{align*}
\end{lem}
\begin{proof}
By the definition of~$\H_k$ in Proposition~\ref{LTI},  
% \begin{align*}
% \I_3 - \mb{H}_k = 
% \begin{bmatrix}
% \dfrac{1-\lambda^2}{2} &  0  & -\dfrac{2\a_k^2\lambda^2L^2}{1-\lambda^2} \\[2ex]
% -\dfrac{\a_k L}{2} &  \mu\a_k  & 0 \\[2ex]
% -\dfrac{27\lambda^2}{1-\lambda^2} & -\dfrac{24\lambda^2\a_k^2L^2}{1-\lambda^2}  & \dfrac{1-\lambda^2}{2}
% \end{bmatrix},
% \end{align*}
we first compute the determinant of~$(\I_3 - \mb{H}_k)$:~$\forall k\geq0$,
\begin{align*}
\det(\I_3 - \mb{H}_k) 
=&~\frac{\mu\a_k(1-\lambda^2)^2}{4}
-\frac{24\a_k^5L^5\lambda^4}{(1-\lambda^2)^2}
-\frac{54\mu\a_k^3L^2\lambda^4}{(1-\lambda^2)^2}\n\\
\geq&~\frac{\mu\a_k(1-\lambda^2)^2}{12}.
\end{align*}
if~$0< \a_k\leq\ol{\a}$, where~$\ol{\a}$ is given in~\eqref{a_bar}.
Moreover, the adjugate of~$\I_3-\H_k$, denoted as~$\ul{\mb H}^*$, is given by
\begin{align*}
\left[\ul{\mb H}^*\right]_{1,2}
&= \dfrac{48\lambda^4\a_k^4L^4}{(1-\lambda^2)^2},\qquad
\left[\ul{\mb H}^*\right]_{1,3}
= \dfrac{2\mu\a_k^3\lambda^2L^2}{1-\lambda^2}, \nonumber\\
\left[\ul{\mb H}^*\right]_{2,2}
&\leq\frac{(1-\lambda^2)^2}{4},\qquad
\left[\ul{\mb H}^*\right]_{2,3}
=\dfrac{\a_k^3L^3\lambda^2}{1-\lambda^2}. 
%\n \\
% \left[\ul{\mb H}^*\right]_{3,2}
% &=12\lambda^2\a_k^2L^2,\qquad
% \left[\ul{\mb H}^*\right]_{3,3}
% =\frac{\mu\a_k(1-\lambda^2)}{2}.
\end{align*}
The proof follows by~$(\I_3-\H_k)^{-1} = \ul{\mb H}^*/\det\left(\I_3-\H_k\right)$ and the definition of~$\mb{u}_k$ given in Proposition~\ref{LTI}.
\end{proof}
\noindent
% We present an additional useful lemma from~\cite{matrix_analysis}.
% \begin{lem}\label{matrix_conv}
% Let~$\mb{X}\in\R^{n \times n}$ and~$\epsilon>0$ be an arbitrary constant. There exists a constant~$C(\mb{X},\epsilon)$ such that~$|[\mb{X}^k]_{i,j}|\leq C(\rho(\mb{X})+\epsilon)^k, \forall k\geq0, \forall i,j\in\{1,\cdots,n\}$. 
% \end{lem}
\noindent
We are now ready to prove Theorem~\ref{conv_PL} that characterizes the performance of~\DSGT~under a constant step-size.
\begin{TH1}
We consider a constant step-size such that~$\a_k = \a, \forall k\geq0$,  with~$0<\a\leq\ol{\a}$ where~$\ol{\a}$ is given in~\eqref{a_bar}. We denote~$\H_k := \H$ and~$\mb{u}_k := \mb{u}, \forall k\geq0$,
and recursively apply~\eqref{lti} from~$k$ to~$1$ to obtain:~$\forall k\geq1$,
\begin{align}\label{PL_cst}
\mb{v}_k 
\leq\H^k\mb{v}_{0} + \sum_{t=0}^{k-1}\H^t\mb{u}    
\leq\H^k\mb{v}_{0} + \left(\I_3-\H\right)^{-1}\mb{u}.
\end{align}
It is then clear that the first two statements in Theorem~\ref{conv_PL} follow by using Lemma~\ref{rho_3} and~\ref{I-H} in~\eqref{PL_cst} and the third statement in Theorem~\ref{conv_PL} follows by Lemma~\ref{F_ave}.
\end{TH1}

\section{Convergence analysis under PL condition: almost sure convergence}\label{S_PL_as}
In this section, we prove Theorem~\ref{PL_as}, i.e., the almost sure sublinear convergence rates of \textbf{\texttt{GT-DSGD}} when the global function satisfies the PL condition under a family of stochastic approximation step-sizes. 
We first establish a key fact that under appropriate step-sizes, the stochastic gradient tracking errors are uniformly bounded in mean squared across all iterations. This fact will also be used in Section~\ref{S_PL_ms_decay}.
\begin{lem}\label{Y_bounded}
Let Assumptions~\ref{f}-\ref{PL} hold. If ${0<\a_k\leq\ol{\a}}$, for~$\ol{\a}$ given in~\eqref{a_bar},
then we have:~$\sup_{k\geq0}\mbb{E}\big[\left\|\mb{y}_{k}-\mb{J}\mb{y}_{k}\right\|^2\big]
\leq\wh{y},$ where~$\wh{y}$ is a positive constant given by
\begin{align}\label{y_hat}
\wh{y} :=&~
\frac{30\lambda^2\ol{\a}^3L^3\kappa\nu_a^2}{(1-\lambda^2)^2} 
+ \frac{60n\lambda^2\ol{\a}^2L^3(F(\ol{\x}_0)-F^*)}{(1-\lambda^2)^2} \n\\
&+  \frac{16n\nu_a^2}{1-\lambda^2} + \lambda^2\|\nf_0\|^2.      
\end{align}
\end{lem}
\begin{proof}
We prove by mathematical induction that for the state vector~$\mb{v}_k$ defined in Proposition~\ref{LTI}, there exists some positive constant vector~$\wh{\mb v}= [\wh{v}_1,\wh{ v}_2,\wh{ v}_3]^\top$ such that
\begin{align}\label{induction_ybound}
\mb{v}_k \leq \wh{\mb v}, \qquad\forall k\geq0.  
\end{align}
if~$0< \a_k\leq\ol{\a}$, where~$\ol{\a}$ is given in~\eqref{a_bar}.
We first note that in order to make~\eqref{induction_ybound} hold when~$k = 0$, according to the definition of~$\mb{v}_0$ and~\eqref{y1}, it suffices to choose~$\wh{\mb v}$ such that
\begin{align}\label{v_hat_0}
\wh{\mb v}^\top \geq \left[0, \frac{F(\ol{\x}_0) - F^*}{L}, \frac{\lambda^2\nu_a^2}{L^2} + \frac{\lambda^2\|\nf_0\|^2}{nL^2}
\right].    
\end{align}
Next, we show that if~$\mb{v}_k \leq \wh{\mb v}$ for some~$k\geq0$ and then we also have~$\mb{v}_{k+1} \leq \wh{\mb v}$ with an appropriate choice of~$\wh{\mb{v}}$. In light of Proposition~\ref{LTI}, we have~$\mb{v}_{k+1} \leq \H_k\mb{v}_k + \mb{u}_k\leq\H_k\wh{\mb v} + \mb{u}_k$, and hence it suffices to choose~$\wh{\mb v}$ such that~$\H_k\wh{\mb v} + \mb{u}_k \leq\wh{\mb v},\forall k$, which is equivalent to the following set of inequalities:
\begin{align}
&\frac{2\a_k^2\lambda^2L^2}{1-\lambda^2}\wh{v}_3 \leq \frac{1-\lambda^2}{2}\wh{v}_1, \label{v_sol_10} \\
&\frac{\kappa}{2}\wh{v}_1 + \frac{\a_k\nu_a^2}{2\mu n} \leq \wh{v}_2, \label{v_sol_20}\\
&\frac{27\lambda^2}{1-\lambda^2}\wh{v}_1 + \frac{24\lambda^2\a_k^2L^2}{1-\lambda^2}\wh{v}_2 + \frac{6\nu_a^2}{L^2}
\leq \frac{1-\lambda^2}{2}\wh{v}_3,
\label{v_sol_30}
\end{align}
where~$0<\a_k\leq\ol{\a}$ and~$\kappa = L/\mu$. 
First, we note that to make~\eqref{v_sol_10} hold, it suffices to choose~$\wh{v}_1$ as
\begin{align}\label{v_1}
\wh{v}_1 = \frac{4\ol{\a}^2\lambda^2L^2}{(1-\lambda^2)^2}\wh{v}_3.    
\end{align}
Second, based on~\eqref{v_hat_0},~\eqref{v_sol_20}, and~\eqref{v_1}, we choose~$\wh{v}_2$ as
\begin{align}\label{v_2}
\wh{v}_2 = \frac{2\ol{\a}^2\lambda^2L^2\kappa}{(1-\lambda^2)^2}\wh{v}_3 + \frac{\ol{\a}\nu_a^2}{2\mu n} + \frac{F(\ol{\x}_0)-F^*}{L}.    
\end{align}
Third, to make~\eqref{v_sol_30} hold, it suffices to choose~$\wh{v}_3$ such that 
\begin{align}\label{v_sol_31}
\wh{v}_3\geq\frac{54\lambda^2}{(1-\lambda^2)^2}\wh{v}_1 + \frac{48\lambda^2\ol{\a}^2L^2}{(1-\lambda^2)^2}\wh{v}_2 + \frac{12\nu_a^2}{L^2(1-\lambda^2)},    
\end{align}
which, using~\eqref{v_1} and~\eqref{v_2}, is equivalent to
\begin{align}\label{v_sol_3f}
\wh{v}_3\geq&~
\frac{216\ol{\a}^2\lambda^4L^2}{(1-\lambda^2)^4}\wh{v}_3  +\frac{96\lambda^4\ol{\a}^4L^4\kappa}{(1-\lambda^2)^4}\wh{v}_3 
+\frac{24\lambda^2\ol{\a}^3L\kappa\nu_a^2}{n(1-\lambda^2)^2} \n\\
&+ \frac{48\lambda^2\ol{\a}^2L(F(\ol{\x}_0)-F^*)}{(1-\lambda^2)^2} +  \frac{12\nu_a^2}{L^2(1-\lambda^2)}.    
\end{align}
By the definition of~$\ol{\a}$ in~\eqref{a_bar}, we have~$\frac{216\ol{\a}^2\lambda^4L^2}{(1-\lambda^2)^4}\leq\frac{6}{49}$ and that $\frac{96\ol{\a}^4\lambda^4L^4\kappa}{(1-\lambda^2)^4}\leq\frac{1}{3456}$; therefore, to make~\eqref{v_sol_3f} hold, it suffices to choose~$\wh{v}_3$ such that
\begin{align*}
\wh{v}_3\geq
\frac{30\lambda^2\ol{\a}^3L\kappa\nu_a^2}{n(1-\lambda^2)^2} 
+\! \frac{60\lambda^2\ol{\a}^2L(F(\ol{\x}_0)-F^*)}{(1-\lambda^2)^2} + \! \frac{15\nu_a^2}{L^2(1-\lambda^2)}.      
\end{align*}
Based on the above inequality and~\eqref{v_hat_0}, we choose~$\wh{v}_3$ as
\begin{align*}
\wh{v}_3 =&~
\frac{30\lambda^2\ol{\a}^3L\kappa\nu_a^2}{n(1-\lambda^2)^2} 
+ \frac{60\lambda^2\ol{\a}^2L(F(\ol{\x}_0)-F^*)}{(1-\lambda^2)^2} \n\\
&+  \frac{16\nu_a^2}{L^2(1-\lambda^2)} + \frac{\lambda^2\|\nf_0\|^2}{nL^2}.      
\end{align*}
The induction is complete and the proof then follows by the definition of~$\mb{v}_k$ in Proposition~\ref{LTI}.
\end{proof}

\noindent
We prove Theorem~\ref{PL_as} using the Robbins-Siegmund almost supermartingale convergence theorem~\cite{RS_theorem}, presented as follows. 

\begin{lem}[\textbf{Robbins-Siegmund}]\label{RS}
Let~$(\Omega,\F,\{\F_k\},\P)$ be a filtered space. Suppose that~$Z_k$,~$B_k$,~$C_k$ and~$D_k$ are nonnegative and~$\F_k$-measurable random variables such that
\begin{align*}
\E\left[Z_{k+1}|\F_k\right] \leq \left( 1 + B_k\right)Z_k + C_k - D_k,~\forall k\geq0.
\end{align*}
Then on the event $\left\{\sum_{k=0}^{\infty}B_k<\infty,~\sum_{k=0}^{\infty}C_k<\infty\right\}$, we have that~$\lim_{k\ra\infty}Z_k$ exists and is finite almost surely, and that $\sum_{k=0}^\infty D_k<\infty$ almost surely.
\end{lem}
We are now ready to present the proof of Theorem~\ref{PL_as}, where we construct appropriate almost supermartingales that characterize the sample path-wise convergence rate of \textbf{\texttt{GT-DSGD}} under a family of stochastic approximation step-sizes.  
\begin{TH2}
We consider the step-size sequence~$\{\a_k\}$ of the following form:~$\forall k\geq0$,
\begin{align}\label{ak_as}
\a_k = \delta(k+\varphi)^{-\epsilon},\quad
\text{where}~\delta\geq1/\mu~\text{and}~\epsilon\in(0.5,1],
\end{align}
such that~$\varphi\geq\max\{(\delta/\ol{\a})^{1/\epsilon},\frac{4}{1-\lambda^2}\}$. Hence,~$0<\a_k\leq\ol{\a}$ for $\ol{\a}$ given in~\eqref{a_bar}. We construct~$\F_k$-adapted processes:~$\forall k\geq0$,
\begin{align*}
&R_k := (k+\vp)^\tau\wt{x}_k := (k+\vp)^\tau n^{-1}\ln\x_k - \mb{J}\x_k\rn^2, \\
&Q_k := (k+\vp)^\tau\Delta_k := (k+\vp)^\tau L^{-1}(F(\ol{\x}_k)-F^*), \end{align*}
where~$\tau = 2\epsilon-1-\epsilon_1$, where~$\epsilon_1\in(0,2\epsilon-1)$ is an arbitrarily small constant. By~$1+x\leq e^x,\forall x\in\R$, we have~$(k+\vp+1)^\tau = (k+\vp)^\tau\big(1+\tfrac{1}{k+\vp}\big)^\tau \leq (k+\vp)^\tau e^{\frac{\tau}{k+\vp}}.$ 
Since~$0<\frac{\tau}{k+\vp}\leq1$, we have:~$\forall k\geq0$,
\begin{align}\label{b1}
(k+\vp+1)^\tau \leq e(k+\vp)^\tau.
\end{align}
Further, by~$e^x\leq 1+x+x^2$ for $0\leq x\leq1$,\footnote{Note that~$e^x = 1 + x + x^2\sum_{k=2}^\infty\frac{x^{k-2}}{k!},\forall x\in\R$. If~$0\leq x\leq1$, then we have~$e^x \leq 1 + x + x^2\sum_{k=2}^\infty\frac{1}{k!} = 1 + x + (e-2)x^2\leq 1 + x + x^2.$}  we have:~$\forall k\geq0$,
\begin{align}\label{b2}
(k+\vp+1)^\tau 
\leq 
\bigg(1 + \frac{\tau}{k+\vp}+\frac{\tau^2}{(k+\vp)^2}\bigg)(k+\vp)^\tau.  
\end{align}
\textbf{Recursion of~$R_k$.} We use Lemma~\ref{Y_bounded} in Lemma~\ref{cons} with the definition of~$\a_k$ in~\eqref{ak_as} to obtain:~$\forall k\geq0$,
\begin{align}\label{as_cons_rate_00}
\E\left[\wt{x}_{k+1}\right] 
\leq&~\dfrac{1+\lambda^2}{2}\E\left[\wt{x}_{k}\right]
+ \dfrac{2\lambda^2\wh{y}}{n(1-\lambda^2)}\frac{\delta^2}{(k+\vp)^{2\epsilon}},
\end{align}
where~$\wh{y}$ is given in~\eqref{y_hat}. We multiply~\eqref{as_cons_rate_00} by~$(k+\vp+1)^\tau$ and then apply~\eqref{b1} and~\eqref{b2} to obtain:~$\forall k\geq0$, 
\begin{align}\label{as_cons_rate_0}
\E\left[R_{k+1}\right] 
\leq&~\underbrace{\dfrac{1+\lambda^2}{2}\bigg(1+\frac{\tau}{k+\vp}+\frac{\tau^2}{(k+\vp)^2}\bigg)}_{T_k}\E\left[R_k\right]\n\\
&+ \dfrac{2e\lambda^2\wh{y}}{n(1-\lambda^2)}\frac{\delta^2}{(k+\vp)^{2\epsilon-\tau}}.
\end{align}
Since~$\vp\geq \frac{4}{1-\lambda^2}$, i.e.,~$\frac{\tau}{k+\vp}\leq\frac{1-\lambda^2}{4},\forall k\geq0$, we have
\begin{align}\label{as_cons_rate_1}
T_k=&~\bigg(1-\frac{1-\lambda^2}{2}\bigg)\bigg(1+\frac{\tau}{k+\vp}+\frac{\tau^2}{(k+\vp)^2}\bigg) \n\\
\leq&~1+\frac{\tau}{k+\vp}+\frac{\tau^2}{(k+\vp)^2} - \frac{1-\lambda^2}{2} \n\\
\leq&~1+\frac{\tau^2}{(k+\vp)^2} - \frac{1-\lambda^2}{4}.
\end{align}
Using~\eqref{as_cons_rate_1} in~\eqref{as_cons_rate_0}, we have:~$\forall k\geq0$, 
\begin{align}\label{as_cons_rate_2}
\E\left[R_{k+1}\right] 
\leq&~\left(1+\frac{\tau^2}{(k+\vp)^2}\right)\E\left[R_k\right] - \frac{1-\lambda^2}{4}\E\left[R_k\right] \n\\
&+ \dfrac{2e\lambda^2\wh{y}}{n(1-\lambda^2)}\frac{\delta^2}{(k+\vp)^{2\epsilon-\tau}}.
\end{align}
Note that~$\sum_{k=0}^\infty(k+\vp)^{-2}<\infty$ and~$\sum_{k=0}^\infty(k+\vp)^{\tau-2\epsilon}<\infty$ since~$2\epsilon-\tau>1$. Applying a special case of Lemma~\ref{RS} for deterministic recursions in~\eqref{as_cons_rate_2} leads to $\sum_{k=0}^\infty\E\left[R_k\right] < \infty.$ Since~$R_k$ is nonnegative, by monotone convergence theorem, we have~$\E\left[\sum_{k=0}^\infty R_k\right] = \sum_{k=0}^\infty\E\left[R_k\right] <\infty$ which implies 
\begin{align}\label{Rk_series}
\P\left(\sum_{k=0}^\infty R_k <\infty\right) = 1.
\end{align}
The first statement in Theorem~\ref{PL_as} then follows by~\eqref{Rk_series}.

\noindent\textbf{Recursion of~$Q_k$}. We recall from Lemma~\ref{ds_PL}:~$\forall k\geq0$,
\begin{align}\label{as_rate_0}
\E\left[\Delta_{k+1}|\F_k\right] 
\leq&~\bigg(1-\frac{\mu\delta}{(k+\vp)^\epsilon}\bigg)\Delta_{k}
+ \frac{L\delta}{2(k+\vp)^\epsilon}\wt{x}_k \n\\ &+\frac{\nu^2_a}{2n}\frac{\delta^2}{(k+\vp)^{2\epsilon}}.
\end{align}
We multiply~\eqref{as_rate_0} by~$(k+\vp+1)^\tau$ and then use~\eqref{b1} and~\eqref{b2} to obtain:~$\forall k\geq0$, 
\begin{align}\label{as_rate_1}
\E[Q_{k+1}|&\F_k] 
\leq\underbrace{\Big(1-\frac{\mu\delta}{(k+\vp)^\epsilon}\Big)\Big(1 + \frac{\tau}{k+\vp}+\frac{\tau^2}{(k+\vp)^2}\Big)}_{P_k}Q_{k} \n\\
&\qquad\quad+ \frac{eL\delta}{2(k+\vp)^\epsilon}R_k  
+\frac{e\nu^2_a}{2n}\frac{\delta^2}{(k+\vp)^{2\epsilon-\tau}}.
\end{align}
We observe that
\begin{align}\label{t1}
P_k\leq&~1 + \frac{\tau}{k+\vp}+\frac{\tau^2}{(k+\vp)^2} -\frac{\mu\delta}{(k+\vp)^\epsilon} \nonumber\\
\leq&~1 +\frac{\tau^2}{(k+\vp)^2} - \frac{\mu\delta-\tau}{(k+\vp)^\epsilon}. 
\end{align}
We use~\eqref{t1} in~\eqref{as_rate_1} to obtain:~$\forall k\geq0$,
\begin{align}\label{as_rate_2}
\E\left[Q_{k+1}|\F_k\right] 
\leq&~\bigg(1+\frac{\tau^2}{(k+\vp)^2}\bigg)Q_{k}
- \frac{\mu\delta-\tau}{(k+\vp)^\epsilon}Q_k \n\\
&+ \frac{eL\delta}{2(k+\vp)^\epsilon}R_k 
+\frac{e\nu^2_a}{2n}\frac{\delta^2}{(k+\vp)^{2\epsilon-\tau}}.
\end{align}
Recall that~$\sum_{k=0}^\infty(k+\vp)^{-2}<\infty$ and~$\sum_{k=0}^\infty(k+\vp)^{\tau-2\epsilon}<\infty$ since~$2\epsilon-\tau>1$. Note that~$\delta \geq 1/\mu$, i.e.,~$\mu\delta>\tau$,
applying Lemma~\ref{RS} in~\eqref{as_rate_2} with the help of~\eqref{Rk_series} gives: 
\begin{align}\label{Z}
\P\Big(\lim_{k\ra\infty}Q_k = Q\Big) = 1,    
\end{align}
where~$Q$ is some almost surely finite random variable, and 
\begin{align}\label{key_inf}
\P\bigg(\sum_{k=0}^{\infty}\frac{\mu\delta-\tau}{(k+\vp)^\epsilon}Q_k <\infty\bigg) = 1. 
\end{align}
Since~$\sum_{k=0}^{\infty}\frac{\mu\delta-\tau}{(k+\vp)^\epsilon} = \infty$, where~$\epsilon\in(0.5,1]$, we have 
\begin{align}\label{key_inclu}
\left\{\sum_{k=0}^{\infty}\frac{\mu\delta-\tau}{(k+\vp)^\epsilon}Q_k<\infty\right\}
\subseteq
\left\{\liminf_{k\ra\infty}Q_k = 0\right\},
\end{align}
where~``$\subseteq$" denotes the inclusion relation for two events.
By the monotonicity of~$\P(\cdot)$,~\eqref{key_inf} and~\eqref{key_inclu} lead to
\begin{align}\label{liminf}
\P\Big(\liminf_{k\ra\infty}Q_k = 0\Big) = 1.     
\end{align}
From~\eqref{liminf} and~\eqref{Z}, we conclude that~$\P\left(Q = 0\right) = 1$ 
% i.e.,
% \begin{align}\label{Qk_conv}
% \P\Big(\lim_{k\ra\infty}Q_k = 0\Big) = 1,    
% \end{align}
and then the proof follows by~\eqref{Rk_series} and Lemma~\ref{F_ave}.
\end{TH2}

%\vspace{-0.3cm}
\section{Convergence analysis under PL condition: asymptotically optimal rate in mean}\label{S_PL_ms_decay}
In this section, we prove Theorem~\ref{F_ave_rate} and Corollary~\ref{TRT}, i.e., the asymptotically optimal convergence rate of \textbf{\texttt{GT-DSGD}} in expectation and the corresponding transient time to achieve network-independent performance, when the global function~$F$ satisfies the PL condition.
Recall that in this context we focus on the following step-size sequence~\cite{OPT_ML}:
\begin{align}\label{ss}
\a_k = \frac{\beta}{k + \gamma}, \qquad\forall k\geq0,
\end{align}
where~$\beta>0$ and~$\gamma>0$ are parameters to be restricted later. We require~$\gamma \geq \beta/\ol{\a}$ so that~$0<\a_k\leq\ol{\a}$ for~$\ol{\a}$ in~\eqref{a_bar}.
We first prove a non-asymptotic rate on the consensus errors.
\begin{lem}\label{cons_rate}
Let Assumption~\ref{f}-\ref{PL} hold. If~$\gamma \geq \max\big\{\frac{\beta}{\ol{\a}},\frac{8}{1-\lambda^2}\big\}$ for~$\ol{\a}$ given in~\eqref{a_bar}, then we have:~$\forall k\geq0$,
\begin{align}\label{induction_cons}
\E\left[\left\|\mb{x}_{k}-\mb{J}\mb{x}_{k}\right\|^2\right] 
\leq \frac{\wh{x}\beta^2}{(k+\gamma)^2}.
% := \dfrac{8\lambda^2\wh{y}}{(1-\lambda^2)^2}\frac{\beta^2}{(k+\gamma)^2}, 
\end{align}
where~$\wh{x} := 8\lambda^2\wh{y}(1-\lambda^2)^{-2}$ for~$\wh{y}$ given in~\eqref{y_hat}.
\end{lem}
\begin{proof}
We prove by induction that there exists a constant~$\wh{x}$ such that~\eqref{induction_cons} holds. First, since~$\mb{x}_0^i = \mb{x}_0^r,\forall i,r\in\mc{V}$,~\eqref{induction_cons} holds trivially when~$k = 0$.
We next show that if~\eqref{induction_cons} holds for some~$k\geq0$ and then it also holds for~$k+1$. From Lemma~\ref{cons} and~\ref{Y_bounded}, we have:~$\forall k\geq0$,
\begin{align*}
\E[\left\|\mb{x}_{k+1}-\mb{J}\mb{x}_{k+1}\right\|^2] 
\leq\dfrac{1+\lambda^2}{2}\E[\left\|\mb{x}_k-\mb{J}\mb{x}_k\right\|^2] + \dfrac{2\lambda^2\wh{y}\a_k^2}{1-\lambda^2}. 
\end{align*}
Therefore, it suffices to choose~$\wh{x}$ such that~$\forall k\geq0$,
\begin{align*}
\dfrac{1+\lambda^2}{2}\frac{\wh{x}\beta^2}{(k+\gamma)^2}
+ \dfrac{2\lambda^2\wh{y}}{1-\lambda^2}\frac{\beta^2}{(k+\gamma)^2}
\leq \frac{\wh{x}\beta^2}{(k+\gamma+1)^2}, 
\end{align*}
which is equivalent to 
\begin{align}\label{con_rate_induction}
\dfrac{2\lambda^2\wh{y}}{1-\lambda^2}
\leq \left(\frac{(k+\gamma)^2}{(k+\gamma+1)^2} - \dfrac{1+\lambda^2}{2}\right)\wh{x}.
\end{align}
Since the RHS of~\eqref{con_rate_induction} monotonically increases with~$k$, we suffice to choose~$\wh{x}$ such that~\eqref{con_rate_induction} holds when~$k=0$, i.e.,$$\frac{2\lambda^2\wh{y}}{1-\lambda^2}
\leq \bigg(\frac{\gamma^2}{(\gamma+1)^2} - \frac{1+\lambda^2}{2}\bigg)\wh{x}
= \bigg(\frac{1-\lambda^2}{2}-\frac{2\gamma+1}{(\gamma+1)^2}\bigg)\wh{x}.$$
Since~$\frac{2\gamma + 1}{(\gamma+1)^2} \leq\frac{2}{\gamma}$,
it suffices to choose~$\wh{x}$ such that~$\frac{2\lambda^2\wh{y}}{1-\lambda^2}
\leq \big(\frac{1-\lambda^2}{2} - \frac{2}{\gamma} \big)\wh{x}.$
Finally, if~$\gamma \geq \frac{8}{1-\lambda^2}$, it can be observed that the induction is complete by setting~$\wh{x} := 8\lambda^2\wh{y}(1-\lambda^2)^{-2}$.
\end{proof}
\noindent
We next present a useful lemma adapted from~\cite{SAbook_AMS,GLE_kar,DSGD_Pu}. 
\begin{lem}\label{sa_ss}
Consider the step-size sequence~$\{\a_k\}$ in~\eqref{ss}. We have: for any nonnegative integers~$a,b$ such that $0\leq a\leq b,$ 
$$\prod_{s = a}^{b}\left(1-\mu\a_s\right) \leq \frac{(a+\gamma)^{\mu\beta}}{(b+\gamma+1)^{\mu\beta}}.$$
\end{lem}
\begin{proof}
By~\eqref{ss} and~$1+x\leq e^x,\forall x\in\R$, we have:~$0\leq a\leq b$,
\begin{align}\label{sa_0}
\prod_{s = a}^{b}\left(1-\mu\a_s\right)
=\prod_{s = a}^{b}\left(1-\frac{\mu\beta}{s+\gamma}\right) 
% \leq&\prod_{s = a}^{b}\exp\left\{-\frac{\mu\beta}{s+\gamma}\right\} \n\\
\leq\exp\bigg\{-\sum_{s=a}^b\frac{\mu\beta}{s+\gamma}\bigg\}.
\end{align}
Since~$\frac{1}{s+\gamma} \geq \int_{s+\gamma}^{s+\gamma+1}\frac{1}{x}dx,\forall s\geq0$, we have:~$0\leq a\leq b$,
\begin{align}\label{sa_1}
\sum_{s=a}^{b}\frac{1}{s+\gamma}
\geq
\sum_{s=a}^{b}\int_{s+\gamma}^{s+\gamma+1}\frac{1}{x}dx 
% =
% \int_{a + \gamma}^{b+\gamma+1}\frac{1}{x}dx
= \log\left(\frac{b+\gamma+1}{a+\gamma}\right).
\end{align}
Applying~\eqref{sa_1} to~\eqref{sa_0} completes the proof.
% {\small\begin{align*}
% \prod_{s = a}^{b}\left(1-\mu\a_s\right)
% \leq\exp\left\{-\mu\beta\log\left(\frac{b+\gamma+1}{a+\gamma}\right)\right\},
% \end{align*}}
\end{proof}
Now we are ready to prove Theorem~\ref{F_ave_rate} through a non-asymptotic analysis inspired by~\cite{GLE_kar,MP_Pu,DSGD_Pu,SGP_olshevsky,DSGD_SPM_Pu}.
\begin{TH3}
We denote~$\Psi_k := \E[L^{-1}(F(\ol{\x}_{k})- F^*)]$. Using Lemma~\ref{cons_rate} in Lemma~\ref{ds_PL} gives: if~$\gamma \geq \max\big\{\frac{\beta}{\ol{\a}},\frac{8}{1-\lambda^2}\big\}$,
\begin{align}\label{Delta}
\Psi_{k+1} 
\leq (1 -\mu\a_k)\Psi_{k} 
+ \wh{u}\a_{k}^2
+ \wh{z}\a_{k}^3
,\quad\forall k\geq0,
\end{align}
where~$\wh{u}$ and~$\wh{z}$ are defined as, for~$\wh{x}$ given in~\eqref{induction_cons},
\begin{align}\label{u_hat_z_hat}
\wh{u} := \frac{\nu_a^2}{2n} \qquad \text{and} \qquad \wh{z} := \frac{L\wh{x}}{2n}.    
\end{align}
%where~$\wh{u} := \nu_a^2/(2n)$ and~$\wh{z} := L\wh{x}/(2n)$ for~$\wh{x}$ given in~\eqref{induction_cons}.
We recursively apply~\eqref{Delta} from~$k$ to~$0$ to obtain\footnote{For a sequence~$\{s_k\}$, we adopt the convention~$\prod_{k = x}^{y}s_k = 1$ if~$y<x$.}:~$\forall k\geq1$,
\begin{align}\label{Delta1}
&\Psi_k \n\\
\leq&~\Psi_0\prod_{t=0}^{k-1}(1-\mu\a_t) + \sum_{t=0}^{k-1}\bigg(\left( \wh{u}\a_{t}^2 + \wh{z}\a_{t}^3\right)\prod_{l = t+1}^{k-1}(1-\mu\a_l)\bigg) \n\\
\leq&~\Psi_0\frac{\gamma^{\mu\beta}}{(k+\gamma)^{\mu\beta}}
+ \sum_{t=0}^{k-1}\Big(\frac{\wh{u}\beta^2}{(t+\gamma)^2}+\frac{\wh{z}\beta^3}{(t+\gamma)^3}\Big)\frac{(t+1+\gamma)^{\mu\beta}}{(k+\gamma)^{\mu\beta}} \n\\
=&~\Psi_0\frac{\gamma^{\mu\beta}}{(k+\gamma)^{\mu\beta}}
+\frac{\wh{u}\beta^2}{(k+\gamma)^{\mu\beta}}\sum_{t=0}^{k-1}
\frac{(t+1+\gamma)^{\mu\beta}}{(t+\gamma)^2} \n\\
&+ \frac{\wh{z}\beta^3}{(k+\gamma)^{\mu\beta}}\sum_{t=0}^{k-1}\frac{(t+1+\gamma)^{\mu\beta}}{(t+\gamma)^3},  
\end{align}
where the second inequality is due to Lemma~\ref{sa_ss}.
Furthermore, by~$1+x\leq e^x,\forall x\in\R$, we have: for~$0\leq t\leq k-1$,
\begin{align}\label{magic_e}
\frac{(t+1+\gamma)^{\mu\beta}}{(t+\gamma)^{\mu\beta}}
= \left(1 + \frac{1}{t+\gamma}\right)^{\mu\beta}
\leq\exp\left\{\frac{\mu\beta}{\gamma}\right\} 
% \leq&\exp\left\{\mu\ol{\a}\right\} \n\\
\leq&\sqrt{e},
\end{align}
where the last inequality uses~$\mu\beta/\gamma\leq\mu\ol{\a}\leq0.5$.
We use~\eqref{magic_e} in~\eqref{Delta1} to obtain:~$\forall k\geq1$,
\begin{align}\label{Delta11}
\Psi_k
\leq&~\Psi_0\frac{\gamma^{\mu\beta}}{(k+\gamma)^{\mu\beta}}
+\frac{\sqrt{e}\wh{u}\beta^2}{(k+\gamma)^{\mu\beta}}\sum_{s=\gamma}^{k-1+\gamma}s^{\mu\beta-2} \n\\
&+ \frac{\sqrt{e}\wh{z}\beta^3}{(k+\gamma)^{\mu\beta}}\sum_{s=\gamma}^{k-1+\gamma}s^{\mu\beta-3}.
\end{align}
By~$s^{\mu\beta-2} \leq \max\big\{\int_{s}^{s+1}x^{\mu\beta-2}dx,\int_{s-1}^{s}x^{\mu\beta-2}dx\big\}$, we have: if~$\beta > 1/\mu$, then~$\forall k\geq1$,
\begin{align}\label{opt_decay}
\sum_{s=\gamma}^{k-1+\gamma}s^{\mu\beta-2} 
\leq\int_{\gamma-1}^{k+\gamma}x^{\mu\beta-2}dx
% =&\frac{(k+\gamma)^{\mu\beta-1}-(\gamma-1)^{\mu\beta-1}}{\mu\beta-1} \n\\
\leq&\frac{(k+\gamma)^{\mu\beta-1}}{\mu\beta-1}.
\end{align}
Likewise, by~$s^{\mu\beta-3} \leq \max\big\{\int_{s}^{s+1}x^{\mu\beta-3}dx,\int_{s-1}^{s}x^{\mu\beta-3}dx\big\}$, we have: if~$\beta > 2/\mu$, then~$\forall k\geq1$,
\begin{align}\label{cons_decay}
\sum_{s=\gamma}^{k-1+\gamma}s^{\mu\beta-3} 
\leq \int_{\gamma-1}^{k+\gamma}x^{\mu\beta-3}dx
% =&\frac{(k+\gamma)^{\mu\beta-2}-(\gamma-1)^{\mu\beta-2}}{\mu\beta-2}\n\\
\leq&\frac{(k+\gamma)^{\mu\beta-2}}{\mu\beta-2}.
\end{align}
Now, we apply~\eqref{opt_decay} and~\eqref{cons_decay} in~\eqref{Delta11} to obtain:~$\forall k\geq1$,
\begin{align}\label{Delta2}
\Psi_k 
\leq&~\frac{\Psi_0\gamma^{\mu\beta}}{(k+\gamma)^{\mu\beta}}
+\frac{\sqrt{e}\wh{u}\beta^2}{(\mu\beta-1)(k+\gamma)} 
+ \frac{\sqrt{e}\wh{z}\beta^3}{(\mu\beta-2)(k+\gamma)^{2}}.
\end{align}
Using~\eqref{Delta2} and~Lemma~\ref{cons_rate} in Lemma~\ref{F_ave}, we obtain:~$\forall k\geq1$,
\begin{align*}
\frac{1}{n}\sum_{i=1}^n\E[F(\x_k^i) - F^*] 
\leq&\frac{2(F(\ol{\x}_0) - F^*)}{(k/\gamma+1)^{\mu\beta}}
+\frac{2\sqrt{e}L\wh{u}\beta^2}{(\mu\beta-1)(k+\gamma)} \n\\
&+ \frac{2\sqrt{e}L\wh{z}\beta^3}{(\mu\beta-2)(k+\gamma)^{2}} 
+ \frac{2\wh{z}\beta^2}{(k+\gamma)^2}.    
\end{align*}
The proof follows by that~$\frac{\wh{z}\beta^2}{(k+\gamma)^2}\leq\frac{L\wh{z}\beta^3}{(\mu\beta-2)(k+\gamma)^{2}}$ and by recalling the definitions of~$\wh{u}$ and~$\wh{z}$ given in~\eqref{u_hat_z_hat}.
\end{TH3}
%Finally, we prove Corollary~\ref{TRT}.
\begin{C2}
We derive the conditions under which the rate expression in Theorem~\ref{F_ave_rate} is network-independent. We first solve for the lower bound on~$k$ such that
\begin{align*}
\frac{L\nu_a^2\beta^2}{n(\mu\beta-1)(k+\gamma)}
\geq \frac{L^2\wh{x}\beta^3}{n(\mu\beta-2)(k+\gamma)^{2}},
\end{align*}
which may be written equivalently as
\begin{align}\label{k0}
k+\gamma \geq \frac{\mu\beta-1}{\mu\beta-2}\frac{L\wh{x}\beta}{\nu_a^2}. 
\end{align}
We suppose that~$\ln\nf_0\rn^2 = \mc{O}(n)$,~$\beta = \theta/\mu$, where~$\theta >2$. Since~$\ol{\a}L=\mc{O}\big(\frac{1-\lambda}{\lambda\kappa^{1/4}}\big)$, for~$\ol{\a}$ defined in~\eqref{a_bar},
we have 
\begin{align*}
\wh{x} = \mc{O}\left(\frac{\lambda^2n\nu_a^2}{(1-\lambda)^3} + \frac{\lambda \kappa^{1/4}\nu_a^2}{1-\lambda} + \frac{\lambda^2nL(F(\ol{\x}_0)-F^*)}{(1-\lambda)^2\kappa^{1/2}}\right),
\end{align*}
where~$\wh{x}$ is defined in~\eqref{induction_cons}.
Therefore, to make~\eqref{k0} hold, it suffices to let
\begin{align}\label{k1}
k \gtrsim \frac{\lambda^2n\kappa}{(1-\lambda)^3}+\frac{\lambda\kappa^{5/4}}{1-\lambda}+ \frac{\lambda^2n\kappa^{1/2}L(F(\ol{\x}_0)-F^*)}{(1-\lambda)^2\nu_a^2}.    
\end{align}
Next, we solve for the range of~$k$ such that for some~${\delta \in [1,\theta)}$, $(\frac{k}{\gamma}+1)^{\theta} \geq (\frac{k+1}{\kappa})^\delta$, i.e.,~$\frac{(k+\gamma)^\theta}{(k+1)^\delta} \geq \frac{\gamma^\theta}{\kappa^\delta}$.
Since~$\gamma >1$, it suffices choose~$k$ such that 
\begin{align}\label{k2}
k \geq \gamma^{\frac{\theta}{\theta-\delta}}\kappa^{-\frac{\delta}{\theta-\delta}}.    
\end{align}
We fix $\gamma = \max\{\frac{\theta}{\mu\ol{\a}},\frac{8}{1-\lambda^2}\} \asymp \max\{\kappa, \frac{\lambda^2\kappa}{(1-\lambda)^2},\frac{\lambda\kappa^{5/4}}{1-\lambda},\frac{1}{1-\lambda}\}
$.
Using~\eqref{k1} and~\eqref{k2} in Theorem~\ref{F_ave_rate}, we have 
\begin{align*}
\frac{1}{n}\sum_{i=1}^n\left(F(\x_k^i) - F^*\right) 
=&~\mc{O}\left(\frac{\kappa^{\delta}\left(F(\ol{\x}_0) - F^*\right)}{k^{\delta}}
+\frac{\kappa\nu_a^2}{n\mu k}\right),
\end{align*}
if~$k \gtrsim \max\left\{K_1,K_2\right\}$,
where $K_1$ and~$K_2$ are given by
\begin{align*}
K_1 =&~\frac{\lambda^2n\kappa}{(1-\lambda)^3}+\frac{\lambda\kappa^{5/4}}{1-\lambda}+ \frac{\lambda^2n\kappa^{1/2}L(F(\ol{\x}_0)-F^*)}{(1-\lambda)^2\nu_a^2},  \\
K_2 =&~\max\left\{\kappa, \frac{\lambda^2\kappa}{(1-\lambda)^2},\frac{\lambda\kappa^{5/4}}{1-\lambda},\frac{1}{1-\lambda}\right\}^{\frac{\theta}{\theta-\delta}}\kappa^{-\frac{\delta}{\theta-\delta}}.
\end{align*}
The proof follows by setting~$\delta = 2$ and~$\theta = 6$ in the above.
\end{C2}

\vspace{-0.2cm}
\section{Conclusion}\label{S_conc}
In this paper, we comprehensively improve the existing convergence results of stochastic first-order methods based on gradient tracking for online stochastic nonconvex problems. In particular, for both constant and decaying step-sizes, we systematically develop the conditions under which the performance of \textbf{\texttt{GT-DSGD}} matches that of the centralized minibatch \textbf{\texttt{SGD}} for both general non-convex functions and non-convex functions that further satisfy the PL condition. Our results significantly improve upon the existing theory, which suggests that \textbf{\texttt{GT-DSGD}} is strictly worse than centralized minibatch \textbf{\texttt{SGD}}. For a family of stochastic approximation step-sizes, we establish the global sublinear convergence to an optimal solution on almost every sample path of~\DSGT~when the global objective function satisfies the PL condition.

% \section*{Acknowledgments}
% We would like to thank Boyue Li from Carnegie Mellon University, Pittsburgh, for his help in numerical experiments.

\vspace{-0.2cm}
\bibliographystyle{IEEEbib}
\bibliography{sample}

\end{document}